\PassOptionsToPackage{dvipsnames}{xcolor}
\documentclass[11pt]{amsart}
\usepackage{amsfonts,amsmath,latexsym,amssymb,verbatim,amsbsy,amsthm,amssymb}
\usepackage{array}
\usepackage{graphicx}
\usepackage{caption}
\usepackage{float}
\usepackage{wrapfig}
\usepackage[shortlabels]{enumitem}
\usepackage{comment}
\usepackage{xfrac}
\usepackage{svg}  % Use the svg package to include SVG images
\usepackage{graphicx}  % To scale the image if necessary

\usepackage[top=1in, bottom=1in, left=1in, right=1in]{geometry}

\usepackage[dvipsnames]{xcolor}
\usepackage[colorlinks=true, pdfstartview=FitV, linkcolor=RoyalBlue,citecolor=ForestGreen, urlcolor=blue]{hyperref}

\theoremstyle{plain}
\newtheorem{THEOREM}{Theorem}[section]

\newtheorem{theorem}[THEOREM]{Theorem}
\newtheorem{corollary}[THEOREM]{Corollary}
\newtheorem{lemma}[THEOREM]{Lemma}
\newtheorem{proposition}[THEOREM]{Proposition}

\theoremstyle{definition}

\theoremstyle{remark}

\newtheorem{remark}[THEOREM]{Remark}

\DeclareMathOperator{\supp}{supp} %
\def \a {\alpha}

\def \e {\varepsilon}

\def \l {\lambda}

\def \cR {\mathcal{R}}

\def \hh {\mathsf{h}}

\usepackage{amsmath}

\newcommand{\R}{\ensuremath{\mathbb{R}}}   %%% reals
\newcommand{\pa}{\partial}

\def \sign {\mathrm{sgn}}

\def \p {\partial}

\makeatletter
\def\th@remark{%
  \thm@headfont{\bfseries}% título en negrita
  \normalfont% texto normal
}
\makeatother

\author{Ángel Castro and Daniel Lear}
\title[Linearized 2D Euler: Non-Radial Dynamics]{Time-Periodic Non-Radial Solutions Near  Monotone \\ Vortices in Linearized 2D Euler}

\address{Instituto de Ciencias Matematicas, Madrid.}
\email{angel\_castro@icmat.es}

\address{Universidad de Cantabria, Santander.}
\email{daniel.lear@unican.es}

\date{\today}



\begin{document}

\maketitle

\begin{abstract}
We study the linearized 2D Euler equations around radial vortex profiles. Previous works have shown that the strict monotonicity of the vorticity profile leads to axisymmetrization and inviscid damping of non-radial perturbations.

Focusing on a representative strictly decreasing radial vortex profile, we construct arbitrarily close (in low H\"older norms $C^\alpha$, with $0<\alpha<1$) radial profiles that are merely non-increasing and for which non-radial, time-periodic solutions to the linearized 2D Euler equations exist.
This shows that axisymmetrization and inviscid damping are not robust under small, low-regularity perturbations of the background profile that violate strict monotonicity.
\end{abstract}

\tableofcontents

%\newpage

\section{Introduction and main result}

In this manuscript, we consider the 2D Euler equation for an incompressible and inviscid fluid, expressed in vorticity form:
\begin{align}
\partial_t \omega + (\mathbf{u} \cdot \nabla)\omega &= 0, \label{eq:2deuler}\\
\nabla\cdot\mathbf{u} & = 0, \nonumber
\end{align}
in the full domain $\mathbb{R}^2$,  where the velocity field is recovered through the scalar vorticity by 
\begin{equation}\label{def:u}
\mathbf{u} = \nabla^\perp \Delta^{-1} \omega,
\end{equation}
with $\nabla^\perp = (-\partial_2, \partial_1)$ and
\begin{equation}\label{eq:invLaplacian}
\Delta^{-1}\omega(\mathbf{x}) = \frac{1}{2\pi} \int_{\mathbb{R}^2} \log(|\mathbf{x} - \mathbf{y}|)\,\omega(\mathbf{y})\,d\mathbf{y}.
\end{equation}

The main objective of this paper is to study the dynamics around radially symmetric vortex profiles that are monotone decreasing. To this end, we  consider linearized 2D Euler equation around a radial vortex $\varpi(r)$, which takes the form
\begin{align}\label{linear}
    \partial_tw +u_r[w]\pa_r\varpi+\frac{u_\theta[\varpi]}{r}\pa_\theta w=0,
\end{align}
where we use polar coordinates $(r,\theta)$ and where $u_r[f]$ and $u_\theta[f]$ are defined by
\begin{align*}
u_r[f](r,\theta)
&:=
\frac{\mathbf e_r}{2\pi}\cdot
\int_0^\infty\int_{-\pi}^{\pi}
\frac{\bigl(r\cos\theta-s\cos\beta,\,
r\sin\theta-s\sin\beta\bigr)^\perp}
{r^2+s^2-2rs\cos(\theta-\beta)}
\,f(s,\beta)\,s\,d\beta\,ds,\\
u_\theta[f](r,\theta)
&:=
\frac{\mathbf e_\theta}{2\pi}\cdot
\int_0^\infty\int_{-\pi}^{\pi}
\frac{\bigl(r\cos\theta-s\cos\beta,\,
r\sin\theta-s\sin\beta\bigr)^\perp}
{r^2+s^2-2rs\cos(\theta-\beta)}
\,f(s,\beta)\,s\,d\beta\,ds.
\end{align*}

In \cite{Bedrossian2017}, Bedrossian, Coti-Zelati and Vicol established that strict monotonicity (up to at zero and infinity) of the vorticity profile $\varpi$ leads to  inviscid damping and axisymmetrization. In this work, we show that strict monotonicity is essential. Specifically, we exhibit a strictly monotone radial profile $\varpi(r)$ for which, arbitrarily close to it in $C^{1^-}$, there exists another radial profile $\overline{\varpi}(r)$ that is monotone but not strictly monotone, and for which the linearized equation \eqref{linear} admits non-radial, time-periodic solutions with $m$-fold symmetry for any $m\ge 2$. \\

\begin{comment}
In this paper, we  work with the profile $\varpi(r) = (1+r^2)^{-1}$. This choice is made for simplicity; however, much more general radial profiles could also be considered, such as $(1+r^2)^{-k}$ with $k=2,3,\dots$, or  $e^{-r^2}$.
\end{comment}
In this paper, we work with the profile $$\varpi(r)=(1+r^2)^{-1}.$$ This explicit choice allows us to keep the presentation transparent while retaining the key qualitative features required by our analysis, namely that the background profile is smooth, radial, and strictly decreasing. The arguments presented here can be adapted to a broader class of smooth radial profiles that are strictly decreasing, including, for instance, $(1+r^2)^{-k}$ with $k=2,3,\ldots$, or $e^{-r^2}$. For clarity of exposition, we restrict ourselves to the above profile.

\subsection{Motivation and background}\label{motivation}

The stability of vortices is one of the most fundamental problems in the theory of hydrodynamic stability and has been considered by many authors, starting with Kelvin~\cite{Kelvin} and Orr~\cite{Orr}, and continuing to the present day in both mathematics and physics. 

Over the years, a combination of experimental observations, computer simulations, and formal asymptotics (see e.g.\ \cite{BassomGilbert,BassomGilbert1,BPPSV,BMPS,W1,W2} and references therein) suggests that a vortex subjected to a sufficiently small disturbance might return to radial symmetry as $t \to \infty$ in a weak sense. This weak convergence is referred to as \emph{vortex axisymmetrization}, and is thought to be relevant to understanding coherent vortices in 2D turbulence \cite{GallayWayne2005}, atmospheric dynamics \cite{Pedlosky1987,Vallis2006}, and various settings in plasma physics \cite{HurstDanielsonSurko2018,SchecterDubinFine2002}. 

From a theoretical perspective, Arnold's variational principle provides a foundational framework for studying the stability of vortices. The  treatment in \cite{Arnold} develops the general variational framework and establishes sufficient conditions for Lyapunov stability based on monotonicity properties of the vorticity, while the more recent work of Gallay and {\v S}ver{\'a}k \cite{GallaySverak2024} applies these ideas to specific  vortex configurations, illustrating Lyapunov stability, for example, in the cases of the Gaussian vortex and the algebraic vortex. 
Nevertheless, the precise dynamics near these steady states remain subtle and largely open.

In this regard, we also highlight the work of Wan and Pulvirenti \cite{WP}, who established the Lyapunov stability of the characteristic function of a circle in 
$L^1$ for circular bounded domains; see also \cite{SV} for results concerning the whole space. This result has been further generalized for monotone smooth vortices in stronger norms by Choi and Lim, see \cite{ChoiLim}. 

Early studies of vortex dynamics focused on point vortex, which are $\delta$-functions centered at points in $\mathbb{R}^2$. Such solutions provide models of vorticity sharply concentrated in small neighborhoods, and have been studied by many authors, including Kirchhoff~\cite{Kirchhoff} and C.C. Lin~\cite{Lin} and the books \cite{MarchioroPulvirenti,MajdaBertozzi} for more references. Building on this, Ionescu and Jia~\cite{IonescuJiaPointVortex} studied the asymptotic stability of these highly concentrated configurations in the full 2D Euler equation, showing how non‑radial perturbations evolve and ultimately axisymmetrize.

The linear stability of smooth, radially symmetric vortices has been studied in several recent works. In particular, Bedrossian, Coti Zelati, and Vicol \cite{Bedrossian2017} established vortex axisymmetrization, inviscid damping, and vorticity depletion for the linearized 2D Euler equation around a class of strictly monotone radial vortices. Their work provides the first general mechanism explaining how non-radial perturbations decay in time around a monotone radial vortex. In the same setting, Ionescu and Jia \cite{IonescuJiaSpectral} carried out a detailed study of the associated spectral density function and obtained precise information on the linearized solutions. In fact, they proved pointwise decay of the stream function consistent with \cite{Bedrossian2017}.\\

\begin{comment}
Results for the linear stability of smooth, radially symmetric vortices have recently begun to emerge. In particular, Bedrossian, Coti Zelati, and Vicol \cite{Bedrossian2017} proved that strict monotonicity of the vorticity profile leads to vortex axisymmetrization, inviscid damping, and vorticity depletion for the linearized 2D Euler equation. Their work provides the first general mechanism explaining how non-radial perturbations decay in time around a monotone radial vortex. Related results in the singular setting of point vortices were obtained by Ionescu and Jia \cite{IonescuJiaSpectral}, who studied the linearized Euler equation around point vortex solutions.\\
\end{comment}

To understand the significance of our work, let us examine the result of \cite{Bedrossian2017}, which is the main motivation of this paper. There, the authors study a
radial vorticity profile $\varpi(r)$ satisfying
\begin{enumerate}
\item [(V1)] $0\leq \varpi(r)\leq (1+r^2)^{-3}$,
\item [(V2)] $|(r\pa_r)^j \varpi(r)|\leq C_j(1+r^2)^{-3}$, for all $j\geq 0$,
\item [(V3)] (Strict monotonicity) $\pa_r\varpi(r)<0$, $r>0$,
\end{enumerate}
and the equation \eqref{linear} with initial data $$w(r,\theta,0)=\sum_{n\neq 0} w^{\text{in}}_n(r)e^{inx},$$ 
and with the extra assumption
\begin{align}\label{nuetralcondition} \int_{0}^\infty w^{\text{in}}_{\pm 1}(r)r^2dr=0.
\end{align}

Roughly speaking, they are able to prove inviscid damping, meaning that 
\begin{align*}
\|u^r\|_{L^2_\text{weight}}\lesssim (1+t^2)^{-1},\qquad \|u^\theta\|_{L^2_{\text{weight}}}\lesssim (1+t^2)^{-1/2}, 
\end{align*}
where $L^2_\text{weight}$ denotes an $L^2$ space with a suitable radial weight. Their result is in fact much stronger, providing sharper estimates for the decay of the velocity, vortex axisymmetrization, and vorticity depletion near the origin. 
The existence of time-periodic solutions for \eqref{linear} would clearly contradict these conclusions.

Condition \eqref{nuetralcondition} is not merely technical. Indeed,
\begin{align}\label{neutraleq}
    \frac{d}{dt}\int_0^\infty w_{\pm 1}(r,t)\, r^2 \, dr = 0,
\end{align}
which prevents the decay of $\widehat{\mathbf{u}}_{\pm 1}$. 
Identity \eqref{neutraleq} arises from the translational invariance of the 2D Euler equation: 
the translation of a radial function is again a stationary solution, and at the linear level, 
this implies the existence of these neutral modes. To avoid such solutions, we will restrict
our attention to modes with $|n|\geq 2$.\\

Another important result we would like to discuss is  \cite{C-ZZ} (see also \cite{Z}). There,  Coti-Zelati and Zillinger prove stability and inviscid damping for \eqref{linear} under some assumptions on $\varpi(r)$, which allow non monotone profiles. The most important condition (H1 in section 1.3), as stated by the authors, is the following (we translate it into our language):

(H1) There exists $0<\gamma_0<1$ such that
\begin{align*}
    n^2|\pa_y U(y)|-\frac{1}{2}\sign(\pa_y U(y))\pa^3_y U(y)\geq \gamma_0 n^2 |\pa_y U(y)|,\quad \forall y\in \mathbb{R},
\end{align*}
where
\begin{align}U(y):=\left.\frac{u^\theta[\varpi](r,\theta)}{r}\right|_{r=e^y}.\label{U}\end{align}

We will see in Remarks \ref{remCZZ1} and \ref{remCZZ2} that our profiles do not satisfy (H1); thus, the existence of time periodic solutions is compatible with \cite{C-ZZ}.\\

We emphasize that this problem shares many features with the study of the nonlinear stability of shear flows. Significant progress has been achieved in this direction, beginning with the inviscid damping breakthrough of Bedrossian and Masmoudi \cite{BedrossianMasmoudi2020} in $\mathbb{R}\times\mathbb{T}$, and subsequently extended to the finite periodic channel $\mathbb{T}\times[0,1]$ through the independent works of Ionescu and Jia \cite{IJ_CMP,IJ_acta} and Masmoudi and Zhao \cite{MZ}.

Motivated by these developments, the linearized equations around more general shear flows have been intensively investigated in recent years. See, for example, \cite{Bedrossian2017,BMlinear,DZ1,DZ2,GNRS,Jia1,Jia2,Wei-Zhang-Zhao,Wei-Zhang-Zhao_1,Wei-Zhang-Zhao_2,Zillinger1,Zillinger2} and the references therein for a representative, though not exhaustive, list of contributions.

Earlier foundational results on stability were established by Lin and Zeng~\cite{LZ}, who proved that nonlinear inviscid damping fails for perturbations of the Couette flow in $H^s$ when $s<3/2$. They also showed the nonexistence of nontrivial traveling waves close to Couette in $H^s$ for $s>3/2$. More specifically, they identified Kelvin’s cat’s eyes steady states in a neighborhood of the Couette shear flow. In fact, the dynamics at this level of regularity are even more intricate: traveling waves of order $O(1)$ speed and zero circulation also arise, as demonstrated by the authors in \cite{CL1} and in \cite{CL2} for the Taylor--Couette setting.

The question of whether nontrivial invariant structures, such as steady states or 
traveling waves, may exist near the Poiseuille flow has attracted significant attention. For 
the quadratic profile $U_P(y)=y^{2}$ in $\mathbb{T}\times[-1,1]$, Coti Zelati--Elgindi--Widmayer
\cite{ZEW} proved that there are no nontrivial traveling waves (TW) arbitrarily 
close to Poiseuille in $H^{5^+}$, establishing the first high-regularity rigidity result around 
this shear flow. Complementing this, Gui--Xie--Xu~\cite{GXX} showed in their classification 
theorem that any 
traveling wave $\omega$ satisfying 
\[
\|\partial_y \omega +2\|_{L^\infty} < 2
\]
must necessarily be a shear flow, thereby ruling out nontrivial TW within this regime.

A different perspective was recently provided by Drivas--Nualart \cite{DrivasNualart}, who proved 
that Poiseuille is isolated in 
the $C^{1}$ topology from non-shear steady states. These results collectively illustrate a rigidity framework near Poiseuille at high regularity.

In contrast to this rigidity, our previous work \cite{CL3} shows that such phenomena break down at lower regularity. 
More precisely, for  any $\varepsilon>0$, we construct nontrivial Lipschitz 
traveling wave solutions $\omega$ satisfying
\[
\| \omega + 2y \|_{H^{3/2^-}} < \varepsilon.
\]
These solutions also satisfy
\[
\| \omega + 2y \|_{C^{1^-}} < \varepsilon.
\]
Moreover, our construction saturates the Gui--Xie--Xu criterion since the solutions satisfies
\[
\|\pa_y\omega+2\|_{L^\infty}=2.
\]
 This optimality was also established in \cite{DrivasNualart}, where the authors additionally proved the existence of smooth stationary states near the Poiseuille flow in the $C^{1^-}$ topology. In fact, Drivas and Nualart go further and demonstrate the existence of smooth stationary states near the shear flow $y^{k}$ in the $C^{(k-1)^-}$ topology (at the vorticity level).\\

Remarkably, the existence of rotating solutions for the (full) 2D Euler equations near strictly monotone vortices, in the $L^\infty$ topology, can be recovered from Theorem 1.1 in the work of García, Hmidi, and Mateu \cite{GMH} (see also \cite{GHS}). Indeed, if one takes, for example, $\varpi_R(r) = (1+r^{2})^{-1}\chi_{[0,R]}(r)$ (which corresponds to $f_0(r)\mathbf{1}_{[0,R]}(r)$ in the notation of \cite{GMH}), then Theorem 1.1(ii) (abundance case) can be applied. To verify this, one simply replaces $\varpi_R$ by $-\varpi_R$ and observes that, although the authors work with $R=1$, an analogous result remains valid for any $R>0$. Consequently, the existence of rotating solutions bifurcating from $\varpi_R$ follows. In addition,
 \begin{align*}
 \| \varpi_R-(1+r^2)^{-1}\|_{L^\infty}=(1+R^2)^{-1},
 \end{align*}
and $\varpi_R$ is as close as we want in $L^\infty$ to $(1+r^2)^{-1}$ by making $R$ large enough. However, this distance, in smoother spaces, such as for example $C^\alpha$, is infinite.\\

Finally, we remark that our strategy to prove the main result of this paper comes from the construction of our solutions in \cite{CL3}.

\subsection{Statement of the main result}\label{s:maineq}

Our goal is to find time-periodic solutions for \eqref{linear} of the form $w(r,\theta,t):=W(r,\theta-\lambda t)$. The previous ansantz give us
\begin{align}\label{Weq}
    -\lambda\pa_\theta W +u_r[W]\pa_r \varpi+\frac{u_\theta[\varpi]}{r}\pa_\theta W=0, \quad (r,\theta)\in [0,\infty)\times \mathbb{T}.
\end{align}
Now, we  take  
\begin{align}\label{coseno}
W(r,\theta):=W_n(r)\cos(n\theta),
\end{align}
with $n\geq 2$, by the parity of the cosine. 

Substituting \eqref{coseno} into \eqref{Weq} yields
\begin{align}\label{Wn}
(\lambda+c(r))W_n(r)-\frac{\partial_r\varpi(r)}{n}\int_0^\infty K_n\left(\frac{r}{s}\right)W_n(s)ds=0, \quad r\in[0,\infty), 
\end{align}
with 
\begin{equation}\label{def:Kn}
K_n(r):=\frac{1}{2r}\left\{
\begin{aligned}
r^{-n} \qquad &r>1,\\
r^{+n} \qquad &r\leq 1.
\end{aligned}
\right.
\end{equation}
and 
\begin{equation}\label{def:c(rho)}
    c(r):=\int_0^{\infty}\varpi'(s)K_1\left(\frac{r}{s}\right)  ds=-\frac{u_\theta(r)}{r}.
\end{equation}

The computations required to go from \eqref{Weq} to \eqref{Wn} follow steps similar to those
in Section 3 of \cite{CFMS1}. For the sake of completeness, we provide all the details in the
appendix.

To solve \eqref{Wn}, we next introduce the ansatz 
\begin{align}\label{Wnansatz}
    W_n(r):=h_n(r)\partial_r\varpi(r)
\end{align}
in \eqref{Wn}, which yields the following equation for $h_n$:
\begin{align*}
    \left((\lambda+c(r))h_n(r)-\frac{1}{n}\int_0^\infty \pa_s\varpi(s)K_n\left(\frac{r}{s}\right)h_n(s)ds\right)\pa_r\varpi(r)=0.
\end{align*}
Therefore, it will be sufficient to solve
 \begin{align}\label{hneq}
    (\lambda+c(r))h_n(r)-\frac{1}{n}\int_0^\infty \pa_s\varpi(s)K_n\left(\frac{r}{s}\right)h_n(s)ds=0, \quad r\in \text{supp}(\pa_r\varpi),
\end{align}
with a function $h_n: \supp(\pa_r\varpi)\to  \mathbb{R}$.\\

The main result of the paper is the following theorem:

\begin{comment}
Aqui va la version anterior:
\begin{theorem}\label{mainthm} Fix $n\in\mathbb{N}$ large enough and $0<\alpha<1$. For all $\e>0$, there exist $\varpi_\e\in W^{1,\infty}$, $h^\e_n\in C^\infty(\emph{supp}(\pa_r \varpi_\e))$ and $\lambda^\e\in \mathbb{R}$ satisfying \eqref{hneq} and such that
\begin{align*}
    \|(1+r^2)^{-1}-\varpi_\e\|_{C^\alpha}\leq \e.
\end{align*}
\end{theorem}
\end{comment}

\begin{theorem}\label{mainthm}
Fix $n\in\mathbb N$ large enough and let $\alpha\in(0,1)$.
For every $\varepsilon>0$, there exist a  radial monotone decreasing profile $\varpi_\varepsilon\in W^{1,\infty}([0,\infty))$,
$h_n^\varepsilon\in C^\infty(\operatorname{supp}(\partial_r\varpi_\varepsilon))$
and $\lambda^\varepsilon\in\mathbb R$
such that
\[
\|(1+r^2)^{-1}-\varpi_\varepsilon\|_{C^\alpha}\le\varepsilon.
\]
Moreover, the linearized Euler operator
$\partial_t+\mathcal L_{\varpi_\varepsilon}$
admits a non-trivial travelling wave solution of the form
\[
w(r,\theta,t)
=h_n^\varepsilon(r)\,\partial_r\varpi_\varepsilon(r)
\cos(n(\theta-\lambda^\varepsilon t)),
\]
where $(h_n^\varepsilon,\lambda^\varepsilon)$ satisfies \eqref{hneq}.
\end{theorem}

\begin{remark}
Here, $\mathcal{L}_{\varpi_\varepsilon}$ denotes the linearized Euler operator around the profile $\varpi_\varepsilon$.
\end{remark}

\begin{remark}
The reason why we take $n$ large enough is quite technical, and we could actually take $n \geq 2$. See Remark \ref{rem:ngrande} for more details.
\end{remark}

\begin{remark}
The profile $\varpi(r)=(1+r^{2})^{-1}$ is strictly monotone, in the sense that 
$\partial_{r}\varpi(r)<0$ for all $r\in(0,\infty)$, as required in (V3). 
In contrast, the approximating profile $\varpi_{\varepsilon}$ does not retain this
strict monotonicity: by construction, one has 
\[
\partial_{r}\varpi_{\varepsilon}(r)=0 \quad \text{for some interval } 
I\subset(1/2,2).
\]
\end{remark}

\begin{remark}
The profiles $\varpi(r)=(1+r^{2})^{-1}$ and $\varpi_{\varepsilon}(r)$ do not satisfy (V1) 
(although both are positive). Nevertheless, one could equally work with other profiles, such as $e^{-r^2}$, which do satisfy (V1) and (V2), and establish an analog 
of Theorem \ref{mainthm} for them. We simply choose $(1+r^{2})^{-1}$ because it leads to 
simpler computations.
\end{remark}

\begin{remark}
The profile $\varpi(r)$ used in \cite{Bedrossian2017} is smooth (see condition (V2)). 
In our case, while $(1+r^{2})^{-1}$ is smooth, the approximating profile 
$\varpi_{\varepsilon}$ is only in $W^{1,\infty}$. We believe that it should be 
possible to construct time-periodic solutions associated with a $C^k$ profile as 
well, although at the cost of significantly more involved computations
\end{remark}

\begin{remark}\label{remCZZ1}
In our case, $\pa_y^3 U(y)$, with $U(y)$ as in \eqref{U}, involves $\pa_r^2\varpi_\e$ and
$$\pa^2_r\varpi_\e(r)=a\delta_{r_1}(r)-b\delta_{r_2}(r)+ \text{low order terms},$$
where $a, b, r_1, r_2 >0$ and $\pa_y U(\log(r_1))$, $\pa_y U(\log(r_2))$ have the same sign. This fact makes it impossible for (H1) to hold. We provide more details in Remark \ref{remCZZ2}.

\end{remark}

\section{Analyzing  equation \eqref{hneq}}

In what follows, our goal is to solve
\begin{equation}\label{eq:kernel}
(\lambda + c(r))f(r) -\frac{1}{n}\int_0^{\infty}\varpi'(s)K_n\left(\frac{r}{s}\right) f(s) ds=0,
\end{equation}
over the support of $\varpi'$, where  $c$ is given by \eqref{def:c(rho)}  and $K_n$ by \eqref{def:Kn}.

\begin{remark}
For the particular case $n=1$, we have a trivial solution given by 
\[
(n,\lambda,f)=(1,0,\text{cte}).
\]    
As discussed above, this solutions corresponds, at the nonlinear level, to a translation of the vortex.  
\end{remark}

\subsection{Boundary conditions induced by the equation}
From definition \eqref{def:Kn}, we have
\[
\lim_{r\to 0^+} r^{2-n}K_n(r)=0, \qquad \lim_{r\to +\infty}r^nK_n(r)=0.
\]
Since we will look for $\lambda\in\mathbb{R}$ such that $(\lambda + c(r))\neq 0$ over the support of $\varpi'$, from \eqref{eq:kernel}, we have 
\begin{equation}\label{bc:fromeq}
\lim_{r\to 0^+}r^{2-n}f(r)=0, \qquad \lim_{r\to +\infty}r^nf(r)=0.
\end{equation}

\subsection{The profile function $\varpi_{\e}$}
Since we are interested in the dynamics near an algebraic vortex, which corresponds to $\varpi(r)=\varpi_0(r)=(1+r^2)^{-1}$, we are going to consider the following perturbed profile $\varpi_\e$ given by the expression:

\begin{equation}\label{def:varpi}
\varpi_\e(r):=\left\{
\begin{aligned}
\varpi_0(r)  + \varpi_0(1+\e/2)-\varpi_0(1-\e/2) \qquad &0< r \leq 1-\e/2,\\
\varpi_0(1+\e/2)  \qquad &1-\e/2< r < 1+\e/2,\\
\varpi_0(r) \qquad &1+\e/2 \leq r< +\infty.
\end{aligned}
\right.
\end{equation}

See \textsc{Figure}~\ref{fig:imagenprofile} for an illustration of their profiles.

\begin{remark}
    Our goal is to find solutions to equation \eqref{eq:kernel} with $\varpi=\varpi_\e$ for all $0<\e<\e_0$. Here and in the rest of the paper $\e_0$ is a small enough number.
\end{remark}

\begin{figure}[h]
    \centering
    \includegraphics[width=0.6\textwidth]{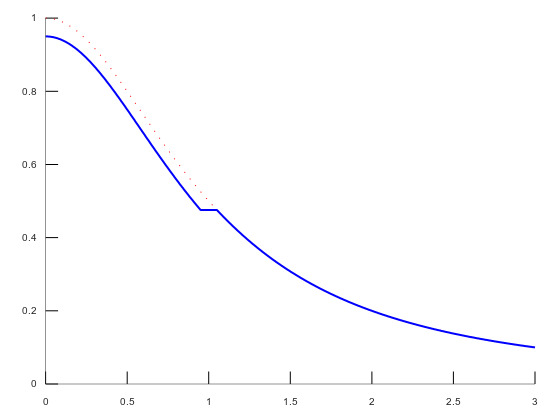}
    \caption{Comparison of $\varpi_0(r)$ and $\varpi_\e(r)$ for $\e=0.1$}
    \label{fig:imagenprofile}
\end{figure}

Importantly, the norm $\|\varpi_0-\varpi_\e\|_{C^\alpha}$ can be explicitly computed. Indeed,
\begin{align*}
    \|\varpi_0-\varpi_\e\|_{C^\alpha}\leq \|\varpi'_0\|_{L^\infty}\e^{1-\alpha}.
\end{align*}

\begin{comment}Indeed, we have that
\begin{align*}
    \varpi_0(r)-\varpi_\e(r)=\left\{\begin{array}{cc} \omega_0(1-\e/2)-\omega_0(1+\e/2) & 0\leq r\leq 1-\e/2\\ \varpi_0(r)-\varpi_0(1+\e/2) & 1-\e/2\leq r\leq 1+\e/2 \\ 0 & r\geq 1+\e/2 \end{array}\right.,
\end{align*}
from which, it is easy to check that $\|\varpi_0-\varpi_\e\|_{L^\infty}\leq C \e.$
Moreover
\begin{align*}
[\varpi_0-\varpi_\e]_{C^\alpha}\leq \sup_{r,s\in [0,\infty)}\frac{|\varpi_0(r)-\varpi_\e(r)-\varpi_0(s)+\varpi_\e(s)|}{|r-s|^\alpha}.
\end{align*}
On the one hand, if, either $r,s\in [0,1-\e/2]$ or $r,s\in [1+\e/2,\infty)$, then
\begin{align*}
    \frac{|\varpi_0(r)-\varpi_\e(r)-\varpi_0(s)+\varpi_\e(s)|}{|r-s|^\alpha}=0
\end{align*}
On the other hand, if $r\in [0,1-\e/2]$ and $s\in[1-\e/2, 1+\e/2]$
\begin{align*}
    \frac{|\varpi_0(r)-\varpi_\e(r)-\varpi_0(s)+\varpi_\e(s)|}{|r-s|^\alpha}=\frac{|\varpi_0(1-\e/2)-\varpi_0(s)|}{|r-s|^\alpha}\leq 
    \frac{|\varpi_0(1-\e/2)-\varpi_0(s)|}{|1-\e/2-s|^\alpha}\leq \|\varpi_0'\|_{L^\infty}\e^{1-\a}.
\end{align*}
And if $r\in[0,1-\e/2]$ and $s\in [1+\e/2,\infty)$
\begin{align*}
    \frac{|\varpi_0(r)-\varpi_\e(r)-\varpi_0(s)+\varpi_\e(s)|}{|r-s|^\alpha}\leq \frac{|\omega_0(1-\e/2)-\varpi_0(1+\e/2)|}{\e^\alpha}\leq ||\varpi_0'||_{L^\infty}\e^{1-\alpha}.
\end{align*}
\end{comment}

As an immediate consequence of the definition, we obtain 
\begin{equation}\label{def:varpi'}
\varpi_\e'(s)=\varpi_0'(s) \cdot \left\{
\begin{aligned}
1 \qquad &0 <  s \leq 1-\e/2,\\
0 \qquad & 1-\e/2 < s < 1+\e/2,\\
1 \qquad &1+\e/2 \leq s< +\infty,
\end{aligned}
\right.
\end{equation}
such that $\varpi_\e\in W^{1,\infty}$ but not in $C^1$, and
\[
\text{supp} (\varpi_\e')=[0,1-\e/2]\cup [1+\e/2,+\infty).
\]

Next, we try to explain the behavior of the function $c(r)$. Since 
\[
\lim_{z\to 0^+} z^2 \varpi_\e(z)=0, \qquad \lim_{z\to + \infty} \varpi _\e(z)=0,
\]
we get (after applying integration by parts) that
\begin{align*}
    \lambda+c(r)= \lambda - \frac{1}{r^2}\int_0^{r} s \varpi_\e(s)ds,
\end{align*}
where the integral term $\int_0^{r} s \varpi_\e(s)ds$ is given by
\[
\begin{array}{ll}
\displaystyle
\int_0^r s\, \varpi_0(s)\, ds + \frac{r^2}{2} \left(\varpi_0(1+\varepsilon/2) - \varpi_0(1-\varepsilon/2)\right), 
& \quad  r \leq 1 - \varepsilon/2, \\[2ex]

\displaystyle
\int_0^{1-\varepsilon/2} \hspace{-0.6 cm}s\, \varpi_0(s)\, ds + \frac{r^2}{2} \varpi_0(1+\varepsilon/2)
- \frac{(1-\varepsilon/2)^2}{2} \varpi_0(1-\varepsilon/2),
& \hspace{-1.25 cm} r \in(1 - \varepsilon/2, 1 + \varepsilon/2), \\[2ex]

\displaystyle
\int_0^{1-\varepsilon/2} \hspace{-0.6 cm} s\, \varpi_0(s)\, ds + \int_{1+\varepsilon/2}^r \hspace{-0.6 cm} s\, \varpi_0(s)\, ds
+ \frac{(1+\varepsilon/2)^2}{2} \varpi_0(1+\varepsilon/2)
- \frac{(1-\varepsilon/2)^2}{2} \varpi_0(1-\varepsilon/2),
& \quad r \geq 1 + \varepsilon/2.
\end{array}
\]

As will be seen below, we look for $\l^\ast\in\mathbb{R}$ such that $\lambda^\ast+c(r_\ast)=0$ for some $r_\ast\in(1-\e/2,1+\e/2)$. That is,
\begin{equation}\label{def:lambdarhostar}
\lambda^\ast=\frac{1}{r_\ast^2}\int_0^{r_\ast} s \varpi_\e(s)ds, \qquad \text{for $r_\ast\in (1-\e/2,1+\e/2)$.}
\end{equation}

To aid the understanding \eqref{def:lambdarhostar}, we plot in \textsc{Figure} \ref{fig:imagen1} the graph of \( c(r) \) in blue for \( r < 1 - \varepsilon/2 \) and in red for \( r > 1 + \varepsilon/2 \), using a fixed value of \(\varepsilon = 0.01\). Furthermore, by zooming in near \( r = 1 \), we observe in \textsc{Figure} \ref{fig:imagen2} a gap that will allow us to define a value \( \lambda^* \) within the green region of points, corresponding to a unique value \( r^* \in (1 - \varepsilon/2, 1 + \varepsilon/2) \).

\begin{figure}[h]
    \centering
    \begin{minipage}{0.45\textwidth}
        \centering
        \includegraphics[width=\linewidth]{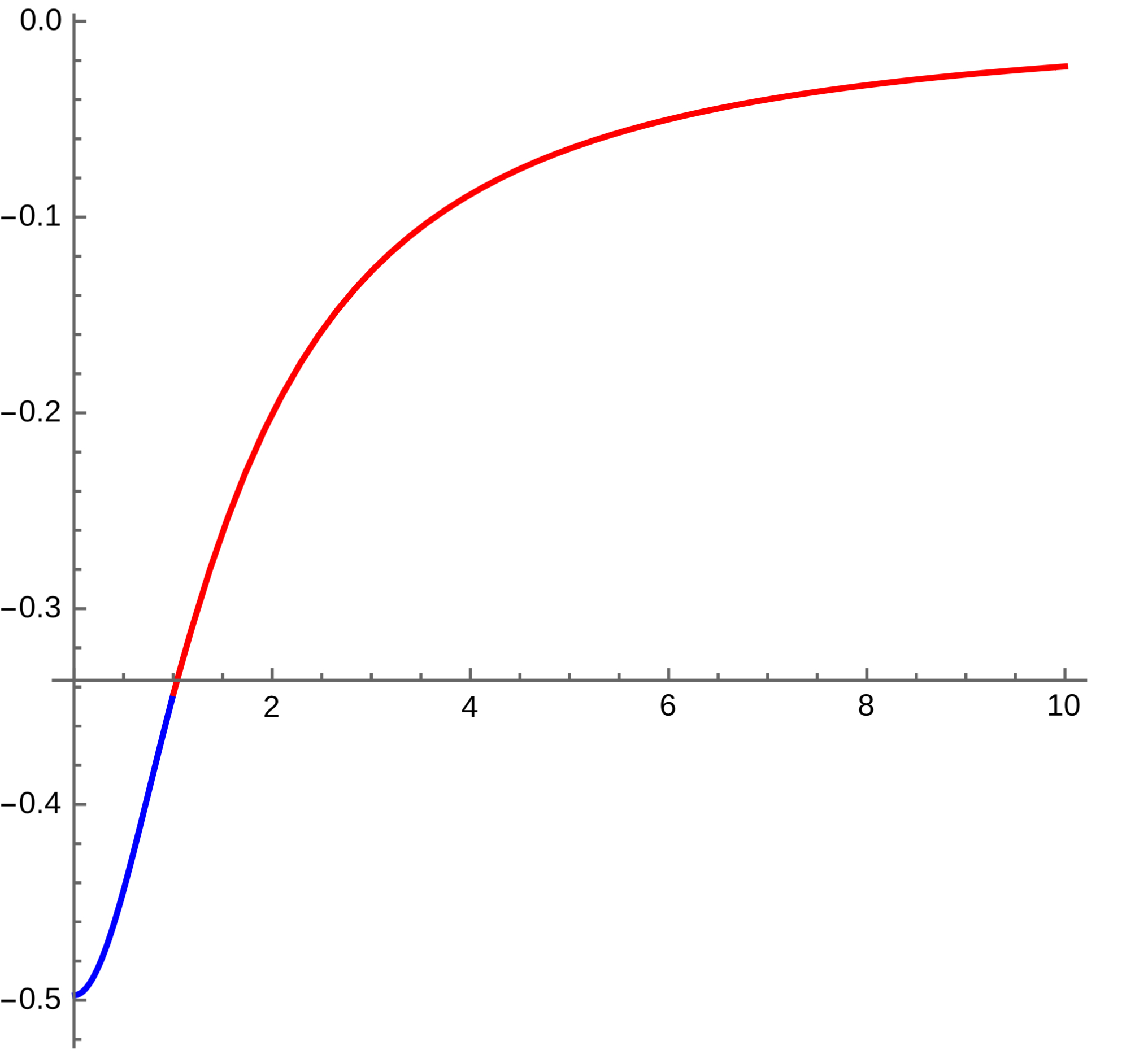}
        \caption{Plot of $c(r)$}
        \label{fig:imagen1}
    \end{minipage}
    \hfill
    \begin{minipage}{0.45\textwidth}
        \centering
        \includegraphics[width=\linewidth]{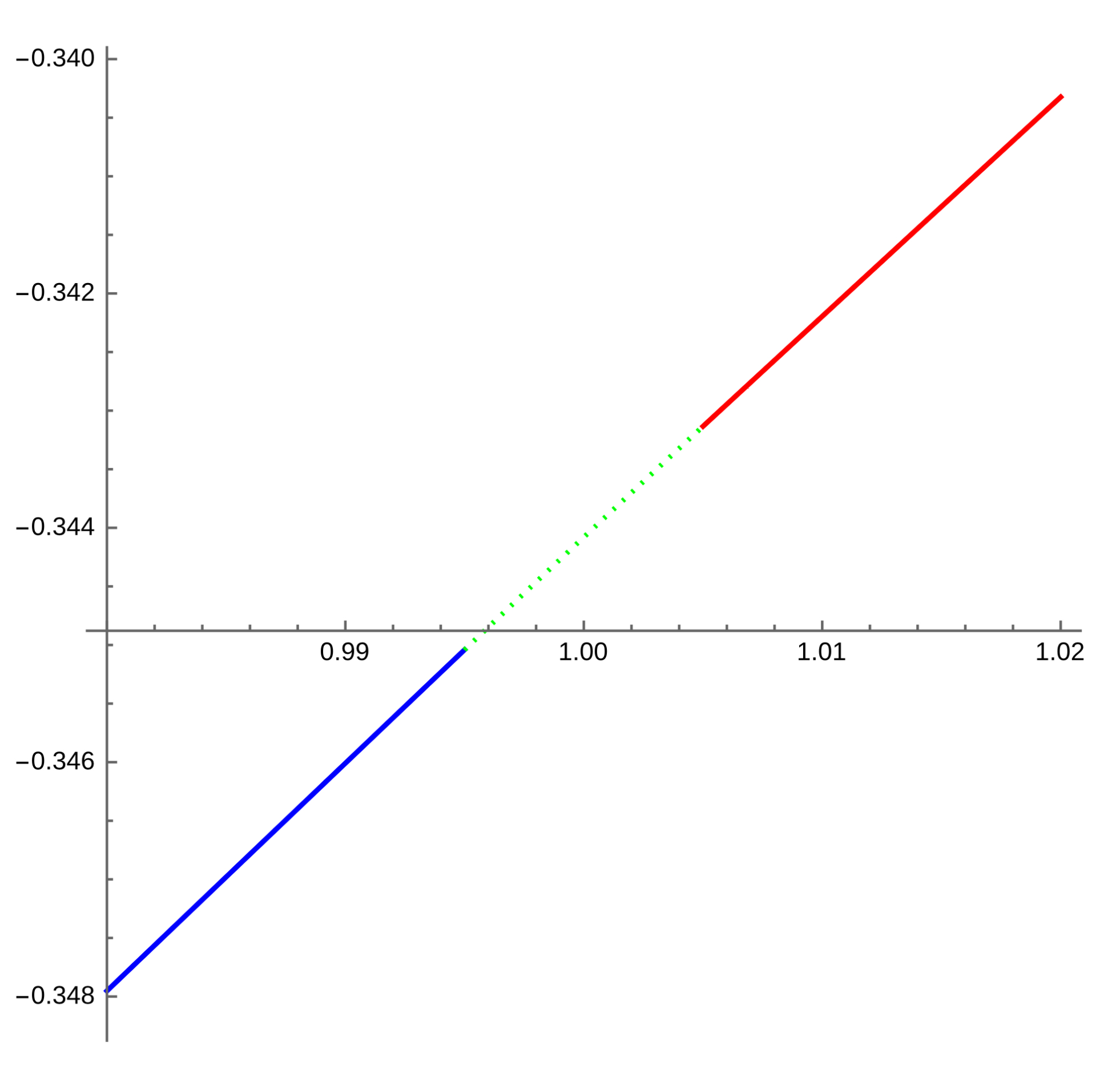}
        \caption{Zoom in $r=1$}
        \label{fig:imagen2}
    \end{minipage}
\end{figure}

From \eqref{def:c(rho)} one can check that
\begin{equation}\label{cond:cinc}
c'(r)=-\frac{1}{r^3}\int_0^r \varpi_\e'(s) s^2 ds.
\end{equation}
From the definition of \( \varpi_\varepsilon \), it follows directly that the function \( c \) is strictly increasing, i.e. 
\begin{equation}\label{c:increasing}
c'(r)>0, \quad  \forall r\in(0,+\infty).    
\end{equation}
Notice that, assuming \eqref{def:lambdarhostar}, the above also implies that 
$\lambda+c(r)\neq 0$ for all $r\in \text{supp} (\varpi_\e')$.
%Below, we include the graphs of $\varpi_0$ and, as an example for $\varepsilon > 0$, that of $\varpi_\varepsilon$ with $\varepsilon = 0.1$.

\begin{remark}\label{remCZZ2}
Let us come back to Remark \ref{remCZZ1}. Since $c(r) = -\frac{u_\theta(r)}{r}$ from \eqref{c:increasing}, we obtain that $\partial_r\left( \frac{u_\theta(r)}{r} \right) < 0$. In addition, we can also write
\begin{align*}
\frac{u_\theta(r)}{r}=\frac{1}{r^2}\int_0^r \varpi_\e(s)sds.
\end{align*}
Therefore
\begin{align*}
    \pa_r^3\left(\frac{u_\theta(r)}{r}\right)&=\frac{1}{r}\pa_r^2\varpi_\e(r)+\text{lower order terms}\\&= -\frac{\varpi'_0\left(1-\frac{\e}{2}\right)}{1-\frac{\e}{2}}\delta_{1-\frac{\e}{2}}(r)+\frac{\varpi'_0\left(1+\frac{\e}{2}\right)}{1+\frac{\e}{2}}\delta_{1+\frac{\e}{2}}(r)+\text{lower order terms}.
\end{align*}
\end{remark}

\section{Preparations for the study of the linear operator}
In the first place, using \eqref{def:Kn}, we observe that \eqref{eq:kernel} can be written  as
\begin{equation}\label{eq:generalIntEq}
2nr (\lambda + c(r))f(r)-r^{-n}\int_0^{r}\varpi_\e'(s)s^{n+1} f(s)ds -r^{n}\int_{r}^{+\infty}\varpi_\e'(s)s^{1-n}f(s)ds=0.
\end{equation}
Here, it is important to emphasize that we have to solve the above integral equation \eqref{eq:generalIntEq} just over the  support of $\varpi_\e'$. That is, 
\[
\text{supp}(\varpi_\e')=[0,1-\e/2]\cup[1+\e/2,+\infty).
\]
\subsection{The problem as a system of integral equations}
As we are working in a disjoint domain, it will be useful to use the notation
\[
f(r):=\begin{cases}
    f_L(r)\qquad r\in[0,1-\e/2],\\
    f_R(r)\qquad r\in [1+\e/2,+\infty).
\end{cases}
\]

Thus, for $r\in(0,1-\varepsilon/2)$ we have 
\begin{multline}\label{inteq:fL}
2nr (\lambda + c(r))f_L(r)-r^{-n}\int_0^{r}\varpi_\e'(s)s^{n+1} f_L(s)ds -r^{n}\int_{r}^{1-\e/2}\varpi_\e'(s)s^{1-n}f_L(s)ds\\
-r^{n}\int_{1+\e/2}^{+\infty}\varpi_\e'(s)s^{1-n}f_R(s)ds=0,
\end{multline}
and for $r\in (1+\varepsilon/2,+\infty)$ we have
\begin{multline}\label{inteq:fR}
2nr (\lambda + c(r))f_R(r)-r^{-n}\int_0^{1-\e/2}\varpi_\e'(s)s^{n+1} f_L(s)ds \\
-r^{-n}\int_{1+\e/2}^{r}\varpi_\e'(s)s^{n+1} f_R(s)ds -r^{n}\int_{r}^{\infty}\varpi_\e'(s)s^{1-n}f_R(s)ds=0.
\end{multline}

Now, since the integrand domain in the above integrals strongly depends on the parameter $\e$, we perform a change of variables to eliminate this dependence and work with a fixed domain.

\subsubsection{Rescaling}
 Let us rescale equations \eqref{inteq:fL} and \eqref{inteq:fR}. To do so, we consider the following rescaled variable $x$ introduced below, which allows us to fix the domain
\[
x:=\begin{cases}
    \frac{r}{1-\varepsilon/2}, \qquad 0 \leq r \leq 1-\varepsilon/2,\\
    \frac{r}{1+\varepsilon/2}, \qquad 1+\varepsilon/2\leq r <+\infty,
\end{cases}
\]
and the rescale functions
\begin{align*}  
f^\ast_L(x):=f_L(x(1-\varepsilon/2)), \quad &x\in [0,1],\\
f^\ast_R(x):=f_R(x(1+\varepsilon/2)), \quad &x\in[1,+\infty).
\end{align*}

Thus, for $0\leq x \leq 1$, we have to solve
\begin{multline}\label{inteq:fLrescaled}
2n   (\lambda + c_L(x,\e))f_L^\ast(x)-(1-\e/2)\int_0^{x}\varpi_0'(s(1-\e/2)) \left(\frac{s}{x}\right)^{n+1}  f_L^\ast(s)ds\\ 
-(1-\e/2)\int_{x}^{1}\varpi_0'(s(1-\e/2)) \left(\frac{s}{x}\right)^{1-n}f_L^\ast(s)ds\\
-(1+\e/2)\left(\frac{1+\e/2}{1-\e/2}\right)^{1-n}\int_{1}^{+\infty}\varpi_0'(s(1+\e/2))\left(\frac{s}{x}\right)^{1-n}  f_R^\ast(s)ds=0,
\end{multline}
and for $x \geq 1$ we have
\begin{multline}\label{inteq:fRrescaled}
    2n(\l+c_R(x,\e))f_R^\ast(x)-(1-\e/2)\left(\frac{1-\e/2}{1+\e/2}\right)^{n+1} \int_0^1 \varpi_0'(s(1-\e/2))\left(\frac{s}{x}\right)^{n+1} f_L^\ast(s)ds\\
    -(1+\e/2)\int_1^x \varpi_0'(s(1+\e/2))\left(\frac{s}{x}\right)^{n+1}f_R^\ast(s)ds -(1+\e/2)\int_x^\infty \varpi_0'(s(1+\e/2)) \left(\frac{s}{x}\right)^{1-n} f_R^\ast(s)ds=0,
\end{multline}
where we have introduced  the functions
\begin{align}\label{cLcRfunctions}
c_L(x,\e):=c(x(1-\e/2)), \quad x\in[0,1] , \qquad c_R(x,\e):=c(x(1+\e/2)), \quad x\in[1,+\infty).
\end{align}

\subsubsection{Asymptotic analysis of $c_L$ and $c_R$ in terms of parameter $\e$.}\label{ss:analysisC}

We recall that
\[
c(r)=-\frac{1}{r^2}\int_0^r s \varpi_0(s)ds -\frac{\varpi_0(1+\e/2)-\varpi_0(1-\e/2)}{2}, \qquad r\in[0, 1-\e/2].
\]
Thus, for $x\in[0,1]$ it is clear that
\[
c_L(x,\e)=-\frac{1}{(x(1-\e/2))^2}\int_0^{x(1-\e/2)} s \varpi_0(s)ds -\frac{\varpi_0(1+\e/2)-\varpi_0(1-\e/2)}{2}.
\]
Using the expression of $\varpi_0$, we have
\[
c_L(x,\e)=
\frac{16\,\varepsilon}{64 + \varepsilon^{4}}
\;-\;
\frac{2\,\ln\!\left(1 + \frac{1}{4}(-2 + \varepsilon)^{2}x^{2}\right)}
{(-2 + \varepsilon)^{2}x^{2}}.
\]
\begin{remark}
It can be checked that $c_L(x,\e)$, which is, in principle, defined in $[0,1]$, can be extended as a real-analytic function to the full interval $[0,+\infty)$. Notice that this extension is not equal to $c(x(1-\e/2))$, for $x\in [1,+\infty)$.    
\end{remark}

By performing an asymptotic analysis with respect to the parameter $\e$, we obtain:
\begin{equation}\label{preparativos:expcL}
c_L(x,\e)=c_L(x,0)+ (\pa_\e c_L)(x,0)\e + \tfrac{1}{2}(\pa_\e^2 c_L)(x,\e^\ast)\e^2, \qquad \text{for } \e^\ast=\e^\ast(x)\in(0,\e).
\end{equation}
Here, we have
\[
c_L(x,0)=-\frac{\log(1+x^2)}{2x^2},
\]
\begin{equation}\label{def:dcL}
(\pa_\e c_L)(x,0)=\frac{1}{4}+\frac{1}{2(1+x^2)}-\frac{\log(1+x^2)}{2x^2},
\end{equation}
and
\begin{align*}
(\pa_\e^2 c_L)(x,\e)=&-\frac{64\e^3 (320-3\e^3)}{(64+\e^4)^3}+\frac{16}{(2-\e)^2 (4+(2-\e)^2x^2)}-\frac{16-4(2-\e)^2x^2}{(4(2-\e)+(2-\e)^3 x^2)^2}\\
&-\frac{\log\left(1+\frac{(2-\e)^2 x^2}{2^2}\right)}{(2-\e)^4 x^2}.
\end{align*}
In addition, another asymptotic expansion that will be used later is
\begin{align}\label{exp:dxcr}
\pa_x c_L(x,\e)=
\frac{1}{x}\left(\frac{\ln(1+x^{2})}{x^2}-\frac{1}{1+x^{2}}\right)
+\e \pa_\e\p_{ x} c_L(x,\e_\ast), \quad \e_\ast=\e_\ast(x)\in (0,\e),
\end{align}
with
\begin{align*}
    \pa_\e \p_{ x}c_L(x,\e)=
\frac{
8 \left(
\dfrac{2(-2 + \varepsilon)^{2} x^{2} \left( 2 + (-2 + \varepsilon)^{2} x^{2} \right)}{\left( 4 + (-2 + \varepsilon)^{2} x^{2} \right)^{2}}
- \ln\!\left( 1 + \frac{1}{4} (-2 + \varepsilon)^{2} x^{2} \right)
\right)
}{
(-2 + \varepsilon)^{3} x^{3}
}.
\end{align*}

\begin{remark}\label{r:pa2cL}
We observe that $ \pa_\e^2 c_L\in C^\infty([0,7/4]\times [0,\varepsilon_0))$, for a fixed $\e_0$ sufficiently small, with
$$ \sup_{(x,\e)\in [0,7/4]\times [0,\varepsilon_0)} |\pa^2_\e c_L(x,\e)|\leq C(\e_0).$$ Also,
$ \pa_\e\pa_x c_L\in C^\infty([0,7/4]\times [0,\varepsilon_0))$, with
\begin{align*}\sup_{(x,\e)\in [0,7/4]\times [0,\varepsilon_0)} |\pa_\e\pa_x c_L(x,\e)|\leq C(\e_0).\end{align*}
\end{remark}

Proceeding in a similar manner, we recall that
\begin{multline*}
c(r)=-\frac{1}{r^2}\left(\int_0^{1-\e/2}s\varpi_0(s)ds + \int_{1+\e/2}^r s\varpi_0(s)ds\right)\\
-\frac{1}{r^2}\left(\frac{(1+\e/2)^2}{2} \varpi_0(1+\e/2)-\frac{(1-\e/2)^2}{2}\varpi_0(1-\e/2)\right), \qquad r\in[1+\e/2,+\infty).
\end{multline*}
Thus, for $x\in[1,+\infty)$,
\begin{align*}
c_R(x,\e)=&-\frac{1}{(x(1+\e/2))^2}\left(\int_0^{1-\e/2}s\varpi_0(s)ds + \int_{1+\e/2}^{x(1+\e/2)} s\varpi_0(s)ds\right)\\
-&\frac{1}{(x(1+\e/2))^2}\left(\frac{(1+\e/2)^2}{2} \varpi_0(1+\e/2)-\frac{(1-\e/2)^2}{2}\varpi_0(1-\e/2)\right).
\end{align*}
Using the expression of $\varpi_0$, we have
\begin{align*}
c_R(x,\e)=    &\frac{-64\,\varepsilon}{(2 + \varepsilon)^2 (64 + \varepsilon^4)\,x^2}\\
-&\frac{2\,(64 + \varepsilon^4)}{(2 + \varepsilon)^2 (64 + \varepsilon^4)\,x^2}\log\left(\frac{\left(2 + \tfrac{1}{4}(-4 + \varepsilon)\,\varepsilon\right)\left(1 + \tfrac{1}{4}(2 + \varepsilon)^2 x^2\right)}{\left(2 + \varepsilon + \tfrac{\varepsilon^2}{4}\right)}\right).
\end{align*}
\begin{remark}
It can be checked that $c_R(x,\e)$, which is defined in $[1,+\infty)$, can be extended as a real-analytic function to $(0,+\infty)$. Notice that this extension is not equal to $c(x(1+\e/2))$, for $x\in [0,1]$.
\end{remark}

An asymptotic expansion with respect to the small parameter $\e$ yields:
\[
c_R(x,\e)=c_R(x,0)+ (\pa_\e c_R)(x,0)\e + \tfrac{1}{2}(\pa_\e^2 c_R)(x,\e^\ast)\e^2, \qquad \text{for } \e^\ast=\e^\ast(x)\in(0,\e).
\]
Here, we have
\[
c_R(x,0)=-\frac{\log(1+x^2)}{2x^2},
\]
\begin{equation}\label{def:dcR}
(\pa_\e c_R)(x,0)=  -\dfrac{1}{4x^2} + \dfrac{1 }{2x^2 \left(1 + x^2\right)} + \dfrac{ \log\left(1 + x^2\right)}{2x^2},
\end{equation}
and
\begin{multline*}
(\pa_\e^2 c_R)(x,\e)=\frac{384\e}{(2 + \e)^4 (64 + \e^4)\, x^2} +\frac{256\,(64 - 3\e^4)}{(2 + \e)^3\,(64 + \e^4)^2\, x^2} +\frac{256\, \e^3\, (320 - 3\e^4)}{(2 + \e)^2\, (64 + \e^4)^3\, x^2}\\
+\frac{16}{(2 + \e)^3} \left( \frac{4(-8 + \e^2)}{(64 + \e^4)x^2} + \frac{2 + \e}{4 + (2 + \e)^2 x^2} \right)
+\frac{4}{(2 + \e)^2} \left( 
\frac{8\e \left( -64 - 16\e^2 + \e^4 \right)}{(64 + \e^4)^2 x^2}
+ 
\frac{-4 + (2 + \e)^2 x^2}{\left( 4 + (2 + \e)^2 x^2 \right)^2}
\right)\\
+\frac{12}{(2 + \e)^4 x^2}\log\left[\frac{\left(2 + \frac{1}{4} (-4 + \e)\e \right)\left(1 + \frac{1}{4}(2 + \e)^2 x^2 \right)}{\left(2 + \e + \frac{\e^2}{4} \right)}\right].
\end{multline*}

In addition, another asymptotic expansion that will be used later is
\begin{align}\label{exp:dxcr}
\pa_x c_R(x,\e)= \frac{1}{x}\left(\frac{\log(1+x^2)}{x^2}-\frac{1}{1+x^2}\right) +\e \pa_\e \p_{ x} c_R(x,\e_\ast), \quad \e_\ast=\e_\ast(x)\in (0,\e),
\end{align}
with
\begin{align*}
&\pa_\e \pa_{x} c_R(x,\e) \\
&= \frac{8}{(2 + \varepsilon)^3 x^3}
\Bigg[
  2 \Bigg(
     \frac{
       3072 
       + \varepsilon \!\left(
         -1536 
         + \varepsilon \!\left(
           256 
           + \varepsilon \!\left(
             128 
             + \varepsilon \!\left(
               48 
               + \varepsilon \!\left(
                 -56 
                 + \varepsilon (4 + \varepsilon (2 + \varepsilon))
               \right)
             \right)
           \right)
         \right)
       \right)
     }{(64 + \varepsilon^4)^2}  \\[6pt]
 & \qquad + \frac{8}{\big(4 + (2 + \varepsilon)^2 x^2\big)^2}
           - \frac{6}{4 + (2 + \varepsilon)^2 x^2}
   \Bigg)  - \log\!\left(
     \frac{
       \left(2 + \tfrac{1}{4}(-4 + \varepsilon)\varepsilon\right)
       \left(4 + (2 + \varepsilon)^2 x^2\right)
     }{
       8 + \varepsilon (4 + \varepsilon)
     }
   \right)
\Bigg].
\end{align*}

\begin{remark}\label{r:pa2cR}
We observe that $ \pa_\e^2 c_R\in C^\infty([1/4,+\infty)\times [0,\varepsilon_0))$, for a fixed $\e_0$ sufficiently small, with
$$ \sup_{(x,\e)\in [1/4,+\infty)\times [0,\varepsilon_0)} |\pa^2_\e c_R(x,\e)|\leq C(\e_0).$$ Also,
$ \pa_\e\pa_x c_R\in C^\infty([1/4,+\infty)\times [0,\varepsilon_0))$, with
\begin{align*}\sup_{(x,\e)\in [1/4,+\infty)\times [0,\varepsilon_0)} |\pa_\e\pa_x c_R(x,\e)|\leq C(\e_0).\end{align*}
\end{remark}

From this paragraph, it is important to emphasize that, to first order, both expressions coincide, although each is defined on a complementary domain: \( c_L \) is defined on \( [0,1] \), while \( c_R \) is defined on \( [1, +\infty) \). That is, we have
\begin{equation}\label{cl0equalcr0}
c_L(x,0) = -\frac{\log(1 + x^2)}{2x^2} = c_R(x,0).
\end{equation}

\subsection{The ansatz for $\lambda$}

In the remainder of the work, we impose the ansatz   
\begin{equation}\label{def:ansatzlambda}
\lambda = \lambda_0 + \e \lambda_1 + \e^2 \log^2(\e) \lambda_2(\e),
\end{equation}
with  
\begin{align}
    \label{lambda0}\lambda_0 &= \frac{\log(2)}{2}, \\
   \label{bound:l1} \lambda_1 &\in \left(-\frac{\log(2)}{2}, -\frac{1 - \log(2)}{2}\right),\\
   \label{bound:l2} |\lambda_2(\e)|&\leq  \mathsf{M},
\end{align}  
where $\mathsf{M}<+\infty$ is a free parameter independent of $\e$ which will be fixed at the end of the paper.  

This choice of $\lambda$ provides  Lemmas \ref{bounddenominator} and \ref{bounddenominatorLeft}, which unravel the behavior of the quotients $(\lambda+c_{R,L}(x,\e))^{-1}$ and  will be used in the following sections.

\begin{lemma}\label{bounddenominator} 
Analysis for $\lambda_0+c_R(x,0)$ and $\l+c_R(x,\e)$. Let $\lambda$ be given by \eqref{def:ansatzlambda}, with $\lambda_0$, $\lambda_1$ and $\lambda_2(\e)$ as in \eqref{lambda0}, \eqref{bound:l1} and \eqref{bound:l2}. Then, there exists $\e_0>0$, such that, for any $0<\e<\e_0$ we have:
\begin{enumerate}
    \item\label{bd1} For $1<x\leq 2$,
    \[
        |\lambda_0+c_R(x,0)|^{-1}, \; |\l+c_R(x,\e)|^{-1} \;\leq\; C|x-1|^{-1}.
    \]

    \item\label{bd2} For $2\leq x<\infty$,
    \[
        |\lambda_0+c_R(x,0)|, \; |\l+c_R(x,\e)| \;\geq\; c.
    \]

    \item\label{bd3} For $1<x\leq 2$,
    \[
        \l+c_R(x,\e) \;\geq\; c(x-1)+c_1\e.
    \]

    \item \label{bd4} There exists $x_R^*\in(0,1)$ such that $\lambda+c_R(x,\e)$ has a simple zero at $x_R^*$ and $\lambda+c_R(x,\e)>0$ in $(x^*_R,1)$. In addition, $x^*_R\to 1$ when $\e\to 0$.
\end{enumerate}
Here $C>0$, $c>0$ are constants that may depend on $\e_0$ but are independent of $\e$.  The constant $c_1>0$ depends on $\e_0$ and $\lambda_1+\frac{\log(2)}{2}>0$ (degenerates if $\lambda_1+\frac{\log(2)}{2}=0$)  and is independent of $\e$. The size of $\e_0$ depends on $\mathsf{M}$  and $\lambda_1+\frac{\log(2)}{2}$.  
\end{lemma}

\begin{proof} By definition $
       \lambda_0+c_R(x,0)=\frac{\log(2)}{2}-\frac{\log(1+x^2)}{2x^2}$ and then $\lambda_0+c_R(1,0)=0$.  Next, prove that $\partial_x c_R(x,0)>0$ for $x\in [1,+\infty)$. This derivative reads
       \begin{align*}
       \partial_x c_R(x,0)=\frac{1}{x}\left(\frac{\log(1+x^2)}{x^2}-\frac{1}{1+x^2}\right).
       \end{align*}
But
\begin{align}\label{creciente}
    \log(1+x^2)-\frac{x^2}{1+x^2}=\int_0^1 \frac{d}{ds}\log(1+sx^2)-\frac{x^2}{1+x^2}ds=\int_0^1 x^2\frac{(1-s)x^2}{(1+x^2)(1+sx^2)}ds>0.
\end{align}
Therefore, we have proven \eqref{bd1} and \eqref{bd2} for $\lambda_0+c_R(x,0)$.

We recall that
\[
\l+c_R(x,\e)=\lambda_0+c_R(x,0)+\e\left(\l_1 -\dfrac{1}{4x^2} + \dfrac{1 }{2x^2 \left(1 + x^2\right)} + \dfrac{ \log\left(1 + x^2\right)}{2x^2}\right) +\e^2\log^2(\e)\lambda_2(\e)+ O(\e^2).
\]
Thus, in order to prove \eqref{bd1} and \eqref{bd3} for $\lambda+c_R(x,\e)$, we first notice that
\begin{align*}
f(\lambda_1,x):=\l_1 -\dfrac{1}{4x^2} + \dfrac{1 }{2x^2 \left(1 + x^2\right)} + \dfrac{ \log\left(1 + x^2\right)}{2x^2},
\end{align*}
satisfies
\begin{align*}
f(\lambda_1,1)=\lambda_1+\frac{\log(2)}{2}>0, \quad \text{for $\lambda_1\in \left(-\frac{\log(2)}{2},\,-\frac{1-\log(2)}{2}\right)$}.
\end{align*}
Then there exists $\bar{x}$, with $1<\bar{x}<2$, such that $f(\lambda_1,x)\geq \frac{1}{2}(\lambda_1+\frac{\log(2)}{2})$, $1\leq x\leq \bar{x}$. Thus
\begin{align*}
    \lambda+c_R(x,\e)\geq \lambda_0+c_R(x,0)+\frac{1}{4}\left(\lambda_1+\frac{\log(2)}{2}\right)\e\geq c(x-1)+\frac{1}{4}\left(\lambda_1+\frac{\log(2)}{2}\right)\e, \quad \text{for $1\leq x\leq \bar{x}$},
\end{align*}
where we have absorbed the terms $o(\e)$ into $\frac{1}{4}(\lambda_1+\frac{\log(2)}{2})$. 

In addition, for $\bar{x}\leq x\leq 2$, $|f(\lambda_1,x)|\leq C$ and 
\begin{align*}
\lambda+c_R(x,\e) \geq c(x-1)-C\e\geq \frac{c}{2}(x-1)+\frac{c}{2}(\bar{x}-1)-C\e\geq \frac{c}{2}(x-1)+\frac{1}{4}\left(\lambda_1+\frac{\log(2)}{2}\right)\e,
\end{align*}
for $\e$ small enough.

We prove (\ref{bd2}) for $\lambda+c_R(x,\e)$. This just follows from the fact that
\begin{align*}
\lambda+c_R(x,\e)\geq \lambda_0+c_R(x,0)-C\e.
\end{align*}
Finally, we prove \eqref{bd4}. By \eqref{bd3}-Lemma \ref{bounddenominator}, $\l+c_R(1,\e)>0$. In addition, $\lambda_0+c_R(1/2,0)<0$. Thus, using Remark \ref{r:pa2cR} we have that $\lambda+c_R(1/2,\e)<0$, for $\e$ small enough. In addition, \eqref{creciente} shows that $\pa_xc_R(x,0)>c>0$, for $x\in[1/2,2]$.  Using \eqref{exp:dxcr} and Remark \ref{r:pa2cR}, and taking $\e$ small enough, we find that $\pa_x c_R(x,\e)>c/2>0$, for $x\in[1/2,2]$. 
\end{proof}

An analogous statement holds on the left-hand side, whose proof is omitted.
\begin{lemma}\label{bounddenominatorLeft}
Analysis for $\lambda_0+c_L(x,0)$ and $\l+c_L(x,\e)$.  Let $\lambda$ be given by \eqref{def:ansatzlambda}, with $\lambda_0$, $\lambda_1$ and $\lambda_2(\e)$ as in \eqref{lambda0}, \eqref{bound:l1} and \eqref{bound:l2}. Then, there exists $\e_0>0$, such that, for any $0<\e<\e_0$ we have:
\begin{enumerate}
    \item\label{bl1} For $0<x< 1$,
    \[
        |\lambda_0+c_L(x,0)|^{-1}, \; |\l+c_L(x,\e)|^{-1} \;\leq\; C|x-1|^{-1}.
    \]

    \item\label{bl2} For $0<x< 1$,
    \[
        \l+c_L(x,\e) \;\leq\; -c(1-x)-c_1\e.
    \]

    \item \label{bl3}  There exists $x_L^*\in(1,+\infty)$ such that $\lambda+c_L(x,\e)$ has a simple zero at $x_L^*$ and $\lambda+c_L(x,\e)<0$ in $(1,x_L^*)$. In addition, $x^*_L\to 1$ when $\e\to 0$.
\end{enumerate}
Here $C>0$, $c>0$ are constants that may depend on $\e_0$ but are independent of $\e$.  The constant $c_1>0$ depends on $\e_0$ and $\lambda_1+\frac{1-\log(2)}{2}<0$  (degenerates if $\lambda_1+\frac{1-\log(2)}{2}=0$) and is independent of $\e$. The size of $\e_0$ depends on $\mathsf{M}$  and $\lambda_1+\frac{1-\log(2)}{2}$. 
\end{lemma}

\subsection{From an integral equation to an ODE}
After all these preparations, we must not lose sight of the point that our objective is to solve the coupled system of integral equations \eqref{inteq:fLrescaled}-\eqref{inteq:fRrescaled}. For the convenience of the readers, both integral equations are recalled below:
\begin{multline*}
2n   (\lambda + c_L(x,\e))f_L^\ast(x)-(1-\e/2)\int_0^{x}\varpi_0'(s(1-\e/2)) \left(\frac{s}{x}\right)^{n+1}  f_L^\ast(s)ds \\
-(1-\e/2)\int_{x}^{1}\varpi_0'(s(1-\e/2)) \left(\frac{s}{x}\right)^{1-n}f_L^\ast(s)ds\\
-(1+\e/2)\left(\frac{1+\e/2}{1-\e/2}\right)^{1-n}\int_{1}^{+\infty}\varpi_0'(s(1+\e/2))\left(\frac{s}{x}\right)^{1-n}  f_R^\ast(s)ds=0, \qquad x\in [0,1],
\end{multline*}
and 
\begin{multline*}
    2n(\l+c_R(x,\e))f_R^\ast(x)-(1-\e/2)\left(\frac{1-\e/2}{1+\e/2}\right)^{n+1}\int_0^1 \varpi_0'(s(1-\e/2))\left(\frac{s}{x}\right)^{n+1}f_L^\ast(s)ds\\
    -(1+\e/2)\int_1^x \varpi_0'(s(1+\e/2))\left(\frac{s}{x}\right)^{n+1} f_R^\ast(s)ds\\ 
    -(1+\e/2)\int_x^\infty \varpi_0'(s(1+\e/2)) \left(\frac{s}{x}\right)^{1-n} f_R^\ast(s)ds=0, \qquad x\in [1,+\infty).
\end{multline*}

\begin{comment}
Taking into account all that has been seen above, it is clear that $\lambda+c_L(x,\e)$ strongly depends on the parameter $\e$. Here, we are going to make explicit that dependence, that is, $\lambda+c_L(x,\e)=\l(\e)+c_L(x,\e)$, where we are imposing the ansatz \eqref{def:ansatzlambda} and the asymptotic expansion \eqref{preparativos:expcL}.

In summary, for $x\in(0,1)$ we have
\begin{multline*}
\lambda+c_L(x,\e)=\l(\e)+c_L(x,\e)=\l_0+ c_L(x,0)+ \e (\l_1 + (\pa_\e c_L)(x,0))\\
\quad +\e^2 (\log(\e)\l_2(\e)+\tfrac{1}{2}(\pa_\e^2 c_L)(x,\e^\ast)), \qquad \text{for some } \e^\ast \in (0,\e).
\end{multline*}

For brevity, we omit the details, but the same type of expansion also holds for $\l+c_R(x,\e)$.
\end{comment}

\begin{remark}
It is important to emphasize that $\l_0$ is fixed from the beginning but at this time $\l_1\in(-\log(2)/2,-(1-\log(2))/2)$ is a free parameter that will be fixed later, see Section \ref{s:fixingl1}. The same happens for $\l_2(\e)$, which will be fixed at the end of the paper by a fixed point.
\end{remark}

\subsubsection{An appropiate auxiliary function}
Now, we come back to the integral equations \eqref{inteq:fLrescaled}-\eqref{inteq:fRrescaled}. Hence, we consider $H=(H_L,H_R)$ given by
\begin{equation}\label{def:hL}
H_L(x):=2nx(1-\e/2)(\l + c_L(x,\e))f_L^\ast(x), \qquad x\in [0,1],
\end{equation}
and
\begin{equation}\label{def:hR}
H_R(x):=2nx(1+\e/2)(\l + c_R(x,\e))f_R^\ast(x), \qquad x\in [1,+\infty).
\end{equation}

Then, equation \eqref{inteq:fLrescaled} can be written as
\begin{multline}\label{inteqleft_final}
H_L(x)-\frac{(1-\e/2)}{2n}\int_0^{x}\varpi_0'(s(1-\e/2)) \left(\frac{s}{x}\right)^{n}  \frac{H_L(s)}{\l+c_L(s,\e)}ds\\ 
-\frac{(1-\e/2)}{2n}\int_{x}^{1}\varpi_0'(s(1-\e/2)) \left(\frac{s}{x}\right)^{-n}\frac{H_L(s)}{\l+c_L(s,\e)}ds\\
-\frac{(1+\e/2)}{2n}\left(\frac{1+\e/2}{1-\e/2}\right)^{-n} \int_{1}^{+\infty}\varpi_0'(s(1+\e/2))\left(\frac{s}{x}\right)^{-n} \frac{H_R(s)}{\l+c_R(s,\e)}ds=0, \qquad x\in [0,1],
\end{multline}
and similarly, equation \eqref{inteq:fRrescaled} can be written as
\begin{multline}\label{inteqright_final}
    H_R(x)-\frac{ (1-\e/2)}{2n}\left(\frac{1-\e/2}{1+\e/2}\right)^{n}\int_0^1\varpi_0'(s(1-\e/2))\left(\frac{s}{x}\right)^{n} \frac{H_L(s)}{\l+c_L(s,\e)}ds\\
    -\frac{(1+\e/2)}{2n}\int_1^x \varpi_0'(s(1+\e/2))\left(\frac{s}{x}\right)^{n}\frac{H_R(s)}{\l+c_R(s,\e)}ds\\
    -\frac{(1+\e/2)}{2n}\int_x^\infty \varpi_0'(s(1+\e/2)) \left(\frac{s}{x}\right)^{-n} \frac{H_R(s)}{\l+c_R(s,\e)}ds=0, \qquad x\in [1,+\infty).
\end{multline}

\subsubsection{Computing derivatives in the integral equations} \label{s:InteqtoODE}
Taking a derivative in the above expression \eqref{inteqleft_final} and using the fact that the boundary terms cancel each other, we have
\begin{multline}\label{eq:1derivative}
H_L'(x) -\frac{(-n)}{x}  \frac{1}{2n}\int_0^{x}(1-\e/2)\varpi_0'(s(1-\e/2)) \left(\frac{s}{x}\right)^{n}  \frac{H_L(s)}{\l+c_L(s,\e)}ds\\ 
-\frac{(+n)}{x}\frac{1}{2n}\int_{x}^{1}(1-\e/2)\varpi_0'(s(1-\e/2)) \left(\frac{s}{x}\right)^{-n}\frac{H_L(s)}{\l+c_L(s,\e)}ds\\
-\frac{(+n)}{x}\frac{1}{2n}\int_{1}^{+\infty}(1+\e/2)\varpi_0'(s(1+\e/2))\left(\frac{s}{x}\right)^{-n} \left(\frac{1+\e/2}{1-\e/2}\right)^{-n}  \frac{H_R(s)}{\l+c_R(s,\e)}ds=0.
\end{multline}

By introducing the auxiliary notation 
\begin{equation}\label{def:A}
A(x):= \frac{1}{2n}\int_0^{x}(1-\e/2)\varpi_0'(s(1-\e/2)) \left(\frac{s}{x}\right)^{n}  \frac{H_L(s)}{\l+c_L(s,\e)}ds, 
\end{equation}
\begin{multline}\label{def:B}
B(x):=\frac{1}{2n}\int_{x}^{1}(1-\e/2)\varpi_0'(s(1-\e/2)) \left(\frac{s}{x}\right)^{-n}\frac{H_L(s)}{\l+c_L(s,\e)}ds \\
+ \frac{1}{2n}\int_{1}^{+\infty}(1+\e/2)\varpi_0'(s(1+\e/2))\left(\frac{s}{x}\right)^{-n}\left(\frac{1+\e/2}{1-\e/2}\right)^{-n} \frac{H_R(s)}{\l+c_R(s,\e)}ds,
\end{multline}
we can write \eqref{inteqleft_final} and \eqref{eq:1derivative} in a more compact form as follows
\begin{align*}
    H_L(x)-A(x)-B(x)=0,\\
    \frac{x}{n} H_L'(x)+A(x)-B(x)=0.
\end{align*}
Thus, we have the following relations:
\begin{equation}\label{eq:relationAyB}
H_L(x)+\frac{x}{n} H_L'(x)=2B(x), \qquad H_L(x)-\frac{x}{n}H_L'(x)=2A(x).
\end{equation}

Now, taking another derivative in \eqref{eq:1derivative} we obtain taking into account that the boundary terms do not cancel out and give us now an extra term that
\begin{multline*}
H_L''(x) -\frac{(-n)(-n-1)}{x^2}  \frac{1}{2n}\int_0^{x}(1-\e/2)\varpi_0'(s(1-\e/2)) \left(\frac{s}{x}\right)^{n}  \frac{H_L(s)}{\l+c_L(s,\e)}ds\\ 
-\frac{(+n)(+n-1)}{x^2}\frac{1}{2n}\int_{x}^{1}(1-\e/2)\varpi_0'(s(1-\e/2)) \left(\frac{s}{x}\right)^{-n}\frac{H_L(s)}{\l+c_L(s,\e)}ds\\
-\frac{(+n)(+n-1)}{x^2}\frac{1}{2n}\int_{1}^{+\infty}(1+\e/2)\varpi_0'(s(1+\e/2))\left(\frac{s}{x}\right)^{-n}\left(\frac{1+\e/2}{1-\e/2}\right)^{-n} \frac{H_R(s)}{\l+c_R(s,\e)}ds\\
+\frac{1-\e/2}{x(1-\e/2)}\frac{(1-\e/2)\varpi_0'(x(1-\e/2)) H_L(x)}{\l+c_L(x,\e)}=0,
\end{multline*}
which can be simplified using the notation introduced above as follows
\[
H_L''(x)-\frac{(-n)(-n-1)}{x^2}  A(x)-\frac{(+n)(+n-1)}{x^2} B(x)+\frac{1}{x}\frac{(1-\e/2)\varpi_0'(x(1-\e/2)) H_L(x)}{\l+c_L(x,\e)}=0.
\]
Finally, using \eqref{eq:relationAyB} and simplifying, we arrive at
\[
H_L''(x)+\frac{1}{x}H_L'(x)-\frac{1}{x^2}\left(n^2-x\frac{(1-\e/2)\varpi_0'(x(1-\e/2))}{\l+c_L(x,\e)}\right)H_L(x)=0,
\]
for $x\in (0,1).$

We now repeat the same procedure by taking derivatives at \eqref{inteqright_final}. Applying the same argument with suitable modifications, we ultimately obtain
\begin{equation*}
H_R''(x)+\frac{1}{x}H_R'(x)-\frac{1}{x^2}\left(n^2-x\frac{(1+\e/2)\varpi_0'(x(1+\e/2))}{\l+c_R(x,\e)}\right)H_R(x)=0,
\end{equation*}
for $x\in (1,+\infty)$.

\begin{remark}
The above computation can also be interpreted from the viewpoint of Green's functions. Indeed, the kernels appearing in \eqref{inteqleft_final}--\eqref{inteqright_final} coincide with the Green function associated with the radial operator
\[
\Delta_n:=\partial_{rr}+\frac1r\partial_r-\frac{n^2}{r^2}.
\]
Thus, after introducing
\[
g(r):=r(\lambda+c(r))h(r),
\]
one may formally apply $\Delta_n$ to the integral equations and directly recover the corresponding ODEs.
\end{remark}

\subsubsection{The boundary conditions of the ODE}
\begin{comment}
In addition, while we initially refer to the boundary conditions (see \eqref{bc:fromeq}), they are not independently prescribed but instead follow from the equation itself. We state them below:
$$\left(\lim_{r \to 0^+}r^{2-n}f(r), \lim_{r\to + \infty}r^n f(r)\right) = (0,0).$$

We now examine how these conditions are affected by the change of variables introduced above. We recall that
\begin{align*}
h_L(x)&=2nx(1-\e/2)(\l + c(x(1-\e/2)))f(x(1-\e/2)), \qquad x\in (0,1),\\
h_R(x)&=2nx(1+\e/2)(\l + c(x(1+\e/2)))f(x(1+\e/2)), \qquad x\in(1,+\infty).
\end{align*}
Since 
\end{comment}
From \eqref{bc:fromeq} together with \eqref{def:hL}, \eqref{def:hR} and
$$\lim_{r \to 0^+} c_L(r,\e)=-\frac{\varpi_0(0)}{2}-\frac{\varpi_0(1+\e/2)-\varpi_0(1-\e/2)}{2}, \qquad \lim_{r \to +\infty}  c_R(r,\e)=0,$$
we deduce 
$$\left(\lim_{x \to 0^+}x^{1-n}H_L(x),\lim_{x\to + \infty}x^{n-1} H_R(x)\right)=(0,0).$$

\subsection{The ODE system}
After all these calculations, we have relaxed our problem from a system of coupled integral equations to a system of ODEs. That is, we have shown that a smooth solution of \eqref{inteqleft_final}-\eqref{inteqright_final}, $H=(H_L,H_R)$,  satisfies the following system of uncoupled ODEs:

\begin{equation}\label{eqfrob:odeleft}
\left\{
\begin{aligned}
    H_L''(x)+\frac{1}{x}H_L'(x)-\frac{1}{x^2}\left(n^2 - x\frac{(1-\e/2)\varpi_0'(x(1-\e/2))}{\lambda + c_L(x,\e)}\right)H_L(x) &= 0, \quad x \in (0,1), \\
    \lim_{x \to 0^+} x^{1-n}H_L(x) &= 0,
\end{aligned}
\right.
\end{equation}
and
\begin{equation}\label{eqfrob:oderight}
\left\{
\begin{aligned}
    H_R''(x) + \frac{1}{x} H_R'(x) - \frac{1}{x^2} \left(n^2 - x\frac{(1+\e/2)\varpi_0'(x(1+\e/2))}{\lambda + c_R(x,\e)}\right) H_R(x) &= 0, \quad x \in (1,+\infty), \\
    \lim_{x \to +\infty} x^{n-1}H_R(x) &= 0,
\end{aligned}
\right.
\end{equation}

The process to find a solution to \eqref{inteqleft_final} and \eqref{inteqright_final} will consist of three main steps:

\begin{enumerate}
    \item Solve the system of ODEs for general $\e > 0$ and the limiting system obtained for $\e = 0$. 
    
    \item Study the difference between the solution for general $\e > 0$ and the limiting one obtained for $\e = 0$. Here, we require a precise control of the difference in terms of the parameter $\e$.

    \item Using the above, we reduce the problem of solving the original system of integral equations to solving a single equation for the only remaining unknown at that point, namely the parameter  $\lambda$.

\end{enumerate}

\section{Existence and uniqueness}
In this section we shall study equations \eqref{eqfrob:odeleft} and \eqref{eqfrob:oderight}.

It is important to emphasize that \( H_L, H_R, c_L, c_R \), and \( \lambda \) implicitly depend on the parameter \( \varepsilon \).
We note that at this time, both systems are decoupled. The only relationship between them is the eigenvalue $\l$ that appears in both.

In addition, we note that the boundary values $H_L(1^-)$ and $H_R(1^+)$ are not fixed.   For convenience, we will impose these two values  to be $1$ and change the names of the functions from $H_R$ and $H_L$ to $h_R$ and $h_L$ to emphasize this fact. 

With this normalization, the general solution to \eqref{eqfrob:odeleft}--\eqref{eqfrob:oderight} can be written as \( H = (A h_L, B h_R) \), where $A$, $B\in \mathbb{R}$ are free parameters. \\

In addition, we denote by $\mathring{h}$ the solution of the limiting problem of \eqref{eqfrob:odeleft}-\eqref{eqfrob:oderight} resulting from making $\varepsilon\to 0$. With a slight abuse of notation, we will use the same symbol \( \mathring{h} \) to refer to both the left and right parts of the solution, \( \mathring{h}_L \) and \( \mathring{h}_R \), when no confusion arises. 

That is, $\mathring{h}$ solves
\begin{equation}\label{eq:odeep=0}
    \mathring{h}''(x)+\frac{1}{x}\mathring{h}'(x)-\frac{1}{x^2}\left(n^2-\frac{x \varpi_0'(x)}{\lambda_0+\mathring{c}(x)}\right)\mathring{h}(x)=0, \qquad \text{for } x\in(0,1)\cup(1,+\infty)
\end{equation}
with
\begin{equation}\label{eqBC:odeep=0}
\lim_{x\to 0^+}x^{1-n}\mathring{h}(x)=0=\lim_{x\to+\infty}x^{n-1}\mathring{h}(x), \qquad \text{and} \qquad \lim_{x\to 1^-}\mathring{h}(x)=1=\lim_{x\to 1^+}\mathring{h}(x).
\end{equation}

With a slight abuse of notation, we let \( \mathring{c}(x) \) denote either \( c_L(x,0) \) or \( c_R(x,0) \) in their respective domains, as they coincide by~\eqref{cl0equalcr0}. That is,
\[
\mathring{c}(x)=-\frac{\log(1 + x^2)}{2x^2}.
\]

For the remainder of this section, and in order to simplify notation, we will denote
\begin{align*}
\Psi_0(x):&=-\frac{x \varpi_0'(x)}{\lambda_0+\mathring{c}(x)}=\frac{2x^2}{(1+x^2)^{2}}\frac{1}{\frac{\log(2)}{2}-\frac{\log(1+x^2)}{2x^2}}.
\end{align*}

The function $\Psi_0$ is real analytic in $[0,1)\cup(1,+\infty)$. Moreover, it has a pole of order 1 in $x=\pm 1$, singular points at $x=\pm i$ and discontinuous branches at $\pm i (1,+\infty).$

\begin{comment}
Here, it is important to recall that $\l_0$ is given by 
$$ \l_0=\int_0^1 s\varpi_0(s)ds=\frac{\log(2)}{2}.$$
\end{comment}

\subsection{The limiting ODE in the right}\label{s:FrobRight}
In this paragraph, we are going to study
\begin{equation}\label{eq:oderight}
    \mathring{h}_R''(x)+\frac{1}{x}\mathring{h}_R'(x)-\frac{1}{x^2}\left(n^2+\Psi_0(x)\right)\mathring{h}_R(x)=0, \qquad \text{for } x\in(1,+\infty),
\end{equation}
with
\begin{equation}\label{eq:bcoderight}
\lim_{x\to 1^+}\mathring{h}_R(x)=1, \qquad \text{and} \qquad \lim_{x\to +\infty}x^{n-1}\mathring{h}_R(x)=0.
\end{equation}

The main result of this section is the following lemma.
\begin{lemma}\label{l:prophr}
The system \eqref{eq:oderight}-\eqref{eq:bcoderight} admits a unique solution $\mathring{h}_R\in C([1,+\infty))\cap C^\infty((1,+\infty))$.  Moreover, this solution satisfies:
\begin{enumerate}
    \item \( \mathring{h}_R(x) > 0 \) for all \( x \in [1,+\infty) \).
    \item $\|x^n\mathring{h}_R\|_{L^\infty((1,+\infty))}<\infty$.
    \item \( \mathring{h}_R(x) \) behaves as \( 1 + O(|x - 1|\log|x - 1|^{-1}) \) as \( x \to 1^+ \).
    \item $\|\mathring{h}_R\|_{C^2([a,+\infty))}<+\infty$, with $a>1$.
\end{enumerate}
\end{lemma}

Now, we have to work with the infinite domain $(1,+\infty).$ To handle it, it will be convenient to introduce the auxiliary variable $z:=1/x\in(0,1)$ and the auxiliary function  $\mathring{\hh}_R(z):=\mathring{h}_R(1/z)$. Then,
\[
\mathring{h}_R'(x)=-z^2 \mathring{\hh}_R'(z), \qquad  \mathring{h}_R''(x) = 2z^3 \mathring{\hh}_R'(z) + z^4 \mathring{\hh}_R''(z),
\]
and \eqref{eq:oderight} translate into
\begin{equation}\label{eq:oderightnew}
\mathring{\hh}_R''(z)+\frac{1}{z}\mathring{\hh}_R'(z)-\frac{1}{z^2}\left(n^2+\Psi_0(1/z)\right)\mathring{\hh}_R(z)=0, \qquad \text{for } z\in(0,1),
\end{equation}
with the boundary conditions \eqref{eq:bcoderight} converted into
\begin{equation}\label{eq:bcoderightnew}
\lim_{z\to 0^+}z^{1-n}\mathring{\hh}_R(z)=0, \qquad \text{and} \qquad \lim_{z\to 1^-}\mathring{\hh}_R(z)=1.
\end{equation}

\subsubsection{Analysis of \eqref{eq:oderightnew}}\label{ss:analysisRHS}

We first consider 
\begin{align*}
    h''(z)+\frac{1}{z}h'(z)-\frac{n^2}{z^2}h(z) = g(z), \qquad \text{for } z\in(0,1). 
\end{align*}
Since the solutions of the homogeneous part are $Az^n+Bz^{-n}$, we can write
\begin{align*}
h(z)= A z^n+ Bz^{-n}+\frac{z^n}{2n}\int_{0}^z y^{-n+1}g(y)dy-\frac{z^{-n}}{2n}\int_0^z y^{n+1}g(y)dy,
\end{align*}
where it is assumed that $g(y)$ has enough cancellation at zero for the first integral to be convergent. Since we need $h(0)=0$ we pick $B=0$. Thus, we look for solutions to \eqref{eq:oderightnew} through the integral equation
\begin{align}\label{eq:oderightnewint}
    \mathring{\hh}_R(z)= Az^n+\frac{z^n}{2n}\int_0^z y^{-n+1}\frac{\Psi_0(1/y)}{y^2} \mathring{\hh}_R(y)dy-\frac{z^{-n}}{2n}\int_0^z y^{n+1}\frac{\Psi_0(1/y)}{y^2} \mathring{\hh}_R(y)dy.
\end{align}
Here we remark that
\[
\frac{\Psi_0(1/y)}{y^2}
=
\frac{4}{
(1+y^{2})^{2}\left(\log 2 - y^{2}\,\log\!\left(1+\frac{1}{y^{2}}\right)\right)
},
\]
and then $\Psi_0(1/y)/y^2\in C([0,a])$ for any $a<1$. We write $C(a):=\|\Psi_0(1/y)/y^2\|_{L^\infty([0,a])}$. We take $\mathring{\hh}_R(z)=z^n f(z)$. Then, \eqref{eq:oderightnewint} reads
\begin{align}\label{eq:oderightnewintf}
f(z)=A+\frac{1}{2n}\int_0^z y\frac{\Psi_0(1/y)}{y^2}f(y)dy -\frac{z^{-2n}}{2n}\int_0^z y^{2n+1}\frac{\Psi_0(1/y)}{y^2}f(y)dy.
\end{align}
Let us fix $A=1$. To solve \eqref{eq:oderightnewintf} we can apply a classical contraction argument. We define the space $ X_\alpha:=\{ f\in  C([0,a]): \|f\|_{X_\alpha}<+\infty\}  $ with the norm $\|f\|_{X_\alpha}:= \sup_{z\in[0,a]}|e^{-\alpha z} f(z)|$, $\alpha>0$. Thus the application
$$F[f](z):= 1+\frac{1}{2n}\int_0^z y\frac{\Psi_0(1/y)}{y^2}f(y)dy -\frac{z^{-2n}}{2n}\int_0^z y^{2n+1}\frac{\Psi_0(1/y)}{y^2}f(y)dy,$$
satisfies
\begin{align*}
\|F[f]\|_{X_\alpha}&\leq 1+ \frac{C(a)}{2n}\|f\|_{X_\alpha}\left(\sup_{z\in [0,a]}\left(e^{-z\alpha}\int_0^z y e^{\alpha y}dy \right)
+ \sup_{z\in [0,a]} \left(e^{-\alpha z }z^{-2n}\int_0^z y^{2n+1} e^{\alpha y}dy\right)\right)\\
&\leq 1+ \frac{C(a)}{2n}\|f\|_{X_\alpha}\frac{C}{\alpha}.
\end{align*}
We then have that $F$ is a linear contraction for $\alpha$ large enough. This last fact gives us that there exists a solution $f(z)$, $z\in (0,1)$, which belongs to $C([0,a])$ for any $a<1$, of \eqref{eq:oderightnewintf}.  By a bootstrapping, we get that this solution is smooth in $(0,a)$ for any $a<1$, since $\Psi_0(1/z)/z^2$ is smooth in $(0,a)$. Also, since $\Psi_0(1/z)/z^2\in C^1([0,a])$, we have $f\in C^2([0,a])$ for any $a<1$. 

Thus we have proven the existence of a solution 
\begin{align}\label{haux}\mathring{\hh}_R(z)=z^n f(z)\in C^2([0,a]),
\end{align}
of \eqref{eq:oderightnew}, smooth in  $(0,a)$, for all $a<1$ and satisfying the first boundary condition of \eqref{eq:bcoderightnew}.

Since $f\in C([0,a])$, the integral equation \eqref{eq:oderightnewintf} can be evaluated at $z=0$. The first integral clearly tends to zero as $z\to0^+$, whereas
\[
z^{-2n}\int_0^z y^{2n+1}\frac{\Psi_0(1/y)}{y^2}f(y)\,dy
=O(z^2),
\]
since $\Psi_0(1/y)/y^2$ is bounded near the origin and the integral is $O(z^{2n+2})$. Hence, evaluating \eqref{eq:oderightnewintf} at $z=0$ yields
\[
f(0)=A=1.
\]
Consequently,
\[
\lim_{z\to0^+}z^{-n}\mathring{\hh}_R(z)
=
\lim_{z\to0^+}f(z)
=
1>0,
\]
which verifies the positivity assumption appearing later in Lemma~\ref{l:positivity_right}.

It remains to study the behavior of this solution as  $z=1$, and to check the second boundary condition in \eqref{eq:bcoderightnew}. Since we have fixed \( A = 1 \) in \eqref{eq:oderightnewintf}, the function \( \mathring{\hh}_R(z) \) defined in \eqref{haux} does not necessarily satisfy \( \mathring{\hh}_R(1^-) = 1 \). However, we will show that \( 0 < \mathring{\hh}_R(1^-) < +\infty \). Therefore, by simply normalizing the function via $\mathring{\hh}_R(z)/\mathring{\hh}_R(1^-),$ 
we obtain a solution to \eqref{eq:oderightnew} satisfying \eqref{eq:bcoderightnew}.

\begin{remark}
The same argument works for  other profiles, as for example, either $e^{-x^2}$ or $(1+x^2)^{-k}$, with  $k>1$,  since the corresponding $\Psi_0(1/x)$ also satisfies that $||\Psi_0(1/x)/x^2||_{ C^1([0,a])}<+\infty$ with $a<1$ (actually it is better than that). 
\end{remark}

\subsubsection{Regularity  of $\mathring{\hh}_R$ at $z_0=1$}

The only point  we need to study, in order to get the regularity of $\mathring{\hh}_R$ in the closed interval $[0,1]$, is $z_0=1$.

Following the notation in \cite{whittaker} and using the classical Frobenius method, we can write \eqref{eq:oderightnew} as
\[
\mathring{\hh}_R''(z)+p(z)\mathring{\hh}_R'(z)+q(z)\mathring{\hh}_R(z)=0,
\]
with
\[
p(z)=\frac{P(z-z_0)}{z-z_0}, \qquad \text{and} \qquad q(z)=\frac{Q(z-z_0)}{(z-z_0)^2}.
\]

\underline{\textbf{Regular-singular point at $z_0=1$}:}
Here, we have 
\[
P(z-z_0)=\frac{z-z_0}{z}, \qquad Q(z-z_0)=-(z-z_0)^2 \,\frac{n^2+\Psi_0(1/z)}{z^2}, \qquad z_0=1.
\]
That is,
\[
P(z)=\frac{z}{z+1}, \quad Q(z)=-\left(\frac{z}{z+1}\right)^2 \left(n^2+\Psi_0\left(\frac{1}{z+1}\right)\right),
\]
with
\[
\Psi_0\left(\frac{1}{z+1}\right)= \left(\frac{1}{1-2\log(2)}\frac{1}{z}+O(1)\right), \qquad \text{as } z\to 0,
\]
and we have that the indicial polynomial
\[
I(r)=r(r-1) + P(0)r + Q(0),
\]
is given by
\[
I(r)=r(r-1),
\]
with the roots $r=1$ and $r=0$.

Therefore, we have two linearly independent solutions:
\begin{itemize}
    \item Associated to the larger root $r = 1$ we have an analytic solution on (0, 1), which we denote by $q_1(x)$,  which can be obtained from the expansion
    \[
    q_1(z)=\sum_{k=1}^{\infty}q_1^{[k]}(z-1)^k,
    \]
    where $q_1^{[k]}\in\mathbb{R}$. We choose, without loss of generality, $q_1^{[1]}=1$.

    \item If the roots differs by a integer number, as in our case, then the second solution can be found through the expression
    \begin{align*}
q_2(z)=q_{2}^{[-1]} q_1(z)\log(1-z)+\sum_{k=0}^\infty q_{2}^{[k]}(z-1)^k,
\end{align*}
where $q_{2}^{[k]}\in \mathbb{R}$ for $k=-1,0,1,...$, with $q_2^{[0]}\neq 0$ (see  \cite{whittaker}, pag. 200-201).
The series converges for $z\in (0,1).$
Here, it is crucial to determine whether the logarithmic term appears. To that end, we observe that, for our equation, the coefficient $q_{2}^{[-1]}$ is given by
\[q_2^{[-1]} = \frac{\varpi_0'(1)}{(\pa_x c_R)(1,0)}\frac{q_2^{[0]}}{q_1^{[1]}} \neq 0,\]
which implies that $q_2(x)$ has a singularity at $z = 1$ of the form $(z - 1) \log|z - 1|$.\\
\end{itemize}

Thus we have that $\mathring{\hh}_R(z)$ in \eqref{haux} can be written as
\begin{align}\label{hraux}
\mathring{\hh}_R(z)=\alpha q_1(z)+\beta q_2(z),\qquad \text{for } z\in (0,1),
\end{align}
for some $\alpha,\, \beta\in \mathbb{R}.$ In particular, this shows that $$\lim_{z\to 1^-}|\mathring{\hh}_R(z)|<+\infty.$$
Next we show that $\mathring{\hh}_R(z)(1^-)>0$. This fact is just an application of the following result.

\begin{lemma}\label{l:positivity_right}
Let \( f \) be a solution of \eqref{eq:oderightnew} that satisfies
\[
\lim_{z\to 0^+} z^{1-n}f(z)=0, \qquad \text{and} \qquad \lim_{z\to 0^+} z^{-n} f(z)>0.
\]
Then, it satisfies $f(z)>0$ for all $z\in(0,1].$ 
\end{lemma}
\begin{proof}
A more general proof of this fact has been postponed to Section \ref{s:generalFrob}, see Lemma \ref{l:generalRIGHT}.
\end{proof}

Then we can redefine $\mathring{\hh}_R(z)$ as $\mathring{\hh}_R(z)/ \mathring{\hh}_R(1^-)$ to produce a solution of \eqref{eq:oderightnew} which satisfies \eqref{eq:bcoderightnew}. We remark that, since $q_1(1)=0$, the coefficient $\beta$ in \eqref{hraux} has to be different from zero. Thus, $\mathring{\hh}_R(z)$ has a singularity of the type $(z-1)\log(1-z)$ at $z=1^-$. 

Uniqueness follows from Lemma \ref{l:positivity_right} and expression \eqref{eq:oderightnewint}. Indeed, if there exists a second solution, the difference $d$ between them satisfies the equation with boundary conditions $d(0)=d(1)=0$.
Therefore, by Lemma \ref{l:positivity_right}, we would find that $A$, in \eqref{eq:oderightnewint}, must be equal to zero. The same contraction argument after \eqref{eq:oderightnewint} would yield that then $d=0$. 

\subsubsection{Back to the original variable}

We now revert the change of variables and return to the original formulation of the equation. That is, for \( x \in (1, +\infty) \), we define
\[
\mathring{h}_R(x) := \mathring{\hh}_R\left( \frac{1}{x} \right), \qquad \text{and} \qquad \tilde{g}_1(x) := q_1\left( \frac{1}{x} \right).
\]
Then, taking the Wronskiano
$$W[\tilde{g}_1,\mathring{h}_R](x)=\tilde{g}_1(x)\mathring{h}_R'(x)-\tilde{g}'_1(x)\mathring{h}_R(x),$$
and using that
\[
\lim_{x\to 1^+}W[\tilde{g}_1,\mathring{h}_R](x)=1,
\]
it is immediate to obtain, by means of a consequence of the Abel-Liouville formula, that
\begin{equation}\label{Wrosk_cte_right}
W[\tilde{g}_1,\mathring{h}_R](x)=\frac{1}{x} \cdot \lim_{x\to 1^+}W[\tilde{g}_1,\mathring{h}_R](x)=\frac{1}{x}, \quad \forall x\in(1,+\infty).
\end{equation}

\begin{comment}
\begin{proof}[Proof of \eqref{Wrosk_cte_right}]
We have the second-order homogeneous differential equation: 
$$d''(x)+\frac{1}{x}d'(x)-\frac{1}{x^2}\left(n^2-\frac{x\varpi_0'(x)}{\lambda_0-\frac{1}{x^2}\int_0^x s \varpi_0(s) ds}\right)d(x)=0.$$
We have seen before that $\mathring{h}_R$ and $\tilde{g}_1$ are two linearly independent solutions of it. Then, aplying Abel-Liouville formula we have 
\[
W'[\tilde{g}_1,\mathring{h}_R](t)=-\frac{1}{t}W[\tilde{g}_1,\mathring{h}_R](t) \qquad \text{for all } t\in(1,+\infty).
\]
Integrating the above, we directly get
\[
\int_{1+}^x \frac{W'[\tilde{g}_1,\mathring{h}_R](t)}{W[\tilde{g}_1,\mathring{h}_R](t)}dt = -\int_{1^+}^x \frac{1}{t}dt \quad \Longrightarrow \quad W[\tilde{g}_1,\mathring{h}_R](x)=\frac{W[\tilde{g}_1,\mathring{h}_R](1^+)}{x},
\]
for all $x\in(1,+\infty)$ which implies the desired result.
\end{proof}
\end{comment}

To close this section, we establish a uniform estimate for $\tilde{g}_1$ that characterizes its asymptotic behavior as $x \to +\infty$.

\begin{lemma}\label{l:decg1} The function $\tilde{g}_1$ satisfies
\begin{align*}
    \|x^{-n}\tilde{g}_1\|_{L^\infty([1,+\infty))}\leq C,
\end{align*}
where $C$ denotes a universal constant.
\end{lemma}
\begin{proof}
We first notice that $\|\tilde{g}_1\|_{L^\infty([1,x_0])}\leq C(x_0)$ for a finite constant $C(x_0)$ and $x_0>1$. The function $g:=\tilde{g}_1$ satisfies \eqref{eq:signf}. Thus
\begin{align*}
    g(x)=Ax^n+Bx^{-n}+\frac{x^{n}}{2n}\int_{x_0}^x z^{-n+1}\frac{\Psi_0(z)}{z^2}g(z)dz-\frac{x^{-n}}{2n}\int_{x_0}^x z^{n+1}\frac{\Psi_0(z)}{z^2}g(z)dz,
\end{align*}
for $x_0>1$ and some constants $A=A(x_0)$, $B=B(x_0)$ ($|A|,\, |B|<\infty$ for $x_0<\infty$). Let us take $G(x)=x^{-n}g(x).$ We obtain that
\begin{align*}
G(x)=A+Bx^{-2n}+\frac{1}{2n}\int_{x_0}^x z\frac{\Psi_0(z)}{z^2}G(z)dz-\frac{x^{-2n}}{2n}\int_{x_0}^x z^{2n+1}\frac{\Psi_0(z)}{z^2}G(z)dz.
\end{align*}
Thus
\begin{align*}
    \|G\|_{L^\infty([x_0,+\infty))}\leq |A|+|B|+ \|G\|_{L^\infty([x_0,\infty))}\left\|x\frac{\Psi_0(x)}{x^2}\right\|_{L^1([x_0,+\infty))}.
\end{align*}
Since $\Psi_0(x)/x$ is integrable in $[x_0,+\infty)$ for $x_0>1$, we have that $\left\|x\frac{\Psi_0(x)}{x^2}\right\|_{L^1([x_0,+\infty))}$ is $o(x_0)$ when $x_0$ goes to infinity. Thus, taking $x_0$ large enough, we can bound $\|G\|_{L^\infty([x_0,+\infty))}$.
\end{proof}

\subsection{The limiting ODE in the left}\label{s:FrobLeft}
In this paragraph, we are going to study
\begin{equation}\label{eq:odeleft}
    \mathring{h}_L''(x)+\frac{1}{x}\mathring{h}_L'(x)-\frac{1}{x^2}\left(n^2-\frac{x \varpi_0'(x)}{\lambda_0+c_L(x,0)}\right)\mathring{h}_L(x)=0, \qquad \text{for } x\in(0,1),
\end{equation}
with
\begin{equation}\label{eq:bcodeleft}
\lim_{x\to 0^+}x^{1-n}\mathring{h}_L(x)=0, \qquad \text{and} \qquad \lim_{x\to 1^-}\mathring{h}_L(x)=1.
\end{equation}

The main results of this section are the following two lemmas.
\begin{lemma}\label{l:hlring}
The system \eqref{eq:odeleft}--\eqref{eq:bcodeleft} admits a unique solution \( \mathring{h}_L \in C([0,1])\cap C^\infty((0,1))\). Moreover, this solution satisfies:
\begin{enumerate}
    \item \( \mathring{h}_L(x) > 0 \) for all \( x \in (0, 1] \).
    \item \( \mathring{h}_L(x) \) behaves as \( 1 + O(|x - 1|\log|x - 1|^{-1}) \) as \( x \to 1^- \).
    \item $\|\mathring{h}_L\|_{C^k([0,a])}\leq C_{k,a}$, $k=1,2,...$ and $0<a<1$.
\end{enumerate}
\end{lemma}

\begin{lemma}\label{g1} There exist $g_1\in C^\infty((0,1])$, linearly independent of $\mathring{h}_L$, solving \eqref{eq:odeleft} and satisfying
\begin{align*}
    \|x^{n}g_1\|_{L^\infty([0,1])}\leq C,
\end{align*}
where $C$ denotes a universal constant, and with Wronskiano
\begin{equation*}
W[\mathring{h}_L,g_1](x)=\frac{1}{x}, \quad \forall x\in(0,1).
\end{equation*}

\end{lemma}

The proofs of Lemmas \ref{l:hlring} and \ref{g1}   are similar to  what we did in Section \ref{s:FrobRight}. In this case we can even apply Frobenius method in both points $x_0=0$ and $x_0=1$. The main difference is that instead of Lemma \ref{l:positivity_right} we had to apply the next result.

\begin{lemma}\label{l:positivity_left}
Let \( f \) be a solution of 
\begin{equation}\label{eq:signf}
f''(x)+\frac{1}{x}f'(x)-\frac{n^2}{x^2}f(x)=\frac{\Psi_0(x)}{x^2}f(x).
\end{equation}
satisfying
\[
\lim_{x\to 0^+}x^{1-n}f(x)=0, \qquad \text{and} \qquad \lim_{x\to 0^+} x^{-n}f(x)>0.
\]
Then, it satisfies $f(x)>0$ for all $x\in(0,1].$ 
\end{lemma}

\begin{proof}
A more general proof of this fact has been postponed to Section \ref{s:generalFrob}, see Lemma \ref{l:generalLeft}.
\end{proof}

\begin{remark}\label{r:singleft}
The same result holds if \( \lim_{x\to 0^+}x^{-n} f(x) > 0 \) is replaced by \( \lim_{x\to 0^+} x^{-n}f(x) < 0 \), but in this case, \( f(x) < 0 \) for all \( x \in (0,1] \). 
\end{remark}

\subsection{The ODE system for $\e>0$}\label{s:generalFrob}
Now, we return to the main objective of this section, namely to solve the original system of ODEs \eqref{eqfrob:odeleft}-\eqref{eqfrob:oderight} depending implicitly on the parameter $\e > 0$. We recall the system below:

\begin{align}\label{eqfrob:leftgeneral}
h_L''(x) + \frac{h_L'(x)}{x}- \frac{1}{x^2}\left(n^2-\frac{x(1-\e/2)\varpi_0'(x(1-\e/2))}{\lambda+c_L(x,\e)}\right)h_L(x)&=0, \qquad \text{for } 0 < x <1,\\  
\lim_{x\to 0^+}x^{1-n}h_L(x)&=0, \nonumber \\
\lim_{x\to 1^-}h_L(x)&=1, \nonumber
\end{align}
and 
\begin{align}\label{eqfrob:rightgeneral}
h_R''(x) + \frac{h_R'(x)}{x} - \frac{1}{x^2}\left(n^2-\frac{x(1+\e/2)\varpi_0'(x(1+\e/2))}{\lambda+c_R(x,\e)}\right)h_R(x)&=0 \qquad \text{for } 1<x<+\infty,\\
\lim_{x\to 1^+}h_R(x)&=1, \nonumber\\
\lim_{x\to +\infty}x^{n-1}h_R(x)&=0. \nonumber
\end{align}

\begin{lemma}\label{l:prop}
Fixed $0<\e_0\ll 1$. Let $0< \e\leq \e_0$. The system \eqref{eqfrob:leftgeneral}-\eqref{eqfrob:rightgeneral} admits a unique solution $(h_L,h_R)\in C([0,1])\times C([1,+\infty))$. The function $h_L$ can be extended to $[0,x_L^\ast)$ and is smooth for $x\in [0,x_L^\ast)$. The function $h_R$ can be extended to $(x_R^\ast,+\infty)$ and  is smooth for $x\in (x_R^\ast,+\infty)$. Moreover:
\begin{enumerate}
    \item $h_L(x)>0$ for all $x\in(0,1].$
    \item $h_R(x)>0$ for all $x\in[1,+\infty).$
    \item $h_L(x)$ behaves as $C_L+O(x-x_L^\ast)\log(|x-x_L^\ast|)$ with  $x_L^\ast(\e)>1.$
    \item $h_R(x)$ behaves as $C_R+O(x-x_R^\ast)\log(|x-x_R^\ast|)$ with  $x_R^\ast(\e)<1.$
    \item $\|h_R\|_{C^2([1,+\infty))}\leq C $.
\end{enumerate}
\end{lemma}

\subsubsection{Solving \eqref{eqfrob:leftgeneral} and \eqref{eqfrob:rightgeneral}}
Lemma \eqref{bd4}–\ref{bounddenominator} and Lemma \eqref{bl3}–\ref{bounddenominatorLeft} imply that each of the functions \( \lambda + c_R(x,\varepsilon) \) and \( \lambda + c_L(x,\varepsilon) \) has a unique simple zero, located respectively at points \( x_R^* \in (0,1) \) and \( x_L^* \in (1,+\infty) \). Obtaining either a solution to \eqref{eqfrob:leftgeneral} and \eqref{eqfrob:rightgeneral}, with their respective boundary conditions, follows the same steps as the construction of $\mathring{h}_L$ and $\mathring{h}_R$ in Sections \ref{s:FrobLeft} and \ref{s:FrobRight}, respectively. This is mainly due to the fact that $x_L^\ast>1$ and $x_R^\ast<1$. Uniqueness also follows by similar arguments that those ones in Sections \ref{s:FrobLeft} and \ref{s:FrobRight}.

The only difficulty is that we have to adapt Lemmas \ref{l:positivity_left} and \ref{l:positivity_right} to the case $\e>0$. Actually, since we did not give a proof of either  \ref{l:positivity_left} or \ref{l:positivity_right}, we will include the case $\e=0$ in the following two results.

\begin{lemma}\label{l:generalLeft}
Let $0\leq \e\leq \e_0$ and let \( f \) be a solution of 
\[
f''(x) + \frac{f'(x)}{x}- \frac{1}{x^2}\left(n^2-\frac{x(1-\e/2)\varpi_0'(x(1-\e/2))}{\lambda(\e)+c(x(1-\e/2))}\right)f(x)=0, \qquad \text{for } 0 < x <x_L^\ast,
\]
that satisfies 
\[
\lim_{x\to 0^+} x^{1-n}f(x) = 0, \qquad \text{and} \qquad \lim_{x\to 0^+}x^{-n}f(x)>0.
\]
Then, taking $n$ big enough, it satisfies $f(x)>0$ for all $x\in(0,x_L^\ast].$     
\end{lemma}

\begin{proof}
In order to proceed, we start writing the above in a more convenient way as follows
\begin{equation*}
f''(x)+\frac{1}{x}f'(x)-\frac{n^2}{x^2}f(x)=\frac{\Psi_\e(x)}{x^2}f(x),
\end{equation*}
where (we write $\lambda \equiv \lambda(\e)$ to make the dependence on $\e$ explicit)
\[
\Psi_\e(x):=-\frac{x(1-\e/2)\varpi_0'(x(1-\e/2))}{\lambda(\e)+c(x(1-\e/2))}.
\]

We find that the general solution is given by
\begin{equation*}
f(x)=A x^n +Bx^{-n}  - \frac{1}{x^n}\frac{1}{2n}\int_0^x \frac{\Psi_\e(z)}{z^2} f(z) z^{1+n}dz + x^n \frac{1}{2n}\int_0^x \frac{\Psi_\e(z)}{z^2} f(z) z^{1-n}dz.
\end{equation*}
Because the boundary condition at zero, and since $\Psi_\e(z)/z^2=O(1)$, the first integral behaves like $x^{2n+2}$ and the second one like $x^2$. Therefore, $B$ has to be equal to $0$. We can then write
\begin{align}\label{formula}
f(x)=Ax^n+\frac{1}{2n}\int_0^x \frac{\Psi_\e(z)}{z}f(z)\left[\left(\frac{x}{z}\right)^n-\left(\frac{z}{x}\right)^n\right]dz,
\end{align}
where $A>0$ to get $\lim_{x\to 0^+} x^{-n}f(x)>0$. Let us take $A=1$ for simplicity. 

Iterating once this formula, we have  the following expression for \( f(x) \):
\begin{align*}
    f(x)=x^n &+\frac{1}{2n}\int_0^x \frac{\Psi_\e(z)}{z}z^n\left[\left(\frac{x}{z}\right)^n-\left(\frac{z}{x}\right)^n\right]dz \\
    &+ \frac{1}{2n}\int_0^x \frac{\Psi_\e(z)}{z}\left(\frac{1}{2n}\int_0^z \frac{\Psi_\e(y)}{y}f(y)\left[\left(\frac{z}{y}\right)^n-\left(\frac{y}{z}\right)^n\right]dy\right)\left[\left(\frac{x}{z}\right)^n-\left(\frac{z}{x}\right)^n\right]dz. 
\end{align*}
Since $\Psi_\e(x)<0$ for all $x\in(0,x_L^\ast]$, and
$$\left(\frac{x}{z}\right)^n-\left(\frac{z}{x}\right)^n\geq 0, \qquad 0<z\leq x,$$
we observe that the last integral term on the right-hand side is non-negative whenever \( f(x) > 0 \).

We know by hypothesis that $f$ is positive in a sufficiently small neighborhood of the origin. As long as $f$ remains positive, the complicated double integral has a favorable sign and may be discarded, leading to the lower bound:
\begin{align}\label{unait}
f(x) > x^n + \frac{1}{2n} \int_0^x \frac{\Psi_\varepsilon(z)}{z} z^n \left[ \left( \frac{x}{z} \right)^n - \left( \frac{z}{x} \right)^n \right] \, dz.
\end{align}

This reduced inequality no longer depends on \( f \), and  its right-hand side is an explicit expression that is considerably simpler to analyze. Therefore, the proof reduces to verifying that this lower bound remains strictly positive for \( x \in (0, x_L^\ast] \).

Recalling the definition of $\Psi_\e$, we have to check that
\[
x^n -\frac{1}{2n}\int_0^x \frac{z(1-\e/2)\varpi_0'(z(1-\e/2))}{\l(\e)+c(z(1-\e/2)}z^{n-1}\left[\left(\frac{x}{z}\right)^n-\left(\frac{z}{x}\right)^n\right]dz>0, \qquad \text{for } x\in(0,x_L^\ast].
\]
Firstly, we consider the change of variable $z(1-\e/2)=\bar{z}$ and introduce the auxiliary notation $x(1-\e/2)=\bar{x}$. Thus, we can just write (after multiplying all the expression by $(1-\e/2)^n$)
\[
\bar{x}^n -\frac{1}{2n}\frac{1}{\bar{x}^n}\int_0^{\bar{x}}  \frac{\varpi_0'(\bar{z})}{\l(\e)+c(\bar{z})}\left(\bar{x}^{2n}-\bar{z}^{2n}\right)d\bar{z}>0, \qquad \text{for } \bar{x}\in(0,(1-\e/2)x_L^\ast].
\]

By definition, \( x_L^\ast \) is a simple zero of \( \l(\e) + c(\cdot(1 - \varepsilon/2)) \). Thus, introducing \( z_L^\ast := (1 - \varepsilon/2)x_L^\ast \) gives us 
\[
\frac{\bar{z}-z_L^\ast}{\l(\e)+c(\bar{z})}\leq C_\e,
\]
with $C_\e=O(1)$. This implies (using also that $(\bar{z}-z_L^\ast)^{-1}\geq (\bar{z}-\bar{x})^{-1}$ as $\bar{z}\in(0,\bar{x})$)
\begin{equation*}
0\leq \int_0^{\bar{x}}  \frac{\varpi_0'(\bar{z})}{\l(\e)+c(\bar{z})}\left(\bar{x}^{2n}-\bar{z}^{2n}\right)d\bar{z}
 \leq C_\e\int_0^{\bar{x}}  \frac{\varpi_0'(\bar{z})}{\bar{z}-\bar{x}}(\bar{x}^{2n}-\bar{z}^{2n})d\bar{z}. 
\end{equation*}
Moreover, using the rough but sufficient estimate \(0 \geq  \varpi_0'(\bar{z}) \geq -2/3 \) and the change of variable \( \bar{z} = \xi \bar{x} \), we obtain
\begin{equation*}
\int_0^{\bar{x}}  \frac{\varpi_0'(\bar{z})}{\bar{z}-\bar{x}}(\bar{x}^{2n}-\bar{z}^{2n})\,d\bar{z}     
\leq \frac{2}{3} \bar{x}^{2n} \int_0^1 \frac{1 - \xi^{2n}}{1 - \xi} \,d\xi 
\leq \frac{2}{3} \bar{x}^{2n} \left( \log(2n) + \gamma + O(n^{-1}) \right),
\end{equation*}
where in the last step we have used an upper bound for the \(2n\)-th harmonic number, and where \(\gamma\) denotes the Euler–Mascheroni constant.

All the above, allows us to say that there exist $n^\ast\gg 1 $ such that for all $n>n^\ast$ we have
\begin{multline*}
\bar{x}^n -\frac{1}{2n}\frac{1}{\bar{x}^n}\int_0^{\bar{x}}  \frac{\varpi_0'(\bar{z})}{\l(\e)+c(\bar{z})}\left(\bar{x}^{2n}-\bar{z}^{2n}\right)d\bar{z}\\
\geq \bar{x}^n\left[1-\frac{C_\e}{n}\left( \log(2n) + \gamma + O\left(\frac{1}{n}\right) \right)\right] >0, \qquad \text{for } \bar{x}\in(0,(1-\e/2)x_L^\ast].
\end{multline*}
\end{proof}

\begin{remark}\label{rem:ngrande}
\begin{comment}Numerical computations show that, in the limiting case $\varepsilon = 0$ where $x_L^\ast=1$, the expression above remains positive for all $x \in (0,1]$ provided $n \geq 4$. This behavior is clearly illustrated in the figure below.  This ensures that \( f \) remains positive throughout the interval, extending the positivity from the neighborhood of the origin to the entire domain under consideration.
\begin{figure}[h]
    \centering
    \includegraphics[width=0.6\textwidth, angle=-90]{monotoniaenN.jpeg}  % Girar 90 grados
    \caption{Plot of the limiting case $\varepsilon = 0$ for $n = 2, 3, 4, 5$.}
    %\label{fig:your_label}
\end{figure}

Moreover, we have also revisited the original, more precise inequality — which still retains the dependence on \( f \) on the right-hand side — and performed numerical computations of the full expression. These computations also indicate that the inequality holds and \( f \) remains positive for all \( x \in (0, 1] \) as soon as \( n \geq  2 \).

Interestingly, additional numerical computations suggest that the behavior improves when \( \varepsilon > 0 \). In particular, for small positive values of \( \varepsilon \), the lower bound appears to be even more favorable, and the positivity of \( f \) seems to hold under the same conditions on \( n \). A rigorous proof in this general setting would require precise control on the locations of \( x_L^\ast \) and \( x_R^\ast \), which is a more involved task. 
\end{comment}

Numerical inspection shows that \eqref{unait} is enough to prove Lemma \ref{l:generalLeft} for $n\geq 4$. In addition, iterating formula \eqref{formula} two more times, one could get a sharper bound for $f$ which would be enough to show Lemma \ref{l:generalLeft} for $n\geq 2$, after some long and tedious computations (we have also checked this fact numerically). To keep the argument accessible, we restricted our analysis to the case of sufficiently large \( n \), where a direct proof can be established. This is the only point in our analysis in which we need to make this assumption.
\end{remark}

\begin{comment}
Similarly, we can solve \eqref{eqfrob:rightgeneral} for variable $x\in(1,+\infty)$ with boundary conditions 
\[
\lim_{x\to 1^+}h_R(x)=1,  \qquad \lim_{x\to +\infty}x^{n-1}h_R(x)=0.
\]
The only new ingredient is to obtain the analogue of Lemma \ref{l:generalLeft}, see Lemma \ref{l:generalRIGHT}, in this new setting.

For convenience, to avoid working on the infinite domain \( x \in (1, +\infty) \), we introduce again the auxiliary variable \( z = 1/x \in (0,1) \) and the auxiliary function \( \hh_R(z) = h_R(1/z) \).

Then, the behavior of \( h_R(x) \) as \( x \to +\infty \) translates into the behavior of \( \hh_R(z) \) as \( z \to 0^+ \), making the analysis more manageable on a bounded interval.

Then, in the new variables, we have to solve
\begin{align*}
\hh_R''(z)+\frac{1}{z}\hh_R'(z)-\frac{1}{z^2}\left(n^2-\frac{z^{-1}(1+\e/2)\varpi_0'(z^{-1}(1+\e/2))}{\l(\e)+c(z^{-1}(1+\e/2))}\right)\hh_R(z)&=0, \qquad \text{for } z\in(0,1),\\
\lim_{z\to 0^+}z^{1-n}\hh_R(z)&=0, \nonumber\\
\lim_{z\to 1^{-}}\hh_R(z)&=1. \nonumber
\end{align*}

In this case, the analogue of Lemma \ref{l:positivity_right} is the following result.

\end{comment}

\begin{lemma}\label{l:generalRIGHT}
Let $0\leq \e\leq \e_0$  and let \( f \) be a solution of 
\[
f''(z) + \frac{f'(z)}{z}- \frac{1}{z^2}\left(n^2-\frac{z^{-1}(1+\e/2)\varpi_0'(z^{-1}(1+\e/2))}{\l(\e)+c(z^{-1}(1+\e/2))}\right)f(z)=0, \qquad \text{for } 0 < z <(x_R^\ast)^{-1}
\]
that satisfies 
\[
\lim_{z\to 0^+} z^{1-n}f(z) = 0, \qquad \text{and} \qquad \lim_{z\to 0^+}z^{-n}f(z)>0.
\]
Then, it satisfies $f(z)>0$ for all $z\in\left(0,(x_R^\ast)^{-1}\right].$    
\end{lemma}
\begin{proof}
Following the same approach as before, we start by rewriting the expression as
\begin{equation*}
f''(z)+\frac{1}{z}f'(z)-\frac{n^2}{z^2}f(z)=\frac{\Phi_\e(z)}{z^2}f(z),
\end{equation*}
where
\[
\Phi_\e(z):=-\frac{z^{-1}(1+\e/2)\varpi_0'(z^{-1}(1+\e/2))}{\l(\e)+c(z^{-1}(1+\e/2))}.
\]

We find that the general solution is given by
\begin{equation*}
f(z)=A z^n +B z^{-n}  - \frac{1}{z^n}\frac{1}{2n}\int_0^z \frac{\Phi_\e(s)}{s^2} f(s) s^{1+n}ds + z^n \frac{1}{2n}\int_0^z \frac{\Phi_\e(s)}{s^2} f(s) s^{1-n}ds.
\end{equation*}
Because of the boundary condition at zero, and since $\Phi_\e(z)/z^2=O(1)$, the first integral behaves like $z^{2n+2}$ and the second one like $z^2$. Therefore, $B$ has to be equal to $0$. We can then write
\[
f(z)=Az^n+\frac{1}{2n}\int_0^z \frac{\Phi_\e(s)}{s}f(s)\left[\left(\frac{z}{s}\right)^n-\left(\frac{s}{z}\right)^n\right]ds,
\]
where $A>0$ to get $\lim_{z\to 0^+} z^{-n}f(z)>0$. Let us take $A=1$ for simplicity.

Since \(\Phi_\e(s) > 0\) for all \(s \in (0, (x_R^\ast)^{-1}]\), and
\[
\left(\frac{z}{s}\right)^n - \left(\frac{s}{z}\right)^n \geq 0, \qquad \text{for } 0 < s \leq z,
\]
we can apply the fact that, by hypothesis, \(f(z) > 0\) near \(z = 0^+\), and bootstrap the positivity of the solution to conclude that \(f\) remains positive throughout the entire interval $z\in\left(0,(x_R^\ast)^{-1}\right].$
\end{proof}

\begin{comment}

Finally, note that we find a solution $(h_L,h_R)$ that can be written as
\[
H(x)=\begin{cases}
    h_L(x)=\sum_{k\geq 1} h_L^{[k]}\log(|x-x_L^\ast|)(x-x_L^\ast)^k + \sum_{k\geq 0}\widetilde{h_L}^{[k]}(x-x_L^\ast)^k, \qquad &x\in (0,1),\\
    h_R(x)=\sum_{k\geq 1} h_R^{[k]}\log(|x-x_R^\ast|)(x-x_R^\ast)^k + \sum_{k\geq 0}\widetilde{h_R}^{[k]}(x-x_R^\ast)^k, \qquad &x\in (1,+\infty),
\end{cases}
\]
where $h_L^{[k]},h_R^{[k]},\widetilde{h_L}^{[k]},\widetilde{h_R}^{[k]}\in\mathbb{R}$.

\end{comment}

\section{Estimates of the Difference}

\subsection{The difference in the right}\label{Linftd}
In this section, we want to make a detailed study of the difference between the solutions of \eqref{eq:oderight}-\eqref{eq:bcoderight} ($\e=0$) and the solution of \eqref{eqfrob:rightgeneral} with $\e>0.$

We will take,
\[
d(x):=h_R(x)-\mathring{h}_R(x), \qquad \text{for } 1<x<+\infty.
\]
To facilitate and alleviate the notation, and since we are interested in 
tracking the dependence on the $\e$ parameter, we introduce and will work  with the following notation:
\[
h_\e(x):=h_R(x), \qquad h_0(x):=\mathring{h}_R(x), \qquad d(x):=h_\e(x)-h_0(x).
\]
In addition, if the domain of a norm is not explicitly specified, we shall understand it as $[1,+\infty).$

\begin{comment}
Therefore, the above systems can be written as
\begin{equation}\label{eq:systemfeRIGHT}
\left\{
\begin{aligned}
h_\e''(x) + \frac{h_\e'(x)}{x}- \frac{1}{x^2}\left(n^2-\frac{x\varpi_0'(x(1-\e/2))(1-\e/2)}{\l+c_R(x,\e)}\right)h_\e(x)&=0, \\
\lim_{x\to 1^+}h_\e(x)&=1,\\
\lim_{x\to +\infty}x^{n-1}h_\e(x)&=0,
\end{aligned}
\right.
\end{equation}
and
\begin{equation}\label{eq:systemf0RIGHT}
\left\{
\begin{aligned}
h_0''(x)+\frac{1}{x}h_0'(x)-\frac{1}{x^2}\left(n^2-\frac{x\varpi_0'(x)}{\lambda_0+c_R(x,0)}\right)h_0(x)&=0, \\
\lim_{x\to 1^+}h_0(x)&=1,\\
\lim_{x\to +\infty}x^{n-1}h_0(x)&=0.
\end{aligned}
\right.
\end{equation}

Taking $d=h_\e-h_0$, with $h_\e$ and $h_0$ solving respectively \eqref{eq:systemfeRIGHT} and \eqref{eq:systemf0RIGHT}, it satisfies
\begin{multline*}
d''(x)+\frac{1}{x}d'(x)-\frac{1}{x^2}\left(n^2-\frac{x\varpi_0'(x)}{\lambda_0+c_R(x,0)}\right)d(x)\\
=-\frac{1}{x^2}\left(\frac{x\varpi_0'(x(1-\e/2))(1-\e/2)}{\l+c_R(x,\e)}-\frac{x\varpi_0'(x)}{\lambda_0+c_R(x,0)}\right)h_\e(x),
\end{multline*}
for $x\in(1,+\infty)$ with boundary conditions
\[
\lim_{x\to 1^+}d(x)=0, \qquad \lim_{x\to +\infty}x^{n-1}d(x)=0.
\]
\end{comment}

We  find that $d$ satisfies
\begin{align*}
d''(x) + \frac{1}{x} d'(x)-\frac{1}{x^2}\left(n^2-\frac{x\varpi_0'(x)}{\lambda_0+c_R(x,0)}\right)d(x)&=\frac{1}{x}E(x)h_\e(x), \qquad x\in(1,+\infty),\\
\lim_{x\to 1^+}d(x)&=0,\\
\lim_{x\to +\infty}x^{n-1}d(x)&=0,
\end{align*}
where
\begin{equation}\label{def:E_R}
E(x):=-\left(\frac{\varpi_0'(x(1+\e/2))(1+\e/2)}{\l+c_R(x,\e)}-\frac{\varpi_0'(x)}{\lambda_0+c_R(x,0)}\right).
\end{equation}

We note that, if $f^{\sharp}$ solves the homogeneous system given by
\begin{align*}
(f^{\sharp})''(x) + \frac{1}{x} (f^\sharp)'(x)-\frac{1}{x^2}\left(n^2-\frac{x\varpi_0'(x)}{\lambda_0+c_R(x,0)}\right)f^\sharp(x)&=0, \qquad x\in(1,+\infty),\\
\lim_{x\to 1^+}f^\sharp(x)&=0,\\
\lim_{x\to +\infty}x^{n-1}f^\sharp(x)&=0,
\end{align*}
then $f^\sharp$ is the trivial solution, i.e., $f^\sharp(x) = 0$ in all $x\in(1,+\infty)$. This is a consequence of Lemma~\ref{l:positivity_right}. Indeed, introducing $\widetilde{f^\sharp}(z):=f^\sharp(1/z)$ with $z\in(0,1),$
the function $\widetilde{f^\sharp}$ satisfies \eqref{eq:oderightnewintf} with
\[
\lim_{z\to0^+}z^{-n}\widetilde{f^\sharp}(z)=A.
\]
If $A>0$, Lemma~\ref{l:positivity_right} gives $\widetilde{f^\sharp}>0$ on $(0,1]$, while if $A<0$, applying the lemma to $-\widetilde{f^\sharp}$ yields $\widetilde{f^\sharp}<0$ on $(0,1]$. Since $\widetilde{f^\sharp}(1)=f^\sharp(1)=0$, both cases are impossible. Hence $A=0$. Therefore, \eqref{eq:oderightnewintf} reduces to the corresponding homogeneous integral equation, whose unique solution is the trivial one by the same contraction argument used above, see section \ref{ss:analysisRHS}. Consequently, $\widetilde{f^\sharp}\equiv0$, and thus $f^\sharp\equiv0$.

Thus, the method of the variation of constants yields 

\begin{equation*}    d(x)=\tilde{g}_1(x)\int_{x}^{+\infty}\frac{h_0(z)}{W[\tilde{g}_1,h_0](z)} \frac{1}{z}E(z)h_\e(z)dz + h_0(x)\int_1^x \frac{\tilde{g}_1(z)}{W[\tilde{g}_1,h_0](z)}\frac{1}{z}E(z)h_\e(z)dz,
\end{equation*}
with $\tilde{g}_1(x)=q_1(1/x)$ as in Section \ref{s:FrobRight}. Recall (see \eqref{Wrosk_cte_right}) that $W[\tilde{g}_1,h_0](z) = 1/z$ for all $z \in (1,+\infty).$

Then, the above expression reduces to 
\begin{equation}\label{dright}
d(x)=\tilde{g}_1(x)\int_{x}^{+\infty}h_0(z) E(z)h_\e(z)dz + h_0(x)\int_1^x \tilde{g}_1(z)  E(z)h_\e(z)dz:=d_1(x)+d_2(x).
\end{equation}
Notice that due to the boundary conditions for $h_\e(z)$ and the definition of $\varpi_0(z)$, the function $E(z)$ has suitable properties in order to yield $\lim_{x\to+\infty}x^{n-1}d(x)=0$.\\

To derive bounds on the difference $d$, it is necessary to establish suitable estimates for $E$. We rely on the following result.

\begin{lemma}\label{estimatesE} 
The following estimates hold for the function $E$:
\begin{enumerate}
    \item\label{eE1} For $1 < x \leq 2$,
    \[
        \int_x^2 |E(z)|\,dz \;\leq\; C\e \,\|\varpi_0'\|_{W^{1,\infty}} \,|x-1|^{-1},
    \]

    \item\label{eE2} 
    \[
        \int_2^{+\infty}|E(z)|\,dz \;\leq\; C\e \left(\|\varpi_0'\|_{L^1} +\|x\varpi_0''\|_{L^1} \right),
    \]
    
    \item\label{eE3} 
    \[
        \int_1^2 |z-1|\,|E(z)|\,dz \;\leq\; C\e \log(\e^{-1}) \,\|\varpi_0'\|_{W^{1,\infty}} ,
    \]
\end{enumerate}
where $C>0$ denotes a constant independent of $\e$, whose value may vary from line to line.
\end{lemma}

To prove Lemma~\ref{estimatesE}, we first need some bounds for \( \lambda_0 + c_R(z,0) \) and \( \lambda + c_R(z,\varepsilon) \), which appear in the denominators of~\eqref{def:E_R}. These bounds follow from Lemma~\ref{bounddenominator}. With this result at hand, we are now ready to prove it.

\begin{proof}[Proof of Lemma \ref{estimatesE}]
We split $E=E_1+E_2$ with
\begin{align}
E_1(z)&=\varpi_0'(z)\frac{\l-\lambda_0+ c_R(z,\e)-c_R(z,0)}{(\lambda_0+c_R(z,0))(\l+c_R(z,\e))} , \label{E1:RIGHT}\\
E_2(z)&=- \frac{\varpi_0'(z(1+\e/2))(1+\e/2)-\varpi_0'(z)}{\l+c_R(z,\e)}. \label{E2_RIGHT} 
\end{align}

Here, we recall that 
\begin{align*}\l-\lambda_0&=\e\lambda_1+\e^2\log^2(\e)\l_2(\e),\\
c_R(x,\e)-c_R(x,0)&=\e \left(-\dfrac{1}{4x^2} + \dfrac{1 }{2x^2 \left(1 + x^2\right)}+\dfrac{\log(1+x^2)}{2x^2}\right)+ \e^2 \pa^2_\e c_R(x,\e^\ast), \quad \text{for some }\e^\ast\in(0,\e),
\end{align*} and consequently (see \emph{Remark \ref{r:pa2cR}})
\begin{equation}\label{aux:comE1right}
|\l-\lambda_0|\leq C\e, \qquad \text{and} \qquad  |c_R(z,\e)-c_R(z,0)| \leq C\e, \qquad \text{for all } z\in(1,+\infty).
\end{equation}

\eqref{eE1}-\ref{estimatesE}. For $1<x\leq 2$, we can bound
\begin{align*}
    \int_{x}^2|E_1(z)| dz \leq C\e \|\varpi'_0\|_{L^\infty} \int_x^2 \frac{dz}{(z-1)^2}\leq C\e \|\varpi'_0\|_{L^\infty}|x-1|^{-1}, 
\end{align*}
using \eqref{aux:comE1right} and \eqref{bd1}-Lemma \ref{bounddenominator}. In addition,
\begin{align*}
\int_{x}^2 |E_2(z)|dz \leq C\e \left(\|\varpi''_0\|_{L^\infty}+\|\varpi'_0\|_{L^\infty}\right)\int_{x}^2 \frac{dz}{|z-1|}\leq  C\e \|\varpi_0'\|_{W^{1,\infty}}\left(1+|\log(|x-1|)|\right),
\end{align*}
where we have used again  \eqref{bd1}-Lemma \ref{bounddenominator} and 
\begin{align}\label{invarpi}
    \left|\varpi'_0\left(z\left(1+\frac{\e}{2}\right)\right)-\varpi_0'(z)\right|\leq \frac{|z|\e }{2}\int_0^1\left|\varpi_0''\left(z\left(1+\frac{\e s}{2}\right)\right)\right|ds\leq C\e \|\varpi_0''\|_{L^\infty},
\end{align}
for  $1<z\leq 2$.

\eqref{eE2}-\ref{estimatesE}. We have that 
\begin{align*}
    \int_2^{+\infty} |E_1(z)|dz \leq C\e \int_{2}^\infty \frac{|\varpi_0'(z)|dz}{|\lambda+c_R(z,\e)||\lambda_0+c_R(z,0)|},
\end{align*}
by using \eqref{aux:comE1right}. In addition, using \eqref{bd2}-Lemma \ref{bounddenominator} we find that
\begin{align*}
    \int_2^{+\infty} |E_1(z)|dz\leq C \e \|\varpi_0'\|_{L^1}.
\end{align*}

For $E_2$, using \eqref{bd2}-Lemma \ref{bounddenominator}, we have that
\begin{align*}
\int_{2}^\infty |E_2(z)|& \leq C\int_2^{+\infty} \left|\left(1+\frac{\e}{2}\right)\varpi'_0\left(z\left(1+\frac{\e}{2}\right)\right)-\varpi_0'(z)\right|dz\\
&\leq C\e \int_2^{+\infty} \left|\varpi'_0\left(z\left(1+\frac{\e}{2}\right)\right)\right|dz+C\int_2^{+\infty} \left|\varpi'_0\left(z\left(1+\frac{\e}{2}\right)\right)-\varpi'_0(z)\right|dz\\
&\leq C\e \|\varpi_0'\|_{L^1}+C\e \|x\varpi''_0\|_{L^1}. 
\end{align*}
In the last inequality we have used that
\begin{align*}
\int_2^{+\infty} \left|\varpi'_0\left(z\left(1+\frac{\e}{2}\right)\right)-\varpi'_0(z)\right|dz\leq \int_2^{+\infty}\frac{|z|\e}{2}\int_0^1 \left|\varpi_0''\left(z\left(1+\frac{\e s}{2}\right)\right)\right|ds\leq C\e \|x\varpi''_0\|_{L^1}. 
\end{align*}

\eqref{eE3}-\ref{estimatesE}. For $E_1$ we find that
\begin{align*}
    \int_{1}^2|z-1||E_1(z)|\leq C\e \|\varpi_0'\|_{L^\infty}\int_{1}^2\frac{|z-1|}{|\lambda_0+c_R(z,0)|}\frac{1}{|\lambda+c_R(z,\e)|}dz.
\end{align*}
By \eqref{bd1} and \eqref{bd3} in Lemma \ref{bounddenominator}, we learn that
\begin{align*}
\int_{1}^2\frac{|z-1|}{|\lambda_0+c_R(z,0)|}\frac{1}{|\lambda+c_R(z,\e)|}dz\leq C \int_{1}^{2}\frac{1}{(z-1)+c\e}dz\leq C\log(\e^{-1}).
\end{align*}
In addition
\begin{align*}
\int_{1}^2|z-1||E_2(z)|dz \leq C\e (\|\varpi_0'\|_{L^\infty}+\|\varpi_0''\|_{L^\infty}),
\end{align*}
by using \eqref{bd1}-Lemma \ref{bounddenominator} and \eqref{invarpi}.
\end{proof}

Having established Lemma \ref{estimatesE}, we now turn to the analysis of the difference between the two functions. 

\begin{lemma}\label{l:dRbound}
Let $\l$ be as in \eqref{def:ansatzlambda} with $\l_1$ satisfying \eqref{bound:l1}, $\l_2(\e)=O(1)$ in terms of the parameter $\e$ and $0<\e<\e_0$.
Then, we have 
\begin{equation*}
    \|x^nd\|_{L^\infty}\leq  C\e \log(\e^{-1})\|x^n h_\e\|_{L^\infty},
\end{equation*}
where the constant \( C \) depends on the parameters \( \varepsilon_0 \)
and on the norms  \( \|x^n h_0\|_{L^\infty} \), \( \|\varpi_0'\|_{W^{1,\infty}} \), \( \|\varpi_0'\|_{L^1} \), \( \|x\varpi_0''\|_{L^1} \), \( \|\tilde{g}_1\|_{W^{1,\infty}([1,2])} \) and \( \|x^{-n}\tilde{g}_1\|_{L^\infty([2,+\infty))}\).
\end{lemma}

\begin{remark}
  By the definition of $\varpi_0$, Lemma \ref{l:prophr} and \ref{l:decg1} the norms, \( \|\varpi_0'\|_{W^{1,\infty}} \), \( \|\varpi_0'\|_{L^1} \), \( \|x\varpi_0''\|_{L^1} \), \( \|x^n h_0\|_{L^\infty} \),  \( \|\tilde{g}_1\|_{W^{1,\infty}([1,2])} \) and \( \|x^{-n}\tilde{g}_1\|_{L^\infty([2,+\infty))} \) are bounded by a universal constant. 
\end{remark}

\begin{remark}\label{continuidadh}
We emphasize that, with minor modifications, the proof of Lemma~\ref{l:dRbound} also yields the continuity of \( h_R \) with respect to \( \lambda_2 \), assumed bounded, \( |\lambda_2| \le M \), for sufficiently small \( \varepsilon \).
\end{remark}

\begin{proof}
We bound $d_1$ and $d_2$ in \eqref{dright} separately. 
We begin by studying $d_1$.

In order to estimate the \( L^\infty \) norm over the full interval \(  (1, +\infty) \), we proceed by decomposing the domain into two parts:
\begin{itemize}
    \item [a)]the bounded interval \( (1, 2] \),
    \item [b)]the unbounded interval \( (2, +\infty) \). 

\end{itemize}
%a) the bounded interval \( [1, 2] \),
%b) the unbounded interval \( [2, +\infty) \). 
This decomposition allows us to treat separately the behavior of the function near \( z = 1^+ \) in the bounded region, and near infinity in the unbounded one.\\

a) $x\in(1,2]$. Since $\|x^nd\|_{L^\infty([1,2])}\sim \|d\|_{L^\infty([1,2])}$ we just bound the second norm. We split it as follows:
\begin{equation*}
d_{1}(x)=\underbrace{\tilde{g}_1(x)\int_{x}^{2}h_0(z)h_\e(z)E(z)dz}_{d_{11}(x)}+\underbrace{\tilde{g}_1(x)\int_{2}^{+\infty}h_0(z)h_\e(z)E(z)dz}_{d_{12}(x)}.
\end{equation*}

For $d_{11}$ we have that
\begin{align*}
|d_{11}(x)|\leq \|h_0\|_{L^\infty}\|\tilde{g}'_1\|_{L^\infty([1,2])}|x-1| \|h_\e\|_{L^\infty}\int_{x}^2 |E(z)|dz.
\end{align*}
Applying \eqref{eE1}-Lemma \ref{estimatesE} we find that
\begin{align*}
    \|d_{11}\|_{L^\infty([1,2])}\leq C\e  \|h_0\|_{L^\infty}\|\tilde{g}'_1\|_{L^\infty([1,2])}\|h_\e\|_{L^\infty}\|\varpi_0'\|_{W^{1,\infty}}.
\end{align*}

For $d_{12}$ we have that
\begin{align*}
    \|d_{12}\|_{L^\infty([1,2])}\leq & \|\tilde{g}_1\|_{L^\infty([1,2])}\int_{2}^{+\infty} |h_0(z)| |h_\e(z)||E(z)|dz\\
    \leq & \|\tilde{g}_1\|_{L^\infty([1,2])}\| h_0\|_{L^\infty}\|h_\e\|_{L^\infty}\|E(z)\|_{L^1([2,+\infty))}\\
    \leq &C\e\left(\|\varpi_0'\|_{L^1}+\|x\varpi_0''\|_{L^1}\right) \|\tilde{g}_1\|_{L^\infty([1,2])}\| h_0\|_{L^\infty}\|h_\e\|_{L^\infty},
\end{align*}
because of \eqref{eE2}-Lemma \ref{estimatesE}.

We have then shown that $
\|d_1\|_{L^\infty([1,2])}\leq C\e \|h_\e\|_{L^\infty}.
$

For $d_2$ we have that 
\begin{align}\label{ddos}
|d_2(x)|\leq &|h_0(x)|\int_1^x |\tilde{g}_1(z)||h_\e(z)|E(z)|dz\leq\|h_0\|_{L^\infty}\|h_\e\|_{L^\infty}\int_1^2 |\tilde{g}_1(z)||E(z)|dz\\
\leq &\|h_0\|_{L^\infty}\|h_\e\|_{L^\infty}\|\tilde{g}'_1\|_{L^\infty([1,2])}\int_1^2 |z-1||E(z)|dz\nonumber\\
\leq &C\e\log(\e^{-1}) \|h_0\|_{L^\infty}\|h_\e\|_{L^\infty}\|\tilde{g}'_1\|_{L^\infty([1,2])}\|\varpi'_0\|_{W^{1,\infty}}\nonumber,
\end{align}
where we have used \eqref{eE3}-Lemma \ref{estimatesE}. This yields $\|d_2\|_{L^\infty([1,2])}\leq C \e\log(\e^{-1})\|h_\e\|_{L^\infty},$ and then we can conclude $
\|d\|_{L^\infty([1,2])}\leq C \e\log(\e^{-1})\|h_\e\|_{L^\infty}.
$

b) $x\in (2,+\infty)$. We first consider again $d_1$. We have that
\begin{align*}
|d_1(x)|\leq |\tilde{g}_1(x) x^{-n}|x^n\int_{x}^\infty|z^n h_0(z)||z^n h_\e(z)|z^{-2n}|E(z)|dz.
\end{align*}
Let us introduce $H_0(x):=x^nh_0(x)$ and $H_\e(x):=x^nh_\e(x)$. Rescaling the integral yields
\begin{align*}
    \int_{x}^\infty|H_0(z)|| H_\e(z)|z^{-2n}|E(z)|dz&=x^{-2n+1}\int_{1}^\infty |H_0(xz)||H_\e(xz)| z^{-2n}E(xz) dz\\&\leq x^{-2n+1} \|H_0\|_{L^\infty([2,+\infty))}\|H_\e\|_{L^\infty([2,+\infty))}
    \int_{1}^\infty |E(xz)|dz\\
    &\leq x^{-2n}\|H_0\|_{L^\infty([2,+\infty))}\|H_\e\|_{L^\infty([2,+\infty))}\|E\|_{L^1([2,+\infty))}.
\end{align*}
Therefore, by \eqref{eE2}-Lemma \ref{estimatesE} we have that
\begin{align*}
    \|x^nd_1\|_{L^\infty([2,+\infty))}\leq C\e (\|\varpi'_0\|_{L^1}+\|x\varpi''_0\|_{L^1})\|x^nh_0\|_{L^\infty([2,+\infty))}\|x^{-n}\tilde{g}_1\|_{L^\infty([2,+\infty))}\|x^n h_\e\|_{L^\infty([2,+\infty))}.
\end{align*}
To bound $x^nd_2(x)$ we proceed as follows
\begin{align*}
    |x^nd_2(x)|&\leq |x^n h_0(x)|\int_{1}^x |\tilde{g}_1(z)||h_\e(z)||E(z)|dz\\
    &\leq |x^n h_0(x)|\int_{1}^2 |\tilde{g}_1(z)||h_\e(z)||E(z)|dz+ |x^nh_0(x)|\int_{2}^x|\tilde{g}_1(z)||h_\e(z)||E(z)|dz.
\end{align*}
The integral $\int_{1}^2 |\tilde{g}_1(z)||h_\e(z)||E(z)|dz$ has already been bounded in \eqref{ddos}. The integral 
\begin{align*}
\int_{2}^x|\tilde{g}_1(z)||h_\e(z)||E(z)|dz
\end{align*}
can be estimated by
\begin{align*}
    \leq \|x^{-n}\tilde{g}_1\|_{L^\infty}\|x^n h_\e\|_{L^\infty}\int_2^{+\infty}|E(z)|dz.
\end{align*}
Thus, by by \eqref{eE2}-Lemma \ref{estimatesE} we have $\|x^nd_2\|_{L^\infty([2,+\infty))}\leq C\e\log(\e^{-1})\|x^nh_\e\|_{L^\infty([2,+\infty))}$ and then $$\|x^nd\|_{L^\infty([2,+\infty))}\leq C\e\log(\e^{-1})\|x^nh_\e\|_{L^\infty([2,+\infty)}.$$ 

Putting together the cases a) and b) we have proved that
\begin{align*}
\|x^n d\|_{L^\infty}\leq C\e\log(\e^{-1})\|x^nh_\e\|_{L^\infty}.
\end{align*}
\end{proof}

\begin{corollary}\label{c:rightd}
Let $\l$ be as in \eqref{def:ansatzlambda} with $\l_1$ satisfying \eqref{bound:l1}, $\l_2(\e)=O(1)$ in terms of the parameter $\e$ and $0<\e<\e_0$.
Then, we have 
$$\|x^n d \|_{L^\infty}\leq C\e\log(\e^{-1}),$$
where the constant \( C \) depends on the parameters \( \varepsilon_0 \)
 and on the norms \( \|h_0\|_{L^\infty} \), \( \|x^n h_0\|_{L^\infty} \), \( \|\varpi_0'\|_{W^{1,\infty}} \), \( \|\varpi_0'\|_{L^1} \), \( \|x\varpi_0''\|_{L^1} \), \( \|\tilde{g}_1\|_{W^{1,\infty}([1,2])} \) and \( \|x^{-n}\tilde{g}_1\|_{L^\infty([2,+\infty))} \).
\end{corollary}
\begin{proof} Thanks to Lemma \ref{l:dRbound} and the fact that \( h_\varepsilon = h_0 + d \), we obtain the conclusion of the result for sufficiently small \( \varepsilon \).
\end{proof}

The next lemma exploits the bound on $\|x^nd\|_{L^\infty}$ derived above to estimate an integral term that plays a key role in the forthcoming analysis.

\begin{lemma}\label{l:Dright}
Let $\l$ be as in \eqref{def:ansatzlambda} with $\l_1$ satisfying \eqref{bound:l1}, $\l_2(\e)=O(1)$ in terms of the parameter $\e$ and $0<\e<\e_0$. Then, the integral 
\begin{equation*}
\int_1^{+\infty} \varpi_0'(x(1+\e/2)) x^{-n} \frac{d(x)}{\lambda+c_R(x,\e)}dx=M[h_0,\lambda_1] + R[h_0,\lambda_1; h_\e,\lambda_2(\e)],
\end{equation*}
where $M[h_0,\lambda_1], R[h_0,\lambda_1;\, h_\e,\lambda_2(\e)] \in \mathbb{R}$ and
\begin{align*}
    M[h_0,\lambda_1]&=O(\e\log^2(\e)),\\
    R[h_0,\lambda_1;\, h_\e,\lambda_2(\e)]&=o(\e\log^2(\e)).
\end{align*}
In addition, $R[h_0,\lambda_1;\, h_\e, \lambda_2(\e)]$ depends continuously on $h_\e$ and $\lambda_2(\e)$. Here, all the estimates may depend on the parameter $\e_0$ and on the norms \( \|h_0\|_{L^\infty} \), \( \|x^n h_0\|_{L^\infty} \), \( \|\varpi_0'\|_{W^{1,\infty}} \), \( \|\varpi_0'\|_{L^1} \), \( \|x\varpi_0''\|_{L^1} \), \( \|\tilde{g}_1\|_{W^{1,\infty}([1,2])} \) and \( \|x^{-n}\tilde{g}_1\|_{L^\infty([2,+\infty))} \), which have been shown to be bounded by a universal constant.
\end{lemma}
\begin{remark}
Notice that, in Lemma \ref{l:Dright}, the term $M$ does not depend on $\lambda_2(\e)$. All the dependence on $\lambda_2(\e)$ is contained in the term $R$. This dependence occurs in two ways: explicitly, and implicitly through $h_\e$.  
\end{remark}
\begin{proof} We will say that a term is an $M$-term if it depends only on $\lambda_1$, or $h_0$, that is, if it does not depend on either $h_\e$ or $\lambda_2(\e)$. We will say that a term is a $R$-term if it does depend on either $h_\e$ or $\lambda_2(\e)$ and the dependence is continuous. Here, if the domain of a norm is not explicitly specified, we shall understand it as $[1,+\infty).$

The continuity of $R$ with respect to $h_\e$ and $\lambda_2(\e)$ is easy to check since every function that arises in the analysis will be continuous with respect to these variables. We focus on proving the estimates $O(\e\log^2(\e))$ and $o(\e\log^2(\e)).$

Consider the factor 
\begin{align}\label{sR}
    \frac{1}{\lambda+c_R(x,\e)}&=\frac{1}{\lambda_0+\e \lambda_1 + c_R(x,\e)+\e^2\log^2(\e)\lambda_2(\e)}\nonumber\\
    &=\frac{1}{\lambda_0+\e \lambda_1 + c_R(x,\e)}+\cR[\lambda_2(\e)](\e,x),
\end{align}
with
\begin{align}\label{Rdefi}
\cR[\lambda_2(\e)](\e,x):=-\frac{\e^2\log^2(\e)\lambda_2(\e)}{(\lambda_0+\e \lambda_1 + c_R(x,\e))(\lambda+c_R(x,\e))}.
\end{align}
From Lemma \ref{bounddenominator} we have that
\begin{align}\label{cR}
    |\cR[\lambda_2(\e)](\e,x)|\leq C\e^2\log^2(\e)\frac{1}{|c(x-1)+c_1\e|^2},\quad \text{for $1<x\leq 2$}.
\end{align}
Notice that we have the same bound for $\lambda_0+\e\lambda_1+c_R(x,\e)$ than for $\lambda+c_R(x,\e)$.
Thus
\begin{align}
    \label{L1R}\|\cR[\lambda_2(\e)](\e,x)\|_{L^1([1,2])}&\leq  C\e\log^2(\e),
\end{align}
and 
\begin{align}
%\label{mala}\|\cR[\lambda_2(\e)](\e,x)\|_{L^\infty([1,2])}&\leq  C\log^2(\e),\\
%\label{buena}\|(x-1)\cR[\lambda_2(\e)](\e,x)\|_{L^\infty([1,2])}&\leq  C\e\log^2(\e),\\
\label{R1}\int_{x}^2|\cR[\lambda_2](\e)(x,\e)|dx&\leq C\e^2\log^2(\e)|x-1|^{-1},\quad \text{for $1<x\leq 2$,}\\
\label{LinfR}|\cR[\lambda_2(\e)](\e,x)|&\leq C \e^2\log^2(\e),\quad \text{for $2<x<+\infty$,}\\
\label{R2}\|(x-1)\cR[\lambda_2(\e)](\e,x)\|_{L^1([1,2])}&\leq  C\e^2\log^3(\e^{-1}).
\end{align}

Using the relation \( h_\varepsilon = h_0 + d \), we can rewrite \( d \) in \eqref{dright} as follows:
\begin{align*}
    d(x)&=\underbrace{\tilde{g}_1(x)\int_x^\infty h_0(z)h_0(z)E(z)dz+h_0(x)\int_1^x \tilde{g}_1(z)h_0(z)E(z)dz}_{d_0^{\lambda_2}}\\
    &+\underbrace{\tilde{g}_1(x)\int_x^\infty h_0(z)d(z)E(z)dz + h_0(x)\int_{1}^x\tilde{g}_1(z)d(z)E(z)dz}_{d^{\lambda_2}}\\
    &=:d_0^{\lambda_2}(x)+d^{\lambda_2}(x),
\end{align*}
and we find that
\begin{align*}
&\int_1^{+\infty} \varpi_0'(x(1+\e/2)) x^{-n} \frac{d(x)}{\lambda+c_R(x,\e)}dx\\
&=\int_1^{+\infty} \varpi_0'(x(1+\e/2)) x^{-n}\frac{d_0^{\lambda_2}(x)}{\lambda_0+\e\lambda_1+c_R(x,\e)}dx\\
&+\int_1^{+\infty} \varpi_0'(x(1+\e/2)) x^{-n}\frac{d^{\lambda_2}(x)}{\lambda_0+\e\lambda_1+c_R(x,\e)}dx\\
&+\int_1^{+\infty} \varpi_0'(x(1+\e/2)) x^{-n}\cR[\lambda_2(\e)](\e,x)d(x)dx.
\end{align*}

Notice that $d^{\lambda_2}$ differs from $d$ (see \eqref{dright}) in that $d$ arises in $d^{\lambda_2}$ instead of $h_\e$. Then we apply Lemma~\ref{l:dRbound} to bound \( d^{\lambda_2} \) (by replacing \( h_\varepsilon \) with \( d \) in that lemma), and we obtain that
\begin{align*}
    &\left|\int_1^{+\infty} \varpi_0'(x(1+\e/2)) x^{-n}\frac{d^{\lambda_2}(x)}{\lambda_0+\e\lambda_1+c_R(x,\e)}dx\right|\\\leq& C   \|d^{\lambda_2}\|_{L^\infty}\left(\|\varpi_0'\|_{L^\infty}\int_1^2 \frac{dx}{(x-1)+c\e}+C\int_2^{+\infty}|\varpi_0'(x(1+\e/2))x^{-n}|dx\right)\\
    \leq& C\e\log(\e^{-1})\|d\|_{L^\infty}\log(\e^{-1}),
\end{align*}
where the constant $C$ depends only on the controlled quantities. Here we have applied Lemma \ref{bounddenominator}. Using  Corollary \ref{c:rightd} yields
\begin{align*}
    &\left|\int_1^{+\infty} \varpi_0'(x(1+\e/2)) x^{-n}\frac{d^{\lambda_2}(x)}{\lambda_0+\e\lambda_1+c_R(x,\e)}dx\right|\leq C\e^2\log^3(\e^{-1})=o\left(\e\log^2(\e)\right).
\end{align*}
    Therefore
\begin{align*}
    \int_1^{+\infty} \varpi_0'(x(1+\e/2)) x^{-n}\frac{d^{\lambda_2}(x)}{\lambda_0+\e\lambda_1+c_R(x,\e)}dx,
\end{align*}
is a $R$-term. In addition,
\begin{align*}
&\int_1^{+\infty} \varpi_0'(x(1+\e/2)) x^{-n}\cR[\lambda_2(\e)](\e,x)d(x)dx\\
&\leq C\|d\|_{L^\infty}\left(\|\varpi'_0x^{-n}\|_{L^\infty}\|\cR\|_{L^1([1,2])}+C\|\cR\|_{L^\infty([2,+\infty))}\|\varpi_0'x^{-n}\|_{L^1}\right)\\
&\leq C\e^2\log^3(\e^{-1})=o(\e\log^2(\e)),
\end{align*}
where we have used Corollary \ref{c:rightd}, \eqref{L1R} and \eqref{LinfR}. Therefore, we have obtained that
\begin{align*}
&\int_1^{+\infty} \varpi_0'(x(1+\e/2)) x^{-n} \frac{d(x)}{\lambda+c_R(x,\e)}dx\\
=&\int_1^{+\infty} \varpi_0'(x(1+\e/2)) x^{-n}\frac{d_0^{\lambda_2}(x)}{\lambda_0+\e\lambda_1+c_R(x,\e)}dx\quad+\quad R\text{-terms}.
\end{align*}
With the remaining term
\begin{align*}
    \int_1^{+\infty} \varpi_0'(x(1+\e/2)) x^{-n}\frac{d_0^{\lambda_2}(x)}{\lambda_0+\e\lambda_1+c_R(x,\e)}dx,
\end{align*}
we still have to do some work since $d_0^{\lambda_2}$ still depends on $\lambda_2(\e)$ through $E$. Indeed,
\begin{align*}
E(x)&=-\frac{\varpi'_0(x(1+\e/2))(1+\e/2)}{\lambda+c_R(x,\e)}+\frac{\varpi'_0(x)}{\lambda_0+c_R(x,0)}\\
&= -\frac{\varpi'_0(x(1+\e/2))(1+\e/2)}{\lambda_0+\lambda_1\e+c_R(x,\e)}+\frac{\varpi'_0(x)}{\lambda_0+c_R(x,0)}\\
&\quad -  \varpi'_0(x(1+\e/2))(1+\e/2)\mathcal{R}[\lambda_2(\e)](x,\e)\\
&=: E^0(x)+E^{\lambda_2}(x).
\end{align*}
Here we have used again the splitting \eqref{sR}-\eqref{Rdefi}. We emphasize that $E^0$ does not depend on $\lambda_2(\e)$ anymore. Thus,
it remains to prove that, on the one hand
\begin{align}\label{Mterm}
    \int_1^{+\infty} \varpi_0'(x(1+\e/2)) x^{-n} \frac{d^0(x)}{\lambda_0+\lambda_1\e+c_R(x,\e)}dx
\end{align}
with
\begin{align*}
   d^0(x):=\tilde{g}_1(x)\int_x^{+\infty} h_0(z)h_0(z)E^0(z)dz+h_0(x)\int_1^x \tilde{g}_1(z)h_0(z)E^0(z)dz,
\end{align*}
is a $M$-term. On the other hand, that the term 
\begin{align}\label{Rterm}
    \int_1^{+\infty} \varpi_0'(x(1+\e/2)) x^{-n} \frac{d^\cR(x)}{\lambda_0+\lambda_1\e+c_R(x,\e)}dx,
\end{align}
with
\begin{align*}
   d^\cR(x):=\tilde{g}_1(x)\int_x^{+\infty} h_0(z)h_0(z)E^{\lambda_2}(z)dz+h_0(x)\int_1^x \tilde{g}_1(z)h_0(z)E^{\lambda_2}(z)dz, 
\end{align*}
is a $R$-term.

The term in \eqref{Mterm} is easy to deal with. We first notice that we have for $\lambda_0+\lambda_1\e+c_R(x,\e)$ an analogous result that Lemma \ref{bounddenominator} for $\lambda+c_R(x,\e)$. Therefore, we have an analogous result that Lemma \ref{estimatesE} for $E^0$ that for $E$. This yields, from Lemma \ref{l:dRbound}, the estimates
\begin{align*}
    \left|\int_1^{+\infty} \varpi_0'(x(1+\e/2)) x^{-n} \frac{d^0(x)}{\lambda_0+\lambda_1\e+c_R(x,\e)}dx\right|\leq C\e\log^2(\e),
\end{align*}
which proves that \eqref{Mterm} is a $M$-term. 

Finally, to verify that \eqref{Rterm} is indeed an $R$-term, we combine \eqref{R1}, \eqref{LinfR}, and \eqref{R2} to obtain the following result for $E^{\lambda_2}$. This result is analogous to Lemma~\ref{estimatesE}, but features a stronger decay in $\varepsilon$; in fact, we gain a factor of $\varepsilon \log^2(\varepsilon)$.
\begin{lemma}$E^{\lambda_2}$ satisfies:
\begin{enumerate}
       \item\label{EL1} For $1 < x \leq 2$,
    \[
        \int_x^2 |E^{\lambda_2}(z)|\,dz \;\leq\; C\e^2\log^2(\e)  \,|x-1|^{-1},
    \]

    \item\label{EL2} 
    \[
        \int_2^{+\infty}|E^{\lambda_2}(z)|\,dz \;\leq\; C\e^2 \log^2(\e),
    \]
    
    \item\label{EL3} 
    \[
        \int_1^2 |z-1|\,|E^{\lambda_2}(z)|\,dz \;\leq\; C\e^2 \log^3(\e^{-1}).
    \]
\end{enumerate}
\end{lemma}
Therefore we can obtain for $d^{\cR}$, the estimate
\begin{align*}
\|d^{\cR}\|_{L^\infty}\leq C\e^2\log^3(\e^{-1}),
\end{align*}
proceeding as in Lemma \ref{l:dRbound}. This finally yields the bound
\begin{align}\label{intRterm}
  \left|  \int_1^{+\infty} \varpi_0'(x(1+\e/2)) x^{-n} \frac{d^\cR(x)}{\lambda_0+\lambda_1\e+c_R(x,\e)}dx\right|\leq C\e^2\log^4(\e),
\end{align}
which shows that \eqref{Rterm} is a $R$-term. 
\end{proof}

\begin{comment}

\begin{lemma}\label{l:rightrough}
Let $\l$ be as in \eqref{def:ansatzlambda} with $\l_1$ satisfying \eqref{bound:l1}, $\l_2(\e)=O(1)$ in terms of the parameter $\e$ and $0<\e<\e_0$.  Then, we have 
\[
\left|\int_1^{\infty} \varpi_0'(x(1+\e/2))(1+\e/2) x^{-n} \frac{d(x)}{\l+c_R(x,\e)}dx\right|\leq C \e \log^2(\e),
\]
where $C$ \textcolor{blue}{(repasar esta dependencia con lambdas)} depends on the parameters \( \varepsilon_0 \)
and on the norms \( \|h_0\|_{L^\infty} \), \( \|x^n h_0\|_{L^\infty} \), \( \|\varpi_0'\|_{W^{1,\infty}} \), \( \|\varpi_0'\|_{L^1} \), \( \|x\varpi_0''\|_{L^1} \), \( \|\tilde{g}_1\|_{W^{1,\infty}([1,2])} \) and \( \|x^{-n}\tilde{g}_1\|_{L^\infty([2,+\infty))}\).
\end{lemma}
\begin{proof}
We just compute
\begin{align*}
&\left|\int_1^{\infty} \varpi_0'(x(1+\e/2))(1+\e/2) x^{-n} \frac{d(x)}{\l+c_R(x,\e)}dx\right|\\&\leq \|d\|_{L^\infty} \int_{1}^\infty\frac{|\varpi'_0(x(1+\e/2))|}{|\lambda+c_R(x,\e)|}dx \\&\leq \|d\|_{L^\infty} \left(\|\varpi'_0\|_{L^\infty}\log(\e^{-1})+\|\varpi'_0\|_{L^1}\right),
\end{align*}
and apply Corollary \ref{c:rightd}
\end{proof}
\end{comment}
\subsection{The difference in the left}

In this section we perform the analogous study to the one carried out in the domain $(1,+\infty)$, but now restricted to the interval $(0,1)$. The problem is essentially the same, and therefore we will only state the main equations and results without entering into detailed proofs.  

We consider the difference between the solution of \eqref{eq:odeleft}--\eqref{eq:bcodeleft} with $\e=0$ and the solution of \eqref{eqfrob:leftgeneral} with $\e>0$, namely
\[
d(x):=h_L(x)-\mathring{h}_L(x), \qquad 0<x<1.
\]
As before, we introduce the simplified notation
\[
h_\e(x):=h_L(x), \qquad h_0(x):=\mathring{h}_L(x), \qquad d(x):=h_\e(x)-h_0(x).
\]

Recalling the expansion (see Section \ref{ss:analysisC})
\[
\l+c_L(x,\e)=\frac{\log(2)}{2}-\frac{\log(1+x^2)}{2x^2}+\e\left(\l_1+\frac{1}{4}+\frac{1}{2(1+x^2)}-\frac{\log(1+x^2)}{2x^2}\right)+O(\e^2),
\]
we find that $d$ satisfies
\begin{align*}
d''(x) + \frac{1}{x} d'(x)-\frac{1}{x^2}\left(n^2-\frac{x\varpi_0'(x)}{\lambda_0+c_L(x,0)}\right)d(x)&=\frac{1}{x}E(x)h_\e(x), \qquad 0 < x< 1,\\
\lim_{x\to 0^+}x^{1-n}d(x)&=0,\\
\lim_{x\to 1^-}d(x)&=0,
\end{align*}
with
\[
E(x):=-\left(\frac{\varpi_0'(x(1-\e/2))(1-\e/2)}{\l+c_L(x,\e)}-\frac{\varpi_0'(x)}{\lambda_0+c_L(x,0)}\right).
\]

As in the case of the interval $(1,+\infty)$, the homogeneous problem admits only the trivial solution. Therefore, by variation of constants we obtain the representation
\[
d(x)=g_1(x)\int_{0}^{x} h_0(z)  E(z)h_\e(z)\,dz + h_0(x)\int_x^1  g_1(z) E(z)h_\e(z)\,dz,
\]
where $g_1$ is the solution introduced in Section \ref{s:FrobLeft}.\\

In what follows, we shall only collect the results that we need, without providing proofs. Here, if the domain of a norm is not explicitly specified, we shall understand it as $[0,1].$

\begin{lemma}
Let $\l$ be as in \eqref{def:ansatzlambda} with $\l_1$ satisfying \eqref{bound:l1}, $\l_2(\e)=O(1)$ in terms of the parameter $\e$ and $0<\e<\e_0$.
Then, we have 
\begin{equation*}
    \|x^{-n}d\|_{L^\infty}\leq  C\e \log(\e^{-1})\|x^{-n} h_\e\|_{L^\infty},
\end{equation*}
where the constant \( C \) depends on the parameters \( \varepsilon_0 \)
and on the norms 
 \( \|x^{-n} h_0\|_{L^\infty} \), \( \|\varpi_0'\|_{W^{1,\infty}} \), \( \|g_1\|_{W^{1,\infty}([1/2,1])} \) and \( \|x^{n}g_1\|_{L^\infty}\) (All these norms are bounded by a universal constant).
\end{lemma}
Since $h_\e = h_0 + d$, for $0<\e<\e_0$ small enough, we obtain the following result.
\begin{corollary}
Let $\l$ be as in \eqref{def:ansatzlambda} with $\l_1$ satisfying \eqref{bound:l1}, $\l_2(\e)=O(1)$ in terms of the parameter $\e$ and $0<\e<\e_0$.
Then, we have 
$$\|x^{-n} d \|_{L^\infty}\leq C\e\log(\e^{-1}),$$
where the constant \( C \) depends on the parameters \( \varepsilon_0 \)
 and on the norms  \( \|x^{-n} h_0\|_{L^\infty} \), \( \|\varpi_0'\|_{W^{1,\infty}} \), \( \|g_1\|_{W^{1,\infty}([1/2,1)} \) and \( \|x^{n}g_1\|_{L^\infty}\) .
\end{corollary}

Moreover, the bound on $\|x^{-n}d\|_{L^\infty}$ yields an estimate for a key integral term needed later.
Since the argument follows similar ideas, we state the result without proof.

\begin{lemma}\label{l:Dleft}
Let $\l$ be as in \eqref{def:ansatzlambda} with $\l_1$ satisfying \eqref{bound:l1}, $\l_2(\e)=O(1)$ in terms of the parameter $\e$ and $0<\e<\e_0$. Then, the integral 
\begin{equation*}
\int_0^{1} \varpi_0'(x(1-\e/2)) x^{n} \frac{d(x)}{\lambda+c_L(x,\e)}dx=M[h_0,\lambda_1] + R[h_0,\lambda_1; h_\e,\lambda_2(\e)],
\end{equation*}
where $M[h_0,\lambda_1], R[h_0,\lambda_1;\, h_\e,\lambda_2(\e)] \in \mathbb{R}$ and
\begin{align*}    M[h_0,\lambda_1]&=O(\e\log^2(\e)),\\
    R[h_0,\lambda_1;\, h_\e,\lambda_2(\e)]&=o(\e\log^2(\e)).
\end{align*}
In addition, $R[h_0,\lambda_1;\, h_\e, \lambda_2(\e)]$ is continuous with respect  to $h_\e$ and $\lambda_2(\e)$. Here, all the terms can depend on the parameters \( \varepsilon_0 \)
and on the norms  \( \|x^{-n} h_0\|_{L^\infty} \), \( \|\varpi_0'\|_{W^{1,\infty}} \), \( \|g_1\|_{W^{1,\infty}([1/2,1])} \) and \( \|x^{n}g_1\|_{L^\infty}\) .
\end{lemma}

\section{Derivation of the equation for $\lambda$}
Here, we reduce the problem of finding solutions to \eqref{inteqleft_final}-\eqref{inteqright_final} to a single equation for $\lambda$. 

\subsection{Solving the original system of integral equations.}\label{argument}
We must not forget that our ultimate goal is to solve the coupled system of integral equations formed by \eqref{inteqleft_final} and \eqref{inteqright_final}.

Let us consider $(h_L, h_R)$, a solution of the ODEs \eqref{eqfrob:leftgeneral}-\eqref{eqfrob:rightgeneral} given in Lemma \ref{l:prop}. We will look for  solutions of   \eqref{inteqleft_final} and \eqref{inteqright_final} with the form
\begin{align*}
    H_L(x)=Ah_L(x),\quad H_R(x)=Bh_R(x),
\end{align*}
where $A,B\in \R$ are two free parameters. Here, we recall that, since we are looking for a non-trivial solution, we need to impose that 
\begin{equation}\label{nontrivialAB}
(A, B) \neq (0, 0). 
\end{equation}

Next we define the function $N_L:(0,1)\to\mathbb{R}$ given by
\begin{align*}
    &N_L(x):=A h_L(x)-\frac{(1-\e/2)}{2n}\int_0^{x}\varpi_0'(s(1-\e/2)) \left(\frac{s}{x}\right)^{n}  \frac{ A h_L(s)}{\l+c_L(s,\e)}ds\\ 
&-\frac{(1-\e/2)}{2n}\int_{x}^{1}\varpi_0'(s(1-\e/2)) \left(\frac{s}{x}\right)^{-n}\frac{A h_L(s)}{\l+c_L(s,\e)}ds\\
&-\frac{(1+\e/2)}{2n}\left(\frac{1+\e/2}{1-\e/2}\right)^{-n}\int_{1}^{+\infty}\varpi_0'(s(1+\e/2))\left(\frac{s}{x}\right)^{-n} \frac{B h_R(s)}{\l+c_R(s,\e)}ds,
\end{align*}
and which satisfies the system, given that \(h_L\) and \(h_R\) satisfy \eqref{eqfrob:odeleft} and \eqref{eqfrob:oderight}, respectively,
\begin{align}\label{eq:ODE-NL}
N_L''(x)+\frac{1}{x}N_L'(x)-\frac{n^2}{x^2}N_L(x)=&0,\\
\lim_{x\to 0^+} x^{1-n} N_L(x)=&0.\nonumber
\end{align}

Similarly, we also define the function $N_R:(1,+\infty)\rightarrow \mathbb{R}$ given by
\begin{align*}
&N_R(x):=B h_R(x)-\frac{(1-\e/2)}{2n}\left(\frac{1-\e/2}{1+\e/2}\right)^{n}\int_0^1 \varpi_0'(s(1-\e/2))\left(\frac{s}{x}\right)^{n} \frac{A h_L(s)}{\l+c_L(s,\e)}ds\\
 &-\frac{(1+\e/2)}{2n}\int_1^x \varpi_0'(s(1+\e/2))\left(\frac{s}{x}\right)^{n}\frac{B h_R(s)}{\l+c_R(s,\e)}ds\\
 &-\frac{(1+\e/2)}{2n}\int_x^\infty \varpi_0'(s(1+\e/2)) \left(\frac{s}{x}\right)^{-n} \frac{B h_R(s)}{\l+c_R(s,\e)}ds,
\end{align*}
that satisfies
\begin{align}\label{eq:ODE-NR}
N_R''(x) + \frac{1}{x} N_R'(x) -\frac{n^2}{x^2}N _R(x)=0,\\ 
\lim_{x\to +\infty} x^{n-1}N_R(x)=0\nonumber.
\end{align}

\begin{proof}[Proof of \eqref{eq:ODE-NL}-\eqref{eq:ODE-NR}] It is sufficient to take two derivatives of the equations and use  \eqref{eqfrob:odeleft} and \eqref{eqfrob:oderight}. For the boundary conditions, it suffices to use the behavior of $h_L$ at zero and of $h_R$ at infinity.
\end{proof}

The general solution of $f''(x)+f'(x)/x-(n^2/x^2) f(x)=0$ is given by
\[
c_1 x^n +c_2x^{-n}.
\]
Now, by imposing the boundary conditions, we have
\[
N_L(x)=c_1 x^n, \qquad N_R(x)=c_2 x^{-n}.
\]
Therefore, if we impose $N_L(1^-)=0$ then $N_L(x)=0$ for all $(0,1)$.  Similarly, the same argument applies in the interval $(1,+\infty)$ simply by imposing $N_R(1^+)=0$.

Thus in order to solve \eqref{inteqleft_final}-\eqref{inteqright_final}, we just need to solve 
\begin{equation}\label{cond:NLandNR}
N_L(1^{-})=0=N_R(1^+),    
\end{equation}
for some $\lambda\in\mathbb{R}$.
Taking limits, we find
\begin{multline*}
N_L(1^-)=A h_L(1)-\frac{(1-\e/2)}{2n}\int_0^{1}\varpi_0'(s(1-\e/2)) s^{n}  \frac{ A h_L(s)}{\l+c_L(s,\e)}ds\\
-\frac{(1+\e/2)}{2n}\left(\frac{1+\e/2}{1-\e/2}\right)^{-n}\int_{1}^{+\infty}\varpi_0'(s(1+\e/2))s^{-n} \frac{B h_R(s)}{\l+c_R(s,\e)}ds,
\end{multline*}
\begin{multline*}
N_R(1^+)= Bh_R(1)-\frac{(1-\e/2)}{2n}\left(\frac{1-\e/2}{1+\e/2}\right)^{n}\int_0^1 \varpi_0'(s(1-\e/2))s^n \frac{Ah_L(s)}{\l+c_L(s,\e)}ds\\
-\frac{(1+\e/2)}{2n}\int_1^\infty \varpi_0'(s(1+\e/2)) s^{-n} \frac{Bh_R(s)}{\l+c_R(s,\e)}ds,
\end{multline*}
and consequently, since $h_R(1)=1=h_L(1)$, we can write $\eqref{cond:NLandNR}$ as
\begin{equation}\label{sys:matrix}
\begin{pmatrix}
1+ I_1 & \left(\frac{1-\e/2}{1+\e/2}\right)^n I_2 \\
\left(\frac{1-\varepsilon/2}{1+\varepsilon/2}\right)^n I_1 & 1+ I_2 \\
\end{pmatrix}
\begin{pmatrix}
A \\
B
\end{pmatrix}
=
\begin{pmatrix}
0 \\
0
\end{pmatrix},
\end{equation}
where we are using the notation 
\begin{align*}
I_1 &:= -\frac{(1-\e/2)}{2n}\int_0^1 \varpi_0'(s(1-\e/2)) s^n \frac{h_L(s)}{\lambda+c_L(s,\e)}ds,\\
I_2 &:=- \frac{(1+\e/2)}{2n}\int_{1}^{+\infty}\varpi_0'(s(1+\e/2)) s^{-n}\frac{ h_R(s)}{\lambda+c_R(s,\e)}ds.
\end{align*}

Since we are looking for a non-trivial solution (see \eqref{nontrivialAB}), we need to impose that the matrix determinant vanishes. That is,

\begin{equation}\label{eq:det}
0=\det\begin{pmatrix}
1+ I_1 & \left(\frac{1-\e/2}{1+\e/2}\right)^n I_2 \\
\left(\frac{1-\varepsilon/2}{1+\varepsilon/2}\right)^n I_1 & 1+ I_2 \\
\end{pmatrix}
=1+ I_1 +I_2 + \left[1 - \left(\frac{1-\varepsilon/2}{1+\varepsilon/2}\right)^{2n}\right] I_1 I_2.
\end{equation}

Here, it is important to emphasize that \eqref{eq:det} is just an equation for the unknown $\lambda$. Since we are imposing from the beginning the ansatz $\lambda = \lambda_0 + \lambda_1 \e + \e^2 \log^2(\e) \lambda_2(\e)$ with $\lambda_0 = \frac{\log(2)}{2},$
$$\lambda_1 \in \left(-\frac{\log(2)}{2}, -\frac{1-\log(2)}{2}\right),$$
and $\lambda_2(\e) = O(1)$ in terms of the parameter $\e$, the above is, at the end of the day, an equation for $\lambda_1$ and $\lambda_2(\e)$. We also remark that \eqref{eq:det} depends on $\lambda$ explicitly and implicitly through $h_L$ and $h_R$.

\subsection{The value of $\l_1$}\label{s:fixingl1}
Up to this point, we have shown that solving \eqref{inteqleft_final}--\eqref{inteqright_final} reduces to solving \eqref{eq:det} for \(\lambda\).
 The next step will be to fix the value of $\lambda_1\in (-\frac{\log(2)}{2},-\frac{1-\log(2)}{2})$. Indeed, we impose
\begin{align}\label{eq:l1equation}
1-\frac{1}{2n}\frac{1}{\log(2)-\frac{1}{2}}\log\left(\frac{-\left(\lambda_1+\frac{1-\log(2)}{2}\right)}{\left(\lambda_1+\frac{\log(2)}{2}\right)}\right)+Z_L[\mathring{h}_L]+Z_R[\mathring{h}_R]=0,
\end{align}
where
\begin{align*}
-2nZ_L[\mathring{h}_L]:= &\varpi'_0(1)\int_0^1\frac{x^n-1}{\left(\log(2)-\frac{1}{2}\right)(x-1)}dx+ \varpi'_0(1)\int_0^1  \frac{\lambda_0+c_L(x,0)-\left(\log(2)-\frac{1}{2}\right)(x-1)}{(\lambda_0+c_L(x,0))\left(\log(2)-\frac{1}{2}\right)(x-1)}x^ndx\\
    -& \frac{\varpi'_0(1)}{\log(2)-\frac{1}{2}}\log\left(\log(2)-\frac{1}{2}\right)+\varpi'_0(1)\int_0^1 \frac{\mathring{h}_L(x)-1}{\lambda_0+c_L(x,0)}x^n dx,
    \end{align*}
and
\begin{align*}
  -2n Z_R[\mathring{h}_R]:=& \varpi'_0(1)\int_{1}^2 \frac{x^{-n}-1}{\left(\log(2)-\frac{1}{2}\right)(x-1)}dx+\varpi'_0(1)\int_{2}^{+\infty} \frac{x^{-n}}{\left(\log(2)-\frac{1}{2}\right)(x-1)}dx \\
+&\varpi'_0(1)\int_1^{+\infty}  \frac{\lambda_0+c_R(x,0)-\left(\log(2)-\frac{1}{2}\right)(x-1)}{(\lambda_0+c_R(x,0))\left(\log(2)-\frac{1}{2}\right)(x-1)}x^{-n}dx+\frac{\varpi'_0(1)}{\log(2)-\frac{1}{2}}\log\left(\log(2)-\frac{1}{2}\right)\\
+&\varpi_0'(1)\int_1^{+\infty} \frac{\mathring{h}_R(x)-1}{\lambda_0+c_R(x,0)}x^{-n}dx.
\end{align*}

Notice that the integrals in \(Z_L[\mathring{h}_L]\) and \(Z[\mathring{h}_R]\) are finite. Indeed, although \(\lambda_0 + c_R(x,0)\) and \(\lambda_0 + c_L(x,0)\) behave like \((x-1)^{-1}\) near \(x=1\), the factors \(\mathring{h}_R(x)-1\) and \(\mathring{h}_L(x)-1\) behave like \((x-1)\log|x-1|\), ensuring the integrals converge.

It is easy to see from \eqref{eq:l1equation} that there exists a unique solution $\lambda_1\in (-\frac{\log(2)}{2},-\frac{1-\log(2)}{2})$.

The motivation behind defining $\l_1$ in this specific manner will become evident in what follows.

\subsection{An equation for $\lambda_2(\e)$}
The goal is to get, from equations \eqref{cond:NLandNR} and the choice \eqref{eq:l1equation}, an equation for $\lambda_2(\e)$ of the type 
\[
\lambda_2(\e) = M[\mathring{h}_L, \mathring{h}_R,\lambda_1,\e] + R[\mathring{h}_L,\mathring{h}_R,\lambda_1,  \e; h_R, h_L, \lambda_2(\e)],
\]
where $M$ is $O(1)$ and $R$ is $o(1)$ in $\e$, and continuous in $\lambda_2(\e)$.

Recall that both \(h_R\) and \(h_L\) depend on \(\lambda_2(\varepsilon)\).

\subsubsection{The main term $1+I_1+I_2$}
We focus our attention on the main part of \eqref{eq:det}. That is,
\[
1 + I_1 +  I_2.
\]
We are going to study the above expression in detail. Thus, we have
\begin{align*}
I_1 &= -(1-\e/2)\frac{1}{2n}\int_0^1  \varpi_0'(x(1-\e/2)) x^n \frac{h_L(x)}{\l+c_L(x,\e)}dx,\\
I_2 &=-(1+\e/2)\frac{1}{2n}\int_1^\infty \varpi_0'(x(1+\e/2))x^{-n} \frac{h_R(x)}{\l + c_R(x,\e)}dx.
\end{align*}

We start by studying $I_1$. We have that
\begin{align*}
 -2n\frac{I_1}{1-\e/2}=K_1 + \varpi_0'(1)\int_0^1 \frac{x^n}{\l +c_L(x,\e)} dx ,   
\end{align*}
where 
\begin{align*}
K_1&:=\int_0^1   \frac{\varpi_0'(x(1-\e/2)) h_L(x)-\varpi_0'(1) }{\l +c_L(x,\e)}x^n dx.
\end{align*}

Next, we analyze the term
\begin{align}\label{mainleft}
 \int_0^1 \frac{x^n }{\l +c_L(x,\e)} dx .  
\end{align}

We shall decompose
\begin{align*}
    \frac{1}{\lambda+c_L(x,\e)}=&\frac{1}{\lambda+c_L(x,\e)}-\frac{1}{\left(\log(2)-\frac{1}{2}\right)(x-1)+\e\left(\lambda_1+\frac{1-\log(2)}{2}\right)+\e^2\log^2(\e)\lambda_2(\e) }\\
    +&\frac{1}{\left(\log(2)-\frac{1}{2}\right)(x-1)+\e\left(\lambda_1+\frac{1-\log(2)}{2}\right)+\e^2\log^2(\e)\lambda_2(\e) },
\end{align*}
where we notice that
\begin{align*}
    \left.\pa_x(\lambda_0+c_L(x,0))\right|_{x=1}=&\log(2)-\frac{1}{2},\\
    \pa_\e c_L(x,0)|_{x=1}=&\frac{1-\log(2)}{2}.
\end{align*}
Therefore
\begin{align*}
    &\lambda+c_L(x,\e)-     \left(\left(\log(2)-\frac{1}{2}\right)(x-1)+\e\left(\lambda_1+\frac{1-\log(2)}{2}\right)+\e^2\log^2(\e)\lambda_2(\e)\right)\\
    =& \lambda_0+c_L(x,0)-\left(\log(2)-\frac{1}{2}\right)(x-1)
    +\e\left(\pa_\e c_L(x,0)-\pa_\e c_L(1,0)\right)
    +\e^2 \pa^2_\e c_L(x,\e^*),
\end{align*}
with
\begin{align*}
    \lambda_0+c_L(x,0)-\left(\log(2)-\frac{1}{2}\right)(x-1)=&O((x-1)^2),\\
    \pa_\e c_L(x,0)-\pa_\e c_L(1,0)=& O(x-1).
\end{align*}

In the integral
\begin{align*}
\int_0^1\left(\frac{1}{\lambda+c_L(x,\e)}-\frac{1}{\left(\log(2)-\frac{1}{2}\right)(x-1)+\e\left(\lambda_1+\frac{1-\log(2)}{2}\right)+\e^2\log^2(\e)\lambda_2(\e) }\right)x^{n}dx,
\end{align*}
we must focus on the terms
\begin{align}
&\label{1} \int_0^1  \frac{\lambda_0+c_L(x,0)-\left(\log(2)-\frac{1}{2}\right)(x-1)}{(\lambda+c_L(x,\e))\left(\left(\log(2)-\frac{1}{2}\right)(x-1)+\e\left(\lambda_1+\frac{1-\log(2)}{2}\right)+\e^2\log^2(\e)\lambda_2(\e)\right)}x^ndx\\   
&\label{2}\int_0^1  \frac{\e (\pa_\e c_L(x,0)-\pa_\e c_L(1,0))}{(\lambda+c_L(x,\e))\left(\left(\log(2)-\frac{1}{2}\right)(x-1)+\e\left(\lambda_1+\frac{1-\log(2)}{2}\right)+\e^2\log^2(\e)\lambda_2(\e)\right)}x^ndx\\
&\label{3}\int_0^1  \frac{\e^2\pa^2_\e c_L(x,\e^*)}{(\lambda+c_L(x,\e))\left(\left(\log(2)-\frac{1}{2}\right)(x-1)+\e\left(\lambda_1+\frac{1-\log(2)}{2}\right)+\e^2\log^2(\e)\lambda_2(\e)\right)}x^ndx
\end{align}
From Lemma \ref{bounddenominator}, \eqref{2} can be bounded by
\begin{align*}
    C\e\int_0^1\frac{1}{|x-1|+c\e}dx\leq C \e\log(\e^{-1}). 
\end{align*}
As usual, we obtain this bound for \(\lambda_2(\varepsilon) = O(1)\) and for \(\varepsilon\) sufficiently small relative to $
\frac{\left|\lambda_1 + \frac{1-\log(2)}{2}\right|}{|\lambda_2(\varepsilon)|}.$

The term \eqref{3} can be bounded by
\begin{align*}
C\e^2\int_0^1 \frac{1}{(|x-1|+c\e)^2}dx\leq C\e.
\end{align*}

Finally, for \eqref{1} we proceed as follows,
\begin{align*}
\text{\eqref{1}}= &\int_0^1  \frac{\lambda_0+c_L(x,0)-\left(\log(2)-\frac{1}{2}\right)(x-1)}{(\lambda_0+c_L(x,0))\left(\log(2)-\frac{1}{2}\right)(x-1)}x^ndx\\
+&\int_0^1\left( \frac{\lambda_0+c_L(x,0)-\left(\log(2)-\frac{1}{2}\right)(x-1)}{(\lambda+c_L(x,\e))\left(\left(\log(2)-\frac{1}{2}\right)(x-1)+\e\left(\lambda_1+\frac{1-\log(2)}{2}\right)+\e^2\log^2(\e)\lambda_2(\e)\right)}\right.\\ &\qquad\left.-\frac{\lambda_0+c_L(x,0)-\left(\log(2)-\frac{1}{2}\right)(x-1)}{(\lambda_0+c_L(x,0))\left(\log(2)-\frac{1}{2}\right)(x-1)}\right)x^ndx\\
=&\int_0^1  \frac{\lambda_0+c_L(x,0)-\left(\log(2)-\frac{1}{2}\right)(x-1)}{(\lambda_0+c_L(x,0))\left(\log(2)-\frac{1}{2}\right)(x-1)}x^ndx\\
+& \text{\textit{Remainder}},
\end{align*}
where \textit{Remainder}$=O(\e\log(\e^{-1}))$.
In order to check this last equality, we split 
\begin{align}\label{splitting}
&\frac{1}{(\lambda+c_L(x,\e))\left(\left(\log(2)-\frac{1}{2}\right)(x-1)+\e\left(\lambda_1+\frac{1-\log(2)}{2}\right)+\e^2\log^2(\e)\lambda_2(\e)\right)}
\\-&\frac{1}{(\lambda_0+c_L(x,0))\left(\log(2)-\frac{1}{2}\right)(x-1)}\nonumber\\
=&\left(\frac{1}{\lambda+c_L(x,\e)}-\frac{1}{\lambda_0+c_L(x,0)}\right)\frac{1}{\left(\log(2)-\frac{1}{2}\right)(x-1)+\e\left(\lambda_1+\frac{1-\log(2)}{2}\right)+\e^2\log^2(\e)\lambda_2(\e)}\nonumber\\
+&\frac{1}{\lambda_0+c_L(x,0)}\left(\frac{1}{\left(\log(2)-\frac{1}{2}\right)(x-1)+\e\left(\lambda_1+\frac{1-\log(2)}{2}\right)+\e^2\log^2(\e)\lambda_2(\e)}
-\frac{1}{\left(\log(2)-\frac{1}{2}\right)(x-1)}\right).\nonumber
\end{align}
In addition
\begin{align*}
    \frac{1}{\lambda+c_L(x,\e)}-\frac{1}{\lambda_0+c_L(x,0)}=-\frac{\e(\lambda_1+(\pa_\e c_L)(x,0))+\e^2\log^2(\e)\lambda_2(\e)+\e^2(\pa_\e^2c_L)(x,\e^*)}{(\lambda+c_L(x,\e))(\lambda_0+c_L(x,0)}.
\end{align*}
Then we can see that
\begin{align}\label{q2}
    \left|\frac{1}{\lambda+c_L(x,\e)}-\frac{1}{\lambda_0+c_L(x,0)}\right|\leq \frac{C\e}{|x-1|^2}.
\end{align}
Also
\begin{multline}\label{q1}
    \left|\frac{1}{\left(\log(2)-\frac{1}{2}\right)(x-1)+\e\left(\lambda_1+\frac{1-\log(2)}{2}\right)+\e^2\log^2(\e)\lambda_2(\e)}
-\frac{1}{\left(\log(2)-\frac{1}{2}\right)(x-1)}\right|\\
\leq \frac{C\e}{|x-1|(1-x+c\e)}.
\end{multline}
Using \eqref{q1}, \eqref{q2} and the splitting \eqref{splitting} in \textit{Remainder}, we obtain
\begin{align*}
|\,\textit{Remainder}\,|\leq C\e \int_0^1\frac{1}{1-x+c\e}dx\leq C\e\log(\e^{-1}).
\end{align*}

Thus we have that
\begin{align}\label{derecha1}
\int_0^1 \frac{x^n}{\lambda+c_L(x,\e)}dx=&\int_0^1\frac{x^n}{\left(\log(2)-\frac{1}{2}\right)(x-1)+\e\left(\lambda_1+\frac{1-\log(2)}{2}\right)+\e^2\log^2(\e)\lambda_2(\e)}dx\\
+&\int_0^1  \frac{\lambda_0+c_L(x,0)-\left(\log(2)-\frac{1}{2}\right)(x-1)}{(\lambda_0+c_L(x,0))\left(\log(2)-\frac{1}{2}\right)(x-1)}x^ndx \nonumber \\
+&O(\e\log(\e^{-1})). \nonumber
\end{align}

We now manipulate \eqref{derecha1} further in order to simplify it. We just notice that
\begin{align*}
&\int_0^1\frac{x^n-1}{\left(\log(2)-\frac{1}{2}\right)(x-1)+\e\left(\lambda_1+\frac{1-\log(2)}{2}\right)+\e^2\log^2(\e)\lambda_2(\e)}dx\\
=& \int_0^1\frac{x^n-1}{\left(\log(2)-\frac{1}{2}\right)(x-1)}dx\\
+&\int_0^1\left(\frac{1}{\left(\log(2)-\frac{1}{2}\right)(x-1)+\e\left(\lambda_1+\frac{1-\log(2)}{2}\right)+\e^2\log^2(\e)\lambda_2(\e)}-\frac{1}{\left(\log(2)-\frac{1}{2}\right)(x-1)}\right)(x^n-1)dx\\
=&\int_0^1\frac{x^n-1}{\left(\log(2)-\frac{1}{2}\right)(x-1)}dx +O(\e\log(\e^{-1})).
\end{align*}
Therefore
\begin{align}\label{derecha2}
\int_0^1 \frac{x^n}{\lambda+c_L(x,\e)}dx=&\int_0^1\frac{1}{\left(\log(2)-\frac{1}{2}\right)(x-1)+\e\left(\lambda_1+\frac{1-\log(2)}{2}\right)+\e^2\log^2(\e)\lambda_2(\e)}dx\\
+&\int_0^1\frac{x^n-1}{\left(\log(2)-\frac{1}{2}\right)(x-1)}dx+ \int_0^1  \frac{\lambda_0+c_L(x,0)-\left(\log(2)-\frac{1}{2}\right)(x-1)}{(\lambda_0+c_L(x,0))\left(\log(2)-\frac{1}{2}\right)(x-1)}x^ndx \nonumber \\
+&O(\e\log(\e^{-1})). \nonumber
\end{align}
We can compute explicitly the first integral in the right hand side of \eqref{derecha2}, namely
\begin{align}\label{explicitright}
  &  \int_0^1\frac{1}{\left(\log(2)-\frac{1}{2}\right)(x-1)+\e\left(\lambda_1+\frac{1-\log(2)}{2}\right)+\e^2\log^2(\e)\lambda_2(\e)}dx\\=&
    \frac{1}{\log(2)-\frac{1}{2}}\log\left(\frac{-\e\left(\lambda_1+\frac{1-\log(2)}{2}\right)-\e^2\log^2(\e)\lambda_2(\e)}{\log(2)-\frac{1}{2}-\e\left(\lambda_1+\frac{1-\log(2)}{2}\right)-\e^2\log^2(\e)\lambda_2(\e)}\right)\nonumber\\
    =&
    \frac{1}{\log(2)-\frac{1}{2}}\log\left(-\e\left(\lambda_1+\frac{1-\log(2)}{2}\right)-\e^2\log^2(\e)\lambda_2(\e)\right)\nonumber\\
    -&\frac{1}{\log(2)-\frac{1}{2}}\log\left(\log(2)-\frac{1}{2}\right)+ O(\e).\nonumber
\end{align}

We then get for $I_1$:
\begin{align}\label{mainleft1}
    -2n \frac{I_1}{1-\e/2}=&K_1+\varpi'_0(1)\int_0^1\frac{x^n-1}{\left(\log(2)-\frac{1}{2}\right)(x-1)}dx\\
    +& \varpi'_0(1)\int_0^1  \frac{\lambda_0+c_L(x,0)-\left(\log(2)-\frac{1}{2}\right)(x-1)}{(\lambda_0+c_L(x,0))\left(\log(2)-\frac{1}{2}\right)(x-1)}x^ndx \nonumber\\
    -& \frac{\varpi'_0(1)}{\log(2)-\frac{1}{2}}\log\left(\log(2)-\frac{1}{2}\right) \nonumber\\
    +& \frac{\varpi_0'(1)}{\log(2)-\frac{1}{2}}\log\left(-\e\left(\lambda_1+\frac{1-\log(2)}{2}\right)-\e^2\log^2(\e)\lambda_2(\e)\right)\nonumber\\
    +& o(\e\log^2(\e)).\nonumber
\end{align}

For $I_2$ we have that
\begin{align*}
    -2n\frac{I_2}{1+\e/2}=K_2+\varpi_0'(1)\int_1^{+\infty} \frac{x^{-n}}{\lambda+c_R(x,\e)}dx,
\end{align*}
where
\begin{align*}
K_2:=\int_1^{+\infty}  \frac{\varpi_0'(x(1+\e/2))h_R(x)-\varpi_0'(1)}{\l + c_R(x,\e)}x^{-n}dx.
\end{align*}
Proceeding in a similar way as for \eqref{mainleft}, we get
\begin{align}\label{izquierda1}
\int_1^{+\infty}\frac{x^{-n}}{\lambda+c_R(x,\e)}dx=& \int_{1}^{+\infty} \frac{x^{-n}}{\left(\log(2)-\frac{1}{2}\right)(x-1)+\e\left(\lambda_1+\frac{\log(2)}{2}\right)+\e^2\log^2(\e)\l_2(\e)}dx\\
+&\int_1^{+\infty}  \frac{\lambda_0+c_R(x,0)-\left(\log(2)-\frac{1}{2}\right)(x-1)}{(\lambda_0+c_R(x,0))\left(\log(2)-\frac{1}{2}\right)(x-1)}x^{-n}dx\nonumber \\
+&O(\e\log(\e^{-1})).\nonumber
\end{align}

Now we manipulate the first integral in the right hand side of \eqref{izquierda1}. We have that
\begin{align*}
    &\int_{1}^{+\infty} \frac{x^{-n}}{\left(\log(2)-\frac{1}{2}\right)(x-1)+\e\left(\lambda_1+\frac{\log(2)}{2}\right)+\e^2\log^2(\e)\l_2(\e)}dx\\
    =&\int_{1}^2 \frac{x^{-n}}{\left(\log(2)-\frac{1}{2}\right)(x-1)+\e\left(\lambda_1+\frac{\log(2)}{2}\right)+\e^2\log^2(\e)\l_2(\e)}dx\\
    +& \int_{2}^{+\infty} \frac{x^{-n}}{\left(\log(2)-\frac{1}{2}\right)(x-1)+\e\left(\lambda_1+\frac{\log(2)}{2}\right)+\e^2\log^2(\e)\l_2(\e)}dx,
\end{align*}
where
\begin{align*}
    &\int_{2}^{+\infty} \frac{x^{-n}}{\left(\log(2)-\frac{1}{2}\right)(x-1)+\e\left(\lambda_1+\frac{\log(2)}{2}\right)+\e^2\log^2(\e)\l_2(\e)}dx\\
    =&\int_{2}^{+\infty} \frac{x^{-n}}{\left(\log(2)-\frac{1}{2}\right)(x-1)}dx+O(\e),
\end{align*}
and
\begin{align*}
    &\int_{1}^2 \frac{x^{-n}}{\left(\log(2)-\frac{1}{2}\right)(x-1)+\e\left(\lambda_1+\frac{\log(2)}{2}\right)+\e^2\log^2(\e)\l_2(\e)}dx\\
    =&\int_{1}^2 \frac{x^{-n}-1}{\left(\log(2)-\frac{1}{2}\right)(x-1)+\e\left(\lambda_1+\frac{\log(2)}{2}\right)+\e^2\log^2(\e)\l_2(\e)}dx\\
    +&\int_{1}^2 \frac{1}{\left(\log(2)-\frac{1}{2}\right)(x-1)+\e\left(\lambda_1+\frac{\log(2)}{2}\right)+\e^2\log^2(\e)\l_2(\e)}dx.
\end{align*}
Proceeding as in the left hand side we find that
\begin{align*}
    &\int_{1}^2 \frac{x^{-n}-1}{\left(\log(2)-\frac{1}{2}\right)(x-1)+\e\left(\lambda_1+\frac{\log(2)}{2}\right)+\e^2\log^2(\e)\l_2(\e)}dx\\
    =&\int_{1}^2 \frac{x^{-n}-1}{\left(\log(2)-\frac{1}{2}\right)(x-1)}dx+O(\e\log(\e^{-1})).
\end{align*}
Therefore
\begin{align}\label{izquierda2}
\int_1^{+\infty} \frac{x^{-n}}{\lambda+c_R(x,\e)}dx=&\int_{1}^2 \frac{1}{\left(\log(2)-\frac{1}{2}\right)(x-1)+\e\left(\lambda_1+\frac{\log(2)}{2}\right)+\e^2\log^2(\e)\l_2(\e)}dx\\
+& \int_{1}^2 \frac{x^{-n}-1}{\left(\log(2)-\frac{1}{2}\right)(x-1)}dx+\int_{2}^{+\infty} \frac{x^{-n}}{\left(\log(2)-\frac{1}{2}\right)(x-1)}dx \nonumber\\
+&\int_1^{+\infty}  \frac{\lambda_0+c_R(x,0)-\left(\log(2)-\frac{1}{2}\right)(x-1)}{(\lambda_0+c_R(x,0))\left(\log(2)-\frac{1}{2}\right)(x-1)}x^{-n}dx\nonumber\\
+&O(\e\log(\e^{-1})).\nonumber
\end{align}
Again we can compute explicitly
\begin{align}\label{explicitleft}
    &\int_{1}^2 \frac{1}{\left(\log(2)-\frac{1}{2}\right)(x-1)+\e\left(\lambda_1+\frac{\log(2)}{2}\right)+\e^2\log^2(\e)\l_2(\e)}dx\\
    =&\frac{1}{\log(2)-\frac{1}{2}}\log\left(\frac{\left(\log(2)-\frac{1}{2}\right)+\e\left(\lambda_1+\frac{\log(2)}{2}\right)+\e^2\log^2(\e)\l_2(\e)}{\e\left(\lambda_1+\frac{\log(2)}{2}\right)+\e^2\log^2(\e)\l_2(\e)}\right)\nonumber\\
    =&\frac{1}{\log(2)-\frac{1}{2}}\log\left(\log(2)-\frac{1}{2}\right)\nonumber\\
    -&\frac{1}{\log(2)-\frac{1}{2}}\log\left(\e\left(\lambda_1+\frac{\log(2)}{2}\right)+\e^2\log^2(\e)\l_2(\e)\right)+ O(\e)\nonumber.
\end{align}

Therefore
\begin{align}\label{mainrigth1}
-2n\frac{I_2}{1+\e/2}=&K_2+\varpi'_0(1)\int_{1}^2 \frac{x^{-n}-1}{\left(\log(2)-\frac{1}{2}\right)(x-1)}dx+\varpi'_0(1)\int_{2}^{+\infty} \frac{x^{-n}}{\left(\log(2)-\frac{1}{2}\right)(x-1)}dx \\
+&\varpi'_0(1)\int_1^{+\infty}  \frac{\lambda_0+c_R(x,0)-\left(\log(2)-\frac{1}{2}\right)(x-1)}{(\lambda_0+c_R(x,0))\left(\log(2)-\frac{1}{2}\right)(x-1)}x^{-n}dx\nonumber\\
+&\frac{\varpi'_0(1)}{\log(2)-\frac{1}{2}}\log\left(\log(2)-\frac{1}{2}\right)
-\frac{\varpi'_0(1)}{\log(2)-\frac{1}{2}}\log\left(\e\left(\lambda_1+\frac{\log(2)}{2}\right)+\e^2\log^2(\e)\l_2(\e)\right)\nonumber\\
+&o(\e\log^2(\e)).\nonumber
\end{align}

Before estimating the remainder terms $K_1$	and $K_2$, we introduce the following upper bound, which will be used in the subsequent analysis:
\begin{align}\label{log2}
\int_0^1 \frac{\log(|x-1|^{-1})}{1-x+\e}dx\leq C\log^2(\e).
\end{align}
\begin{proof}[Proof of \eqref{log2}]
\begin{align*}
\int_0^1\frac{-\log(1-x)}{1-x+\e}dx&=\int_0^\e \frac{-\log(x)}{x+\e}dx +\int_\e^1 \frac{-\log(x)}{x+\e}dx\\
&\leq -\frac{1}{\e}\int_0^\e\log(x)dx-\log(\e)\int_\e^1\frac{1}{x+\e}dx\leq C\log^2(\e).
\end{align*}
\end{proof}

For $K_1$ we have that
\begin{align*}
K_1&=\int_0^1\varpi_0'(x(1-\e/2))\frac{h_L(x)-\mathring{h}_L(x)}{\lambda+c_L(x,\e)}x^{n}dx+\int_0^1\frac{\varpi_0'(x(1-\e/2))\mathring{h}_L(x)-\varpi'_0(1)}{\lambda+c_L(x,\e)}x^{n}dx\\
 &=: K_{11}+K_{12}.
 \end{align*}
From  Lemma \ref{l:Dleft} we have that 
\begin{align*}
    K_{11}=M[\mathring{h}_L,\lambda_1]+R[\mathring{h}_L,\lambda_1; h_L, \lambda_2(\e)]. 
\end{align*}
In addition, we can estimate $K_{12}$ as follows
\begin{align*}
    K_{12}=&\int_0^1\frac{\varpi'_0(x(1-\e/2))-\varpi'_0(1)}{\lambda+c_L(x,\e)}\mathring{h}_L(x)x^{n}dx+\varpi'_0(1)\int_{0}^1\frac{\mathring{h}_L(x)-1}{\lambda+c_L(x,\e)}x^{n}dx\\
    =& O(\e\log(\e^{-1}))+\varpi'_0(1)\int_{0}^1\frac{\mathring{h}_L(x)-1}{\lambda+c_L(x,\e)}x^{n}dx.
\end{align*}
Also,
\begin{align*}
    \int_{0}^1\frac{\mathring{h}_L(x)-1}{\lambda+c_L(x,\e)}x^{n}dx=&\int_0^1(\mathring{h}_L(x)-1)\left(\frac{1}{\lambda+c_L(x,\e)}-\frac{1}{\lambda_0+c_L(x,0)+\e\left(\lambda_1+\pa_\e c_L(x,0)\right)}\right)x^ndx\\
    +&\int_{0}^1\frac{\mathring{h}_L(x)-1}{\lambda_0+c_L(x,0)+\e\left(\lambda_1+\pa_\e c_L(x,0)\right)}x^{n}dx.
\end{align*}
We will use, from Lemma \ref{l:hlring}, that $|\mathring{h}_L(x)-1|=O(|x-1|\log(|x-1|^{-1})$, to obtain
\begin{align*}
&\left|\int_0^1(\mathring{h}_L(x)-1)\left(\frac{1}{\lambda+c_L(x,\e)}-\frac{1}{\lambda_0+c_L(x,0)+\e\left(\lambda_1+\pa_\e c_L(x,0)\right)}\right)x^ndx\right|\\ 
    &\leq C\e^2\log^2(\e)\int_0^1\frac{\log(|x-1|^{-1})}{1-x+c\e}x^n\leq C\e^2\log^4(\e)=o(\e\log^2(\e)),
\end{align*}
where in the last inequality we have used \eqref{log2},
and
\begin{align*}
&\int_{0}^1\frac{\mathring{h}_L(x)-1}{\lambda_0+c_L(x,0)+\e\left(\lambda_1+\pa_\e c_L(x,0)\right)}x^{n}dx \\ =& 
\int_0^1 \frac{\mathring{h}_L(x)-1}{\lambda_0+c_L(x,0)}x^n dx\\
+&\int_{0}^1(\mathring{h}_L(x)-1)\left(\frac{1}{\lambda_0+c_L(x,0)+\e\left(\lambda_1+\pa_\e c_L(x,0)\right)}-\frac{1}{\lambda_0+c_L(x,0)}\right)x^ndx\\
\leq & \int_0^1 \frac{\mathring{h}_L(x)-1}{\lambda_0+c_L(x,0)}x^n dx+C\e\int_0^1 \frac{\log(|x-1|^{-1})}{1-x+c\e}dx=\int_0^1 \frac{\mathring{h}_L(x)-1}{\lambda_0+c_L(x,0)}x^n dx +O(\e\log^2(\e)).
\end{align*}

Therefore
\begin{align*}
&\int_{0}^1\frac{\mathring{h}_L(x)-1}{\lambda_0+c_L(x,0)+\e\left(\lambda_1+\pa_\e c_L(x,0)\right)}x^{n}dx=\int_0^1 \frac{\mathring{h}_L(x)-1}{\lambda_0+c_L(x,0)}x^n dx+M[\mathring{h}_L,\lambda_1],
\end{align*}
where
\begin{align*}
\int_0^1 \frac{\mathring{h}_L(x)-1}{\lambda_0+c_L(x,0)}x^n dx=O(1).
\end{align*}

We then have that
\begin{align*}
K_{12}=    \int_0^1 \frac{\mathring{h}_L(x)-1}{\lambda_0+c_L(x,0)}x^n dx+M[\mathring{h}_L,\lambda_1]+R[\mathring{h}_L,\lambda_1; h_L, \lambda_2(\e)].
\end{align*}
Thus
\begin{align}\label{finalK1}
    K_1=\varpi'_0(1)\int_0^1 \frac{\mathring{h}_L(x)-1}{\lambda_0+c_L(x,0)}x^n dx+M[\mathring{h}_L,\lambda_1]+R[\mathring{h}_L,\lambda_1; h_L, \lambda_2(\e).
\end{align}

Proceeding in a similar way with $K_2$ we obtain that
\begin{align}\label{finalK2}
K_2=\varpi_0'(1)\int_1^{+\infty} \frac{\mathring{h}_R(x)-1}{\lambda_0+c_R(x,0)}x^{-n}dx+ M[\mathring{h}_R,\lambda_1]+R[\mathring{h}_R,\lambda_1; h_R, \lambda_2(\e)].
\end{align}

We summarize the results \eqref{mainleft1}, \eqref{mainrigth1}, \eqref{finalK1} and \eqref{finalK2} in the following lemma.
\begin{lemma}\label{I1andI2}The following equalities hold
\begin{align*}
-2nI_1=&\frac{\varpi_0'(1)}{\log(2)-\frac{1}{2}}\log\left(-\e\left(\lambda_1+\frac{1-\log(2)}{2}\right)-\e^2\log^2(\e)\lambda_2(\e)\right)\\
    -&2n Z_L[\mathring{h}_L]+M[\mathring{h}_L,\lambda_1]+R[\mathring{h}_L,\lambda_1; h_L, \lambda_2(\e)],\end{align*}
    where
    \begin{align*}
    -2nZ_L[\mathring{h}_L]:= &\varpi'_0(1)\int_0^1\frac{x^n-1}{\left(\log(2)-\frac{1}{2}\right)(x-1)}dx+ \varpi'_0(1)\int_0^1  \frac{\lambda_0+c_L(x,0)-\left(\log(2)-\frac{1}{2}\right)(x-1)}{(\lambda_0+c_L(x,0))\left(\log(2)-\frac{1}{2}\right)(x-1)}x^ndx\\
    -& \frac{\varpi'_0(1)}{\log(2)-\frac{1}{2}}\log\left(\log(2)-\frac{1}{2}\right)+\varpi'_0(1)\int_0^1 \frac{\mathring{h}_L(x)-1}{\lambda_0+c_L(x,0)}x^n dx.\\
    \end{align*}
    
    \begin{align*}
    -2nI_2=-&\frac{\varpi'_0(1)}{\log(2)-\frac{1}{2}}\log\left(\e\left(\lambda_1+\frac{\log(2)}{2}\right)+\e^2\log^2(\e)\l_2(\e)\right)\\
-&2n Z_R[\mathring{h}_R]+M[\mathring{h}_R,\lambda_1]+R[\mathring{h}_R,\lambda_1; h_R, \lambda_2(\e)],
\end{align*}
where
\begin{align*}
  -2n Z_R[\mathring{h}_R] =& \varpi'_0(1)\int_{1}^2 \frac{x^{-n}-1}{\left(\log(2)-\frac{1}{2}\right)(x-1)}dx+\varpi'_0(1)\int_{2}^{+\infty} \frac{x^{-n}}{\left(\log(2)-\frac{1}{2}\right)(x-1)}dx \\
+&\varpi'_0(1)\int_1^{+\infty}  \frac{\lambda_0+c_R(x,0)-\left(\log(2)-\frac{1}{2}\right)(x-1)}{(\lambda_0+c_R(x,0))\left(\log(2)-\frac{1}{2}\right)(x-1)}x^{-n}dx+\frac{\varpi'_0(1)}{\log(2)-\frac{1}{2}}\log\left(\log(2)-\frac{1}{2}\right)\\
+&\varpi_0'(1)\int_1^{+\infty} \frac{\mathring{h}_R(x)-1}{\lambda_0+c_R(x,0)}x^{-n}dx.\\
\end{align*}
\end{lemma}
\begin{proof}We just notice that the leading order term for $K_1$ in \eqref{finalK1} is $O(1)$ and, then, the leading order term for $-2nI_1(1-\e/2)^{-1}$ in \eqref{mainleft1} is $O(\log(\e^{-1}))$. Therefore
\begin{align*}
    -2nI_1=-2n\frac{I_1}{1-\e/2}+2n \e/2\frac{I_1}{1-\e/2}=-2n\frac{I_1}{1-\e/2}+o(\e\log^2(\e)),
\end{align*}
and the same estimates \eqref{mainleft1} and \eqref{finalK1} for $-2nI_1(1-\e/2)^{-1}$ are valid for $-2nI_1$. The same argument holds for $-2nI_2$.
\end{proof}

Lemma \ref{I1andI2} implies that
\begin{align*}
-2n(I_1+I_2)= &\frac{\varpi'_0(1)}{\log(2)-\frac{1}{2}}\log\left(\frac{-\left(\lambda_1+\frac{1-\log(2)}{2}\right)-\e\log^2(\e)\lambda_2(\e)}{\left(\lambda_1+\frac{\log(2)}{2}\right)+\e\log^2(\e)\l_2(\e)}\right)\\-&2nZ_L[\mathring{h}_L]-2nZ_R[\mathring{h}_R]+M[\mathring{h}_{L,R},\lambda_1]
+R[\mathring{h}_{L,R},\lambda_1;h_{R,L},\lambda_2(\e)].
\end{align*}

In addition, direct computations give us that
\begin{align*}
&\log\left(\frac{-\left(\lambda_1+\frac{1-\log(2)}{2}\right)-\e\log^2(\e)\lambda_2(\e)}{\left(\lambda_1+\frac{\log(2)}{2}\right)+\e\log^2(\e)\l_2(\e)}\right)\\=&
\log\left(\frac{-\left(\lambda_1+\frac{1-\log(2)}{2}\right)}{\left(\lambda_1+\frac{\log(2)}{2}\right)}\right)+
\frac{\log(2)-\tfrac{1}{2}}{\left(\l_1 +\frac{1-\log(2)}{2}\right)\left(\l_1 + \frac{\log(2)}{2}\right)}\e\log^2(\e)\l_2(\e)+o(\e\log^2(\e 
)).
\end{align*}
Therefore, using the definition of $\lambda_1$ in \eqref{eq:l1equation}, we get
\begin{align*}
   1+I_1+I_2=&  -\frac{\varpi'_0(1)}{2n}\frac{1}{\left(\l_1 +\frac{1-\log(2)}{2}\right)\left(\l_1 + \frac{\log(2)}{2}\right)}\e\log^2(\e)\l_2(\e)\\
&+M[\mathring{h}_{L,R},\lambda_1]
+R[\mathring{h}_{L,R},\lambda_1;h_{R,L},\lambda_2(\e)].
\end{align*}

\subsubsection{The term $I_1 I_2$}

From Lemma \ref{I1andI2} we have that 
\begin{align*}
I_1=I^0_1+O(1),
\end{align*}
where
$$I_1^0=\frac{1}{2n}\frac{\varpi_0'(1)}{\log(2)-\frac{1}{2}}\log(\e^{-1})=O(\log(\e^{-1})),$$
and it does not depend on $\lambda_2(\e)$ in any way. The same holds for $I_2$. Therefore
\begin{align*}
I_1\cdot I_2= I_1^0\cdot I^0_2+ O(\log(\e^{-1})).
\end{align*}
Since $1-(1-\e/2)^{2n}/(1+\e/2)^{2n}=O(\e)$ we find that
\begin{align*}
\left(1-\frac{(1-\e/2)^{2n}}{(1+\e/2)^{2n}}\right)I_1I_2=M[\mathring{h}_{L,R},\lambda_1]+R[\mathring{h}_{L,R},\lambda_1; h_{L,R},\lambda_2(\e)].
\end{align*}

We now summarize the content of this section in the following result.

\begin{lemma}\label{eqlambda2}The following equality holds
\begin{align*}
&1+I_1+I_2+\left(1-\frac{(1-\e/2)^{2n}}{(1+\e/2)^{2n}}\right)I_1I_2=\\&-\frac{\varpi'_0(1)}{2n}\frac{\e\log^2(\e)\l_2(\e)}{\left(\l_1 +\frac{1-\log(2)}{2}\right)\left(\l_1 + \frac{\log(2)}{2}\right)}+M[\mathring{h}_{L,R},\lambda_1]
+R[\mathring{h}_{L,R},\lambda_1;h_{L,R},\lambda_2(\e)].
\end{align*}
\end{lemma}

\subsection{Completion of the proof using a Fixed Point argument}\label{fixpoint}
By Lemma \ref{eqlambda2}, we have
\[
\lambda_2(\varepsilon) = M[\mathring{h}_L, \mathring{h}_R,\lambda_1] + R[\mathring{h}_L, \mathring{h}_R,\lambda_1; h_{L}[\lambda_2(\varepsilon)], h_{R}[\lambda_2(\varepsilon)], \lambda_2(\varepsilon)],
\]
where \(|M[\mathring{h}_L,\mathring{h}_R,\lambda_1]|\leq \mathsf{M}-1\), for some constant $\mathsf{M}>1$ independent of  $\e$. Consider the unit ball 
\[
B_1\equiv \{ \mu : |\mu -M[\mathring{h}_{L}, \mathring{h}_{R},\lambda_1]| < 1 \}.
\]
This provides an a priori bound for \(\mu\). Indeed, $|\mu|\leq \mathsf{M}$ (here is where we fix the $\mathsf{M}$ in \eqref{bound:l2}). 

Now, for any \(\mu \in B_1\), we have \(R[\mathring{h}_L,\mathring{h}_R,\lambda_1; h_L[\mu], h_R[\mu], \mu] = o_\mathsf{M}(1)\) as \(\varepsilon \to 0\). In addition, the map
\[
\mathcal{B}(\mu) := M[\mathring{h}_L,\mathring{h}_R,\lambda_1] + R[\mathring{h}_L,\mathring{h}_R,\lambda_1; h_L[\mu], h_R[\mu], \mu]
\]
is continuous in \(\mu\) for \(\mu \in B_1\) and sufficiently small \(\varepsilon\). Moreover, for \(\varepsilon\) small enough, \(\mathcal{B}\) maps \(B_1\) into \(B_{1/2}\), since
\[
|\mathcal{B}(\mu) - M[\mathring{h}_L,\mathring{h}_R,\lambda_1]| \leq o_\mathsf{M}(1).
\]

As a consequence, by Brouwer's fixed point theorem, we obtain a solution \(\lambda_2(\varepsilon)\) to equation \eqref{eq:det}.

\section{Proof of main theorem}
For $\e_0$ small enough, set $0<\e<\e_0$ and let $\lambda_2(\e)$ be the solution found in Section \ref{fixpoint} and $\lambda=\lambda_0+\e\lambda_1+\e^2\log^2(\e)\lambda_2(\e)$, with $\lambda_1\in (-\log(2)/2, -(1-\log(2))/2)$ given by \eqref{eq:l1equation}. Take  $(h_L,h_R)$ the solutions of \eqref{eqfrob:leftgeneral}-\eqref{eqfrob:rightgeneral}, given by Lemma \ref{l:prop},  for that $\lambda$. We know that there exist $(A,B)\neq (0,0)$ such that $H_L=A h_L$ and $H_R=B h_R$ solve the integral equations \eqref{inteqleft_final} and \eqref{inteqright_final}. This is the content of Section \ref{argument}. Therefore
\begin{align*}
f_L^\ast(x)=&\frac{H_L(x)}{2nx(1-\e/2)(\l + c_L(x,\e))}, \qquad x\in [0,1],\\
f_R^\ast(x)=&\frac{H_R(x)}{2nx(1+\e/2)(\l + c_R(x,\e))}, \qquad x\in[1,+\infty),
\end{align*}
solve \eqref{inteq:fLrescaled}-\eqref{inteq:fRrescaled}. Notice that, both $f_R^\ast$ and $f_L^\ast$ are smooth. Then, the rescaled version $f_L$ and $f_R$ satisfy \eqref{inteq:fL}-\eqref{inteq:fR} as we want. 

\section*{Appendix}\label{appendix:levelsets}
%\addcontentsline{toc}{section}{Appendix}

This appendix contains the computations leading to the equations that appear in Section \ref{s:maineq}. While the steps involved are elementary, they have been omitted from the body of the article to improve readability. We include them here for completeness and for the benefit of readers interested in verifying the derivation.\\

Firstly, we will compute the velocities, 
\begin{align*}
&u_r[\pa_r\varpi(r)h_n(r)\cos(n\theta)](r,\theta)\\&=\frac{\mathbf{e}_r}{2\pi}\cdot\int_{0}^\infty\int_{-\pi}^\pi\frac{(r\cos(\theta)-s\cos(\beta),\, r\sin(\theta)-s\sin(\beta))^\perp}{r^2+s^2-2rs\cos(\theta-\beta)} \pa_s\varpi(s)h_n(s)\cos(n\beta)s\, d\beta ds,\\
&u_\theta[\varpi](r,\theta)\\&= \frac{\mathbf{e}_\theta}{2\pi}\cdot \int_0^\infty\int_{-\pi}^\pi \frac{(r\cos(\theta)-s\cos(\beta),\, r\sin(\theta)-s\sin(\beta))^\perp}{r^2+s^2-2rs\cos(\theta-\beta)} \varpi(s) s\, d\beta ds.
\end{align*}

We find that
\begin{align*}
&u_r[\pa_r\varpi(r)h_n(r)\cos(n\theta)]\\&=-\frac{1}{2\pi}\int_{0}^\infty\int_{-\pi}^\pi \pa_s\varpi(s)\frac{\sin(\theta-\beta)}{r^2+s^2-2rs\cos(\theta-\beta)}h_n(s)\cos(n\beta)s^2 \, d\beta ds
    \\&=-\frac{1}{2\pi}\int_{0}^\infty \pa_s\varpi(s)h_n(s)\int_{-\pi}^\pi\frac{\sin(\beta)}{r^2+s^2-2rs\cos(\beta)}h_n(s)\cos(n(\theta-\beta))d\beta s^2ds\\
    &=-\sin(n\theta)\int_0^\infty \pa_s\varpi(s)h_n(s)K_n(s/r)\frac{s^2}{r^2}ds=-\sin(n\theta)\int_0^\infty \pa_s\varpi(s)h_n(s)K_n(r/s)ds,
\end{align*}
where (see Appendix of \cite{CFMS1})
\begin{equation}\label{residues}
    K_n(r)=\frac{1}{2\pi}\int_\mathbb{T}\frac{\sin(\beta)\sin(n\beta)}{1+r^2-2r\cos(\beta)}d\beta =\begin{cases}
        \frac{1}{2 r}r^{-n}, \qquad r>1,\\
        \frac{1}{2r}r^{+n}, \qquad r\leq 1.
    \end{cases}
\end{equation}

With the angular velocity, we proceed as follows. The full velocity for $\varphi(\mathbf{x})=\varpi(r)$ is given by
\begin{align*}
\mathbf{u}[\varphi](\mathbf{x})=\frac{1}{4\pi}\int_{\mathbb{R}^2}\log(|\mathbf{x}-\mathbf{y}|^2)\nabla^\perp\varphi(\mathbf{y})d\mathbf{y}.
\end{align*}
In polar coordinates
\begin{align*}
\mathbf{u}[\varpi](r,\theta)=\frac{1}{4\pi}\int_0^\infty\int_{-\pi}^\pi \log(r^2+s^2-2rs\cos(\theta-\beta))\mathbf{e}_\theta(\beta)\partial_s\varpi(s)d\beta sds.
\end{align*}
Thus
\begin{align*}
\frac{u_\theta[\varpi](r,\theta)}{r} =\frac{\mathbf{u}[\varpi](r,\theta) \cdot \mathbf{e}_\theta(\theta)}{r}=\frac{1}{4\pi r}\int_0^\infty \int_{-\pi}^\pi \log(r^2+s^2-2rs\cos(\theta-\beta))\cos(\theta-\beta)\pa_s\varpi(s)d\beta sds.
\end{align*}    
Integrating by parts yields
\begin{align*}
&-\frac{1}{4\pi} \int_\mathbb{T} \cos(\beta) \log(r^2 +s^2 -2r s \cos(\beta)) d\beta =\frac{1}{4\pi }  \int_{\mathbb{T}} \sin(\beta) \pa_\beta  \log(r^2 +s^2 -2r s \cos(\beta)) d\beta \\
&=\frac{r}{2\pi}\int_\mathbb{T}\frac{ s\sin^2(\beta)}{r^2+s^2 -2r s \cos(\beta)}d \beta,
\end{align*}
thus
\begin{align*}
\frac{u_\theta[\varpi](r,\theta)}{r}=&-\int_0^\infty \pa_s\varpi(s)s^2\frac{1}{2\pi}\int_{-\pi}^\pi \frac{\sin^2(\beta)}{r^2+s^2-2rs\cos(\beta)}d\beta ds\\
\\=& -\int_0^\infty \pa_s\varpi(s)\frac{s^2}{r^2}K_1(s/r)ds=-\int_0^\infty \pa_s\varpi(s) K_1(r/s)ds.
\end{align*}

Therefore, inserting \eqref{Wnansatz} in   \eqref{Weq}, we find that
\begin{align*}
    \lambda n\sin(n\theta) h_n(r)+n\sin(n\theta)h_n(r)\int_0^\infty \pa_s\varpi(s)K_1(r/s)ds-\sin(n\theta)\int_0^\infty \pa_s\varpi(s)h_n(s)K_n(r/s)ds=0
\end{align*}
which yields \eqref{Wn}. 
\vspace{0.5 cm}

\textbf{Acknowledgments:} 
AC acknowledges financial support from the 2023 Leonardo Grant for Researchers and Cultural Creators, BBVA Foundation. The BBVA Foundation accepts no responsibility for the opinions, statements, and contents included in the project and/or the results thereof, which are entirely the responsibility of the authors. AC was partially supported by the Severo Ochoa Programme for Centers of Excellence Grant CEX2019-000904-S and CEX-2023-001347-S funded by MCIN/AEI/10.13039/501100011033. AC and DL were partially supported by the MICINN through the grant PID2020-114703GB-I00  and PID2024-158418NB-I00.
DL is supported by RYC2021-030970-I funded by MCIN/AEI/10.13039/501100011033 and the NextGenerationEU. 
DL was partially supported by the MICINN through the grant PID2022-141187NB-I00.\\

The authors thank Joan Mateu for helpful discussions. They also thank the anonymous referees for their careful reading and valuable comments, which helped improve the presentation of the paper.

%\bibliography{references}
\bibliographystyle{abbrv}

\end{document}